\documentclass[a4paper,12pt]{article}

\usepackage[utf8]{inputenc}
\usepackage{verbatim}
\usepackage{amsfonts}
\usepackage{lmodern}
\usepackage{a4wide}
\usepackage{amsthm}
\usepackage{color}
\usepackage{xcolor}
\usepackage{graphicx}
\usepackage{amsmath,amssymb}
\usepackage{cancel}
\usepackage{mathtools}
\usepackage{hyperref}
\usepackage{doi}
\usepackage[numbers]{natbib}
\usepackage{esint}
\usepackage{wrapfig}
\usepackage[normalem]{ulem}
\usepackage{marginnote}
\usepackage{setspace}
\usepackage{soul}

\newcommand{\Hd}{\mathcal H}
\newcommand{\Lb}{\mathcal L}
\newcommand{\T}{\boldsymbol T}
\newcommand{\N}{\mathcal N}
\newcommand{\A}{\mathcal A}
\newcommand{\M}{\mathcal M}
\newcommand{\F}{\mathcal F}
\newcommand{\injN}{\mathrm{inj}_{\mathcal N}}
\newcommand{\radN}{\mathrm{rad}_{\mathcal N}}

\newcommand{\bv}{\boldsymbol v}
\newcommand{\bu}{\boldsymbol u}

\newcommand{\ubu}{{\underline{\boldsymbol u}}}
\newcommand{\bpi}{\boldsymbol \pi}

\newcommand{\bgamma}{\boldsymbol \gamma}
\newcommand{\tbgamma}{\widetilde{\boldsymbol \gamma}}
\newcommand{\bdelta}{\boldsymbol \delta}
\newcommand{\bGamma}{\boldsymbol \Gamma}
\newcommand{\bp}{\boldsymbol p}
\newcommand{\bq}{\boldsymbol q}
\newcommand{\be}{\boldsymbol e}
\newcommand{\ba}{{\boldsymbol a}}

\newcommand{\bff}{\boldsymbol f}
\newcommand{\bw}{\boldsymbol w}
\newcommand{\bnu}{\boldsymbol \nu}

\newcommand{\tbv}{\widetilde{\boldsymbol v}}

\newcommand{\bz}{\boldsymbol z}

\newcommand{\tv}{\widetilde v}

\newcommand{\dd}{\,\mathrm d}

\newcommand{\bmu}{\boldsymbol \mu}

\newcommand{\supp}{\mathrm{supp}\,}
\newcommand{\Hess}{\mathrm{Hess}\,}

\newcommand{\eps}{\varepsilon}

\newcommand{\R}{\mathbb R}
\renewcommand{\S}{\mathbb S}
\newcommand{\dist}{\mathrm{dist}}

\newcommand{\Rn}{\mathbb R^n}
\newcommand{\RN}{\mathbb R^N}

\newcommand{\bbN}{\mathbb N}

\DeclareMathOperator*{\esssup}{ess\,sup}
\DeclareMathOperator*{\hav}{hav}

\newcommand{\res}{\,\raisebox{-.127ex}{\reflectbox{\rotatebox[origin=br]{-90}{$\lnot$}}}\,}

\numberwithin{equation}{section}
\newtheorem{prop}{Proposition}[section]
\newtheorem{thm}[prop]{Theorem}
\newtheorem{lemma}[prop]{Lemma}
\newtheorem{defn}[prop]{Definition}
\newtheorem{cor}[prop]{Corollary}

\theoremstyle{remark}

\makeatletter
\newcommand{\subalign}[1]{%
	\vcenter{%
		\Let@ \restore@math@cr \default@tag
		\baselineskip\fontdimen10 \scriptfont\tw@
		\advance\baselineskip\fontdimen12 \scriptfont\tw@
		\lineskip\thr@@\fontdimen8 \scriptfont\thr@@
		\lineskiplimit\lineskip
		\ialign{\hfil$\m@th\scriptstyle##$&$\m@th\scriptstyle{}##$\hfil\crcr
			#1\crcr
		}%
	}%
}
\makeatother

\newcommand{\ignore}[1]{{}}

\newcommand{\lgcomm}[1]{{\color{blue}{\bf [#1]}}}
\newcommand{\sm}[1]{{\color{red!90!black}{#1}}}

\author{Lorenzo Giacomelli*, Micha{\l} {\L}asica$^\dagger$, Salvador
Moll$^\ddagger$ \\
\scriptsize *SBAI Department, Sapienza University of Rome, Via Antonio Scarpa, 16, 00161 Roma, Italy \\
\scriptsize $^\dagger$Institute of Mathematics of the Polish Academy of Sciences,  ul.\;Śniadeckich 8, 00-656 Warszawa, Poland \\
\scriptsize $^\ddagger$Department of Mathematical Analysis, University of
Valencia, C/Dr.\;Moliner, 50, Burjassot, Spain
}
\title{Total variation flow of curves in Riemannian manifolds}
\date{\today}

\begin{document}
 \maketitle

\begin{abstract}
We consider the functional of total variation of maps from an interval into a Riemannian submanifold of $\RN$. We define a notion of strong solution to the system of equations corresponding to the $L^2$-gradient flow of this functional. We prove global existence of strong solutions for initial data of bounded variation. We show that the solutions satisfy a variational equality, and deduce uniqueness in the case of non-positive sectional curvature. We prove convergence of strong solutions to a constant map in finite time.

\end{abstract}
\medskip

{\small

}

\begin{spacing}{0.5}
\tableofcontents
\end{spacing}
\setlength{\jot}{.4\baselineskip}
\renewcommand{\baselinestretch}{1}\normalsize
%


\section{Introduction}\label{sec:intro}

Let $(\N, g)$ be a complete and connected $n$-dimensional Riemannian manifold without boundary. We assume without loss of generality \cite{nash,mueller} that $\N$ is isometrically embedded in $\RN$ and
that the embedding is closed. We will also assume that {the sectional curvature of $\N$ is bounded from above}.

Let $I$ denote a bounded open interval. We consider the {constrained} vectorial total variation functional $TV_I^\N\colon L^2(I, \N) \to [0, +\infty]$, which can be defined as the lower semicontinuous
envelope of the functional ${\cal TV}_I^\N\colon C^1(\overline{I}, \N)\to [0,+\infty[$ given by
\[{\cal TV}_I^\N(\bw) = \int_I |\bw_x| \dd \Lb^1.\]
Here, $\bw_x$ denotes the derivative of a differentiable function $\bw$. The effective domain of $TV_I^\N$ turns out to coincide with $BV(I, \N) \subset BV(I, \RN)$.
In fact, we have the following representation formula (see Section \ref{sec:preli})
\begin{equation}\label{representation_form} TV_I^\N(\bw) =|\bw_x|_{\N}(I), \end{equation}
where $|\bw_x|_{\N}$ is the measure on $I$ defined by
\begin{equation}
\label{def:measure} |\bw_x|_{\N}(A):=|\bw_x^d|(A)+\sum_{x\in A\cap J_{\bw}}{\dist}_\N(\bw^-(x),\bw^+(x)).
\end{equation}
(here $\bw^\pm(x)$ denote the one-sided limits of $\bw$ at $x$ and the superscript $d$ stands for the diffuse part of the measure, see Section \ref{sec:preli}).

In the present paper, we are concerned with existence of the steepest descent flow of $TV_I^\N$ with respect to the $L^2(I, \N)$ distance---the total variation flow of parametrized curves in a Riemannian manifold $\N$. Though \emph{curve} commonly refers to continuous objects, here by a (parametrized) curve we mean a function $\bw\in BV(I, \N)$, which may therefore have jumps. Following such a flow is a natural way of decreasing the total variation of a function while keeping its range inside $\N$ and controlling its distance from the initial datum. The scalar case, $\N=\R$, which is
known as the total variation flow (see the monograph \cite{acmbook}) received a lot of attention since the seminal work of Rudin, Osher and Fatemi (ROF model) on total variation denoising \cite{rof}. Since then, different choices of manifolds have been proposed for different applications in image processing. The case $\N\subset \S^{N-1}$ has been applied to denoising of
direction data such as color component of RGB images \cite{TSC00, TSC01}. Further, $\N=\R^2\times S^1$, $\N=SO(3)\times \R^3$, $\N=Sym_+(3)$ (the space of positive definite matrices) have been considered as models for Luminance-Chromacity-Hue color space data \cite{weinmanndemaretstorath}, orientation data such as camera trajectories \cite{lellmann}, and diffusion tensor imaging data \cite{weinmanndemaretstorath},
respectively.

We note that ${\cal TV}_I^\N(\bw)$ is equal to the length $L(\bw)$ of a parametrized $C^1$ curve $\bw$. Therefore, $TV_I^\N$ can be seen as a generalization of curve length to parametrized
curves of bounded variation. However, the total variation flow of curves is different than the usual curve shortening flow, which (at least formally) can also be seen as a gradient flow of curve length.
In particular, the former depends on parametrization.

%
In the case that $\N$ has non-positive sectional curvature (NPC), the functional $TV_I^{\N}$ is (geodesically) convex (\cite[Theorem 4.3]{CR-M-PJ}). Therefore, one could use Mayer's generalization  \cite{mayernpc} of the semigroup generation results due to K\={o}mura and Brezis to the case of functions valued in an NPC space to show that the $L^2(I,\N)$ gradient flow of $TV_I^{\N}$ exists for any initial data in the energy space.
%

%
In general, however, the functional $TV_I^\N$ fails to be geodesically $\lambda$-convex in the sense of \cite{ambrosiogiglisavare} for any $\lambda \in \R$, as we demonstrate in Section \ref{sec:counter}. Hence, neither the variational existence result of Mayer \cite{mayernpc}, nor its far-reaching generalization by Ambrosio, Gigli and Savar\' e \cite{ambrosiogiglisavare} can be directly applied. One
may wonder if a similar approach to the one in \cite{ambrosiogiglisavare} for the space of probability measures might work. Some recent progress on gradient flows in the space of
Cartesian currents can be found in Kampschulte's doctoral dissertation \cite{Kampschulte} where he develops relevant theory and uses it to obtain existence for the harmonic map flow in the 2-dimensional case.

\medskip

Our aim in the present paper is two-fold. First of all, we obtain global existence of trajectories of the constrained total variation flow emanating from data of bounded variation in the case of a generic Riemannian manifold constraint. Secondly, our approach permits to characterize the trajectory in PDE terms, thus clarifying Mayer's abstract solution in the case of non-positive sectional curvature. In particular, we are able to identify the vector field $\bz=\bu_x/|\bu_x|$ on the solution's jump set, relating it to the unit tangent vectors to the geodesic connecting the jump values (see \eqref{zeqnj} below).

Let us comment on previous results about $1$-harmonic flows: In \cite{gigakashimayamazaki}, the authors obtained existence of Lipschitz local solutions to the system of PDEs corresponding to the flow,
for small initial data, in the case of the domain being a flat torus. In \cite{gigakurodayamazaki}, they obtained existence and uniqueness of solutions to a discretized Dirichlet problem in the case of
$\N=\S^N$. In \cite{giacomellimoll}, existence of solutions and blow-up in finite time were obtained for the Dirichlet problem under rotational symmetry restrictions (here $\N=\S^2$). In
\cite{gmm1,gmmn} existence of solutions in the case of $\N$ being a circle or a hyperoctant of the $N-$dimensional sphere were shown. Later, the case of $\N$ being a smooth closed curve in $\R^2$
was studied in \cite{dicastrogiacomelli}. There, the authors also proposed a suitable concept of solution for the general case.
 Finally, in \cite{giacomellilasicamoll}, we obtained local existence and uniqueness in the case of Lipschitz initial data. We also mention the numerical study \cite{GigaSakakibara}. 

 If $\bu \colon ]0,T[\times \overline{I} \to \N$ were a smooth trajectory of the flow such that $\bu_x$ does not vanish anywhere on $]0,T[\times \overline{I}$ 
 , the PDE system would read as follows:
\begin{equation} \label{smootheqn}
\bu_t = \pi_{\bu} \left(\frac{\bu_x}{|\bu_x|}\right)_x \quad \text{ on } ]0,T[ \times I,
\end{equation}
\begin{equation} \label{smoothbc}
\frac{\bu_x}{|\bu_x|} = 0 \quad\text{ on } ]0,T[\times \partial I.
\end{equation}
In \eqref{smootheqn} and further on, $\pi_{\bu} \colon \RN \to T_{\bu} \N$ stands for the orthogonal projection onto $T_{\bu} \N$. Note that the meaning of \eqref{smoothbc} is not immediately clear,
as it cannot be satisfied if $\frac{\bu_x}{|\bu_x|}$ is understood naively as a pointwise quotient. 
In general, it needs to be
interpreted as a selection of a multivalued function, see Lemma \ref{localest} below for the proper definition of Lipschitz solutions. In the case of BV solutions, the definition is even more involved
since we need to specify the value of this object $|\bu_x(t)|$-a.\,e.\ in $I$.

We let $\injN$ denote the injectivity radius of the manifold $\N$. Recall that for any two points $\bp_1, \bp_2 \in \N$, the condition $\dist_{\N}(\bp_1, \bp_2) < \injN$ implies that there is exactly
one length-minimizing geodesic segment in $\N$ joining $\bp_1$ and $\bp_2$. We will denote by $\bgamma_{\bp_1}^{\bp_2}$ its parametrization by the distance from its midpoint. Given
$\{\bp_1,\ldots,\bp_n\}$ such that  $\dist_{\N}(\bp_{i}, \bp_{i+1}) < \injN$, for $i=1,\ldots,n-1$ we denote by $[\bp_1,\ldots,\bp_n]$ the arc-length parametrized piecewise-smooth curve $\gamma$ arising as concatenation of reparametrized minimizing segments $\bgamma_{\bp_i}^{\bp_{i+1}}$. 
Given $\bw \in BV(I, \N)$ such that $\dist_{\N}(\bw^-,\bw^+) < \injN$ on $J_{\bw}$ and $x \in J_{\bw}$, we denote by $\boldsymbol T_{\bw}^\pm(x)$ the unit tangent vector to
$\bgamma_{\bw^-(x)}^{\bw^+(x)}$ at $\bw^\pm(x)$ directed consistently with the parametrization of $\bgamma_{\bw^-(x)}^{\bw^+(x)}$ (see Figure \ref{fig:tangent}), i.\,e.\ at any point in $J_{\bw}$ we
have
\[\boldsymbol T_{\bw}^- = \left(\bgamma_{\bw^-}^{\bw^+}\right)'\left(- \tfrac{1}{2}\dist_\N(\bw^-, \bw^+)\right) = \frac{\log_{\bw^-} \bw^+}{\dist_\N(\bw^-, \bw^+)},\]
\[\boldsymbol T_{\bw}^+ = \left(\bgamma_{\bw^-}^{\bw^+}\right)'\left( \tfrac{1}{2}\dist_\N(\bw^-, \bw^+)\right) = - \frac{\log_{\bw^+} \bw^-}{\dist_\N(\bw^-, \bw^+)}{,}\]
\begin{figure}[htb]
	\begin{center}
		\includegraphics[width=0.4\textwidth]{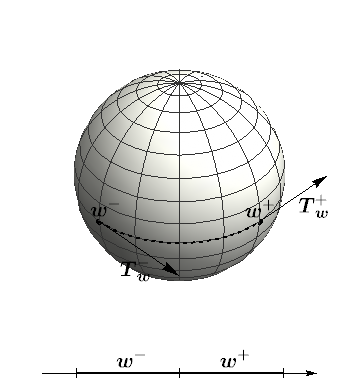}
		\caption{Vectors $\boldsymbol T_{\bw}^\pm$.}\label{fig:tangent}
	\end{center}
\end{figure}
{where here and in the rest of the paper $\log_{\bp}{\bq}$ (${\bp}, {\bq}\in\mathcal N$) denotes the inverse of the exponential map at $T_{\bp}\mathcal N$. \medskip}

With this understanding, as well as with standard ones concerning $BV$ functions (see Section \ref{sec:preli}), we are ready to introduce the notion of strong solutions.

\begin{defn} \label{defsol}
Let $T \in ]0, + \infty]$. Let
\[{\bu \in H^1_{loc}([0,T[; L^2(I, {\N})) \cap L^\infty_{loc}([0,T[; BV(I,\N))}\]
be such that $\dist_{\N}(\bu^-(t,x), \bu^+(t,x)) < \injN$ for all $x \in I$ and a.\,e.\ $t\in ]0,T[$.  We say that $\bu$ is a strong solution to (\ref{smootheqn},\ref{smoothbc}) in $[0,T[$ if there
exists
\[{\bz \in L^\infty_{w^*}(0,T;L^\infty(I, \R^N)) \cap L^2_{w^*,loc}([0,T[;BV(I,\R^N))}\]
such that for a.\,e.\ $t\in ]0,T[$ there holds

\begin{align} \label{maineqn}
	&\bu_t(t) = \pi_{\bu(t)} \bz_x^{a}(t) \quad \Lb^1\text{-a.\,e.\ in } I, \\
	\label{cantor-m}
	&\pi_{\bu(t)}\bz_x^{c}(t)=0 \quad\text{as measures}, \\
	\label{ztang}
	&\bz^{\pm}(t) \in T_{\bu^\pm(t)} \N \quad \text{in } I, \\
	\label{zineq}
	&|\bz(t)| \leq 1 \quad \text{in } I, \\
	\label{zeqn}
	&\bz(t) = \frac{\bu_x(t)}{|\bu_x(t)|} \quad |\bu_x^d(t)|\text{-a.\,e.\ in }I, \\
	\label{zjumpset}
	& J_{\bz(t)} \subseteq J_{\bu(t)}, \\
	\label{zeqnj}
	&\bz^\pm(t)=	\boldsymbol T_{\bu(t)}^\pm \quad \text{on } J_{\bu(t)}, \\
	\label{zbc}
	&\bz(t) = 0 \quad \text{on } \partial I.
\end{align}

	\end{defn}

Though Definition \ref{defsol} may seem complex at first sight, it contains just the essential ingredients for a characterization of $\bz$ and $\bu_t$ and is quite natural in this respect: the PDE \eqref{maineqn} relates $\bu_t$ to $\bz_x$, \eqref{zeqn}-\eqref{zbc} identify $\bz$ on $\bu_x\ne \boldsymbol 0$ (in particular, on $J_{\bu}$) and on $\partial I$, and \eqref{ztang}-\eqref{zineq} provide information on $\bz$ where $\bu_x=\boldsymbol 0$. Finally, \eqref{cantor-m} characterizes $\bz_x^c$ and will be indispensable for our uniqueness result, see Theorem \ref{NPC-result} below. In Section \ref{sec:sphere} we show that Definition \ref{defsol} is equivalent to the one introduced in \cite{giacomellimazonmollAML,gmmn} for the special case $\N=\mathbb S^{N-1}_+$, and in Section \ref{sec:examples} we provide two (semi-)explicit examples of solution which may help understanding its features.

\medskip

For the above notion of solution we are able to prove existence \emph{in the large}. However, we need to restrict a little the class of initial data.

\begin{defn} \label{rad}
	Let $K_\N$ denote the supremum over all sectional curvatures in all points of $\N$, and assume $K_\N$ to be finite. We say that $\bw \in BV(I, \N)$ is {\rm rad} {(and we write $\bw\in BV_{\rm
rad}(I,\N)$)} if
	\[	\dist_{\N}(\bw^-(x), \bw^+(x)) < 2\radN \quad \text{for all  } x\in I,
		\]
	where
	\begin{equation}\label{def-radN}
\radN = \left\{\begin{array}{ll}\tfrac{1}{2}\min\left\{\injN, \tfrac{\pi}{\sqrt{ K_\N}}\right\} & \text{if } K_\N>0,\\ \tfrac{1}{2}\injN & \text{if } K_\N\leq 0.\end{array} \right.
\end{equation}
\end{defn}
\noindent Clearly, if $K_\N \leq 0$ and $\injN = +\infty$, then every $\bw \in BV(I, \N)$ is rad. In the case $\N = \mathbb S^{N-1}$, $\radN=\pi/2$, hence jumps of $\bw$ need to be smaller than $\pi$, i.\,e.\ we exclude jumps between antipodal points.
The number $\radN$ is the standard estimate from below on the radius of convexity of $\N$, see \cite[Theorem 29]{petersen}.

\begin{thm} \label{existence}
	For any $\bu_0 \in BV_{\rm rad}(I, \N)$ and any $T>0$, there exists a strong solution $\bu$ to (\ref{smootheqn},\ref{smoothbc}) in $[0,T[$ such that
	\begin{equation}\label{ic}
	\bu(0) = \bu_0.
	\end{equation}
Furthermore, $\bu$ satisfies the energy inequality 
	\begin{equation}
		\label{energylimit}
		 TV_I^{\N}(\bu(t)) +
		\int_0^t\!\!\int_I |\bu_t|^2 \leq TV_I^{\N}(\bu_0) \quad \text{for a.\,e.\ }t \in ]0,T[
	\end{equation}
 and the pointwise monotonicity property
 \begin{equation}
  \label{clelimit} |\bu_x(t)|_{\N} \leq |\bu_x(s)|_{\N} \leq |(\bu_0)_x|_{\N} \quad\text{as measures for a.\,e.\ } s, t \in ]0,T[,\ s\le t.
\end{equation}
\end{thm}
\noindent To our knowledge, this is the first existence result for (\ref{smootheqn},\ref{smoothbc}) in the natural energy space for generic target manifolds. 

To clarify, \eqref{clelimit} means that $|\bu_x(t)|_{\N}(B)\leq |\bu_x(s)|_{\N}(B)\leq |(\bu_0)_x|_{\N}(B)$ for any Borel $B \subset I$ and a.\,e.\ $s,t\in ]0, T[$ with $s\le t$.
Notably, in the light of this inequality, $\bu(t) \in BV_{\rm rad}(I, \N)$ for a.\,e.\ $t\in]0,T[$.

Technically, we construct a solution using a two-step approximation procedure. In the first step, we consider a properly smoothed initial datum and produce a global \emph{regular solution} in the sense
of \cite{giacomellilasicamoll}. On the way, we obtain a Lipschitz version of estimate \eqref{clelimit} (see  \eqref{cle2} below), which is an essential tool in the proof of Theorem \ref{existence}.
\begin{lemma}\label{localest}
For any $\bu_0 \in Lip(I, \N)$ there exists a unique
 \[\bu \in C([0, +\infty[\times \overline{I}, \N)\text{ with } \bu_t \in L^2(]0, +\infty[\times I,
\R^N),\ \bu_x \in L^\infty(]0,+\infty[ \times I,
\R^{N})\]
such that there exists
\[\bz \in L^\infty(]0,+\infty[\times I,
\R^N) \text{ with } \bz_x \in L^2_{loc}([0,+\infty[ \times \overline I,
\R^N)\]
satisfying
\begin{equation}
\label{inclZ}
\bz = \tfrac{\bu_x}{|\bu_x|} \ \text{ if } \ \bu_x \neq 0, \quad \bz \in B(\boldsymbol 0,1)\cap T_{\bu}\N \ \text{ if }\  \bu_x =
 0 \qquad \Lb^{2}\text{-a.\,e.\ in }]0,+\infty[\times I,
\end{equation}
\begin{equation}
\label{ctvflowZ}
\bu_t = \pi_{\bu} \bz_x \qquad \Lb^{2}\text{-a.\,e.\ in }]0,+\infty[\times I,
\end{equation}
\begin{equation}
 \label{NeumannZ}
\bz = 0 \qquad \Lb^1\text{-a.\,e.\ in }]0,+\infty[\times \partial I
\end{equation}
and
\begin{equation} \label{icapprox}
\bu(0, \cdot) = \bu_0.
\end{equation}
Furthermore, for a.\,e.\ $t>0$
\begin{equation}
\label{energyinequality}
\int_I |\bu_x(t, \cdot)| +
\int_0^{t}\!\!\! \int_I |\bu_t|^2 \leq \int_I |\bu_{0,x}|
\end{equation}
and
\begin{equation}
\label{cle}
|\bu_x(t,\cdot)|\le |\bu_{0,x}|\qquad  \Lb^1\text{-a.\,e.\ in } I\, .
\end{equation}

\end{lemma} 	
It follows from uniqueness that the approximate solutions in Lemma \ref{localest} satisfy
\begin{equation}
\label{cle2}
|\bu_x(t,\cdot)|\le |\bu_x(s,\cdot)|\leq |\bu_{0,x}|\quad  \Lb^1\text{-a.\,e.\ in } I\quad \text{for a.\,e.\ } 0\leq s\leq t.
\end{equation}

In the second step, given an initial datum of bounded variation, we strictly approximate it with Lipschitz regular functions, for which existence of the flow is guaranteed by Lemma \ref{localest}. Then,
we use the energy inequality \eqref{energyinequality} to obtain the existence of a pair $(\bu,\bz)$ satisfying the regularity assumptions in Definition \ref{defsol} and $\bmu\in L^\infty_{w^*}(]0,T[;
M(I,\R^N))$ satisfying
\[
\bu_t(t)=\bz_x(t)+\bmu(t) \quad \text{for a.\,e.\ } t\in ]0,T[.
\]
The hardest part of this work is the identification of the vector field $\bz$ (and the corresponding measure $\bmu$). This identification is split into two parts, outside the jump set of the
solution and inside it. In both parts, we heavily use estimate \eqref{cle}.

%
%
\medskip
Strong solutions satisfy a variational inequality (actually equality, see Lemma \ref{var_eq}), which in the case of non-positive sectional curvature implies uniqueness.

\begin{thm}
  \label{NPC-result} Let $\N$ be such that $K_\N\leq 0$ and $\injN = \infty$, and let $T>0$. If $\bu$ is a strong solution to (\ref{smootheqn},\ref{smoothbc}) in $[0,T[$ then, for all $\bv\in BV(I,\mathcal N)$,
  \begin{equation}\label{var_ineq}
  \frac{1}{2}\frac{\dd}{\dd t}\int_I \dist_{\mathcal N}^2(\bu(t),\bv)+TV_I^{\N}(\bu(t))\leq TV_I^{\N}(\bv) \quad \text{for a.\,e.\ }t \in]0,T[.
\end{equation}
Hence, $\bu$ is the only strong solution to (\ref{smootheqn},\ref{smoothbc}) in $[0,T[$ with initial datum $\bu(0)$.
\end{thm}

Uniqueness of solutions to the variational inequality \eqref{var_ineq} follows from \cite[4.013]{ambrosiogiglisavare}. Moreover, since $TV_I^{\N}$ is (geodesically) convex if $K_\N\leq 0$ (\cite[Theorem 4.3]{CR-M-PJ}), solutions coincide with those given by the abstract theory of gradient flows for geodesically $\lambda$-convex and lower semicontinuous functionals by Ambrosio--Gigli--Savaré. 

\medskip

On the other hand, we do not know whether uniqueness holds in general for $BV_{rad}(I,\N)$ initial data. However, we believe that both the upper bound $\dist_{\N}(\bu^-(t,x), \bu^+(t,x)) < \injN$ and \eqref{zeqnj}---in particular, the assumption that $\bz^\pm$ are tangent to a minimizing segment---are indispensable for a uniqueness result.

\medskip

The variational inequality can also be used to show that strong solutions converge to a constant in finite time.

\begin{thm} \label{thm:asymp}
	Let $\bu_0\in BV_{rad}(I,\N)$. There exists $T^*\geq 0$ such that if $\bu$ is a strong solution to (\ref{smootheqn},\ref{smoothbc}) in $[0,T[\supsetneqq [0,T^*[$ with initial datum $\bu_0$ and (\ref{energylimit},\ref{clelimit}) hold, then there exists $\overline{\bu} \in \N$ such that $\bu\equiv \overline{\bu}$ in $[T^*, T[$.
\end{thm}

To clarify, $\overline{\bu}$ might depend on the particular solution $\bu$ (which, as far as we know, may in general be non-unique); however $T^*$ is determined by $\bu_0$, albeit in a complicated manner. Thus, Theorem \ref{thm:asymp} implies that we can actually take $T= +\infty$ in Theorem \ref{existence}.

We note that this finite-time stopping property is related to the Neumann boundary condition \eqref{zbc}. It may fail in the case of Dirichlet boundary conditions \cite{GigaKuroda2015}. Recall that for the scalar total variation flow in one dimension, the property holds in both Neumann and Dirichlet case \cite{bonfortefigalli}.

\medskip
The plan of the paper is as follows: in Section \ref{sec:preli}, we recall some definitions and results about manifolds and BV functions and we show how the representation formula
\eqref{representation_form} can be obtained. In Section \ref{sec:counter}, we provide counterexamples to $\lambda$-convexity of $TV_I^{\N}$. Section \ref{sec:cle} is devoted to the proof of Lemma {\ref{localest}}. In
Section \ref{sec:pas}, we give the proof of Theorem \ref{existence}. 
In Section \ref{sec:unique} we show that strong solutions satisfy a variational equality and deduce the uniqueness result, Theorem \ref{NPC-result}. In Section \ref{sec:asymp} we prove Theorem \ref{thm:asymp} on asymptotic behavior.
In Section \ref{sec:sphere} we deal with the special case $\N=\mathbb S^{N-1}_+$ and in Section \ref{sec:examples} we provide two (semi-)explicit examples of solutions. Finally, we collect in the Appendix several results in Riemannian geometry that we use in the paper.

\section{Preliminaries}\label{sec:preli}

In this section we introduce some notation and some preliminary
results that we need in the sequel.

\medskip

 {\noindent\bf General notations}

\smallskip

 Throughout this paper, $C$ denotes a generic positive constant depending on $\N$, $|I|$, $T$ and $\|u_0\|_{BV(I,\N)}$, that may vary from line to line; further dependencies will be stated
 explicitly. $\mathcal H^{0}$ denotes the $0$-dimensional Hausdorff measure on $\R$ and $\mathcal L=\mathcal L^1$ the $1$-dimensional Lebesgue measure. Given an open set $\Omega\subset \R$, we denote by
 $M(\Omega,\R^N)$ the space of $\R^N$-valued finite Radon measures on $\Omega$ (see \cite[Def. 1.40]{afp}). {For $J\subset \R$, we will often write $\int_J f$ instead of $\int_J f \dd\mathcal L$.} Given two measures $\bnu,\bmu\in  M(\Omega,\R^N)$, we write $\bnu\ll \bmu$ if $\nu$ is
 absolutely continuous with respect to $\bmu$. In this case, $\frac{\bnu}{\bmu}$ denotes the Radon--Nikodym derivative of $\bnu$ with respect to $\bmu$.

If $ A\subset \R^N$ and $\Upsilon(\Omega,\R^N)$ is a subspace of the space of $\Lb^m$-measurable functions, we use the notation
\[\Upsilon(\Omega,A):=\left\{\bu\in \Upsilon(\Omega,\R^N)\colon \ \bu(x)\in A \ \text{ for }\Lb^m\text{-a.\,e.\ }x\in \Omega\right\}.\]
For instance, $Lip(I,\mathcal N)$ will denote the space of $\N$-valued Lipschitz functions on $I$.

\medskip

{\noindent\bf Functions of bounded variation on an interval}

\smallskip

It is well known (see \cite[Theorem 3.28]{afp}) that any $\bw \in BV(I, \RN)$ admits a representative that has one-sided {left and right} limits, $\bw(x)^\pm$, at every point $x\in I$. We will always
identify any $\bw \in BV(I, \RN)$ with such a representative. We will also identify $\bw$ with its extension to $\overline{I}$ by its respective one-sided limits. The limits $\bw^\pm$ coincide
everywhere except on a countable set $J_{\bw}$, the jump set of $\bw$.  Furthermore, the distributional derivative of any $\bw \in BV(I, \RN)$ coincides with a vector Radon measure that we will also
denote $\bw_x$. The measure $\bw_x$ can be decomposed into diffuse and atomic (supported on $J_{\bw}$) parts, i.\,e.
\[
\bw_x = \bw_x^d + \bw_x^j, \quad \text{where } \
\bw^j_x = (\bw^+ - \bw^-) \Hd^0\res_{J_{\bw}}\ \ \text{ and }\ \bw_x^j \perp \bw_x^d.
\]
On the other hand, using Radon--Nikodym theorem, we can also decompose $\bw_x$ into absolutely continuous (with respect to the Lebesgue measure $\Lb$) and singular parts:
\[ \bw_x = \bw_x^a \,\Lb + \bw_x^s, \quad \bw_x^s \perp \Lb.\]
We note that $\bw_x^j \perp \Lb$. Thus, denoting by $\bw_x^c$ the Cantor part $\bw_x^c \colon \!\!\! = \bw_x^s - \bw_x^j$, we have
\[\bw_x = \bw_x^a \,\Lb + \bw_x^c + \bw_x^j.\]
We will also make use of the precise representative $\bw^*$ of $\bw \in BV(I,\R^N)$, defined by $\bw^*= (\bw^++\bw^-)/2$. Since $\bp \mapsto \pi_{\bp}$ can be identified with a smooth function $\N \to
 \R^{N \times N}$, we have $\pi_{\bw} \in BV(I, \R^{N \times N})$  and we can define in the same way $\pi_{\bw}^*= (\pi_{\bw}^++\pi_{\bw}^-)/2$.
We state, in the very particular case we need, the following Leibniz formula (proved in \cite{crasta_decicco}).
%
%
\begin{thm} Let $\bw,\bz\in BV(I,\R^N)$. Then $\bw \cdot \bz \in BV(I)$ and
		\begin{equation}\label{green}
			(\bw\cdot \bz)_x = \bw^* \cdot \bz_x + \bz^* \cdot \bw_x.
	\end{equation}
\end{thm}

\medskip{\noindent\bf Comparison of distances in $\N$ and $\R^N$}\smallskip

We will use the following observation, which is proved in the Appendix.
\begin{lemma}\label{lem-dist-comp}
Let $(\N,g)$ be a complete and connected n-dimensional Riemannian manifold without
boundary, isometrically and closedly embedded
in $\R^N$. Then for every $R>0$ there exists $C_R\ge 1$ such that
\begin{equation}\label{equiv-metr}
C_R^{-1} \dist_\N(\bp_1,\bp_2)\le |\bp_1-\bp_2|\le \dist_\N(\bp_1,\bp_2) \quad\forall \bp_1,\bp_2\in \overline{B(0,R)}\cap \N.
\end{equation}
\end{lemma}

\noindent {\bf Relaxation}
\medskip

Relaxation results for the functional $\mathcal{TV}_I^{\N}$ were obtained by Giaquinta and Mucci  \cite[Proposition 6.1,Theorem 6.2]{giaquinta-mucci} \cite{giaquinta-mucci}. The authors of
\cite{giaquinta-mucci} worked under a standing assumption that $\N$ is compact and oriented. However, we observe that neither of these restrictions is essential for this result in the 1-dimensional
case. First of all, orientability is never used in the proofs in the case of a 1-dimensional domain. Secondly, by Lemma \ref{lem-dist-comp}, inside a ball distances in $\N$ are comparable with distances
in $\R^N$. A sequence that converges weakly-* in $BV(I, \R^N)$ has a subsequence that converges a.\,e.\ and therefore, the proof of Proposition 6.1 works in our case. Theorem 6.2 already treats the case
of a 1-dimensional domain and the strictly convergent sequence constructed in Theorem 6.2 is already contained in a ball in $\R^N$. Thus, we can state the following lower semicontinuity and density results.

\begin{lemma}\label{GM_lsc}
	Let $\bw \in BV(I, \N)$. For every sequence $\bw^k \in C^1(I, \N)$ such that $\bw^k \stackrel{*}{\rightharpoonup} \bw$ in $BV(I, \R^N)$,
	\[\liminf_{k \to \infty} \int_I |\bw^k_x| \geq \left|\bw_x\right|_{\N}(I).\]
\end{lemma}
\begin{lemma}\label{GM_density}
	For any $\bw\in BV(I,\N)$ there exists a sequence $\bw^k\in C^1(I,\N)$ such that $\bw^k \to \bw$ strictly in $BV(I, \N)$, by which we mean that
	\begin{equation} \label{strict_conv_N}
\bw^k \stackrel{*}{\rightharpoonup} \bw \ \mbox{ in } \ BV(I, \R^N)\qquad\mbox{and}\qquad \int_I |\bw^k_x|  \to  \left|\bw_x\right|_{\N}(I).
\end{equation}
In particular, as $k\to +\infty$,
		\begin{equation}\label{u0d-1-cons}
\int_{a}^{b} |\bw^k_{x}| \to \int_{]a, b[}\!\!\!\! \dd |\bw_{x}|_{\N} \quad\mbox{if $a,b\notin J_{\bw}$.}
		\end{equation}
\end{lemma}
\noindent

The limit in \eqref{u0d-1-cons} follows from \eqref{strict_conv_N} and lower semi-continuity via an indirect argument.
From these results we deduce
\begin{cor} \label{relaxation}
	The representation formula \eqref{representation_form}-\eqref{def:measure} holds.
\end{cor}

We will also need the following mild generalization of Lemma \ref{GM_lsc}.
\begin{lemma}\label{phi_lsc}
	Let $\varphi \in C_c(I)$ be nonnegative and let $\bw \in BV(I, \N)$. If a sequence $\bw^k \in C^1(I, \N)$ is such that $\bw^k
\stackrel{*}{\rightharpoonup} \bw$ in $BV(I, \R^N)$, then
	\begin{equation}\label{phi_lsc_eq}
\liminf_{k \to \infty} \int_I \varphi |\bw^k_x| \geq \int_I \varphi \dd\! \left|\bw_x\right|_{\N}.
\end{equation}
\end{lemma}
\begin{proof} 	Let $\bw^{k_l}$ be a subsequence of $\bw^k$ such that
	\[\lim_{l \to \infty} \int_I \varphi |\bw^{k_l}_x| = \liminf_{k \to \infty} \int_I \varphi |\bw^k_x|  . \]
	Since the sequence $|\bw^{k_l}_x|$ is bounded in $M(I)$, there
exists a non-negative measure $\widetilde{\nu} \in M(I)$ and a subsequence which we again label $\bw^{k_l}$ such that $|\bw^{k_l}_x| \stackrel{*}{\rightharpoonup} \widetilde{\nu}$ in $M(I)$. Let $J$ be
any open interval compactly contained in $I$ such that $\widetilde{\nu}(\partial J) = \left|\bw_x\right|_{\N}(\partial J)= 0$. By \cite[Proposition 1.62(b)]{afp},
	\[\int_{\overline{J}} |\bw^{k_l}_x|\to \widetilde{\nu}(\overline{J}).\]
	Thus, by Lemma \ref{GM_lsc} (with $J$ in place of $I$)
\begin{equation}\label{star1}
\widetilde{\nu}(\overline{J}) = \lim_{k \to \infty}\int_{\overline{J}} |\bw^{k_l}_x|  = \lim_{k \to \infty}\int_{J} |\bw^{k_l}_x| \geq  \left|\bw_x\right|_{\N} (J) =
\left|\bw_x\right|_{\N}(\overline{J}).
\end{equation}
Property \eqref{star1} is satisfied for all intervals $J$ with $\widetilde{\nu}(\partial J) = \left|\bw_x\right|_{\N}(\partial J)= 0$. Hence,
by the Besicovitch covering lemma \cite[1.5.2., Corollary 1]{evansgariepy}, for any open set $U \subset I$ there exists a countable family $\mathcal J$  of disjoint closed intervals contained in $U$ whose boundary points are not charged by $\widetilde{\nu}$ or $\left|\bw_x\right|_{\N}$ and such that $\widetilde{\nu}(U \setminus \bigcup \mathcal{J}) = 0$. Therefore, it follows from \eqref{star1} that
\[
\widetilde{\nu}(U) \geq \left|\bw_x\right|_{\N}(U) \qquad\mbox{for any open $U \subset I$.}
\]
In view of \cite[1.1.1., Lemma 1]{evansgariepy}, given a Borel set $B\subset I$ and $\delta > 0$ we can find an open set $U\supset B$ with $\widetilde{\nu}(U \setminus B) \le \delta$. Therefore $\left|\bw_x\right|_{\N}(B)\le \left|\bw_x\right|_{\N}(U) \le \widetilde{\nu}(U) \le \widetilde{\nu}(B)+\delta$. As $\delta>0$ is arbitrary,
%
%
we deduce
\begin{equation}\label{pippo}
\widetilde{\nu} \geq \left|\bw_x\right|_{\N} \qquad\mbox{as Borel measures in $I$}
\end{equation}
and \eqref{phi_lsc_eq} follows from the definition of $\widetilde{\nu}$.
%
%
\end{proof}

\medskip
\noindent
{\bf Weakly-* measurable functions}
\smallskip

Let $X$ be a Banach space with dual $X^*$. We say that a function $f \colon [0,T] \to X^*$ (possibly defined up to a Lebesgue null set) is weakly-* (Lebesgue) measurable if
\[[0,T] \ni t \mapsto \langle f(t), w\rangle \]
is (Lebesgue) measurable for all $w \in X$. This notion gives rise to spaces $L^{p'}_{w^*}(0,T; X^*)$ of $p'$-integrable weakly-* measurable functions, $p' \in ]1, \infty]$, cf.\;\cite[Chapter
VI]{IonescuTulcea}.
Similarly, a function $f \colon [0,T] \to X$ is weakly (Lebesgue) measurable if
\[[0,T] \ni t \mapsto \langle w, f(t)\rangle \]
is (Lebesgue) measurable for all $w \in X^*$ and $L^p_w(0,T; X)$ denotes the corresponding space of $p$-integrable weakly measurable functions, $p \in [1, \infty[$.

\smallskip

The feature that makes spaces of weakly-* measurable functions indispensable to us, is the following representation theorem, which will allow us to extract convergent sequences via the Banach--Alaoglu theorem:
\begin{equation}\label{iso}
\mbox{$L^\infty_{w^*}(0,T; X^*)$ is isometric to $L^1(0,T;X)^*$}
\end{equation}
(see \cite[Chapter VII, Theorem 7]{IonescuTulcea} or \cite[Corollary 2.3]{schwartz}). We note that (see \cite[Proposition 2.1]{FUSCO20181370}) $BV(I)=X^*$ with $$X=\{\varphi+\psi_x : \varphi, \psi\in C_0(I)\}.$$ Therefore, in view of \eqref{iso} one can show that
\begin{equation}\label{impl}
\left\{\begin{array}{cc} w^k \text{ bounded in } L^\infty(]0,T[\times I) \\
w^k \stackrel{*}\rightharpoonup w \text{ in } L^\infty_{w^*}(0,T; BV(I))\end{array}\right. \ \Longrightarrow\  \left\{\begin{array}{l} w^k \stackrel{*}\rightharpoonup w \text{ in } L^\infty_{w^*}(0,T; L^\infty(I)), \\  w^k_x
\stackrel{*}\rightharpoonup w_x \text{ in } L^\infty_{w^*}(0,T; M(I)).\end{array}\right.
\end{equation}
%
%
%
Pettis' theorem (\cite[V.4.]{Yosida})  implies that if $X$ is separable,
then $L^p_w(0,T; X)$ coincides with the Bochner space $L^p(0,T; X)$ for all $p\in [1,\infty]$. Therefore, if $X^*$ is separable and reflexive, $L^{p'}_{w^*}(0,T; X^*)$ coincides with the Bochner space
$L^{p'}(0,T; X^*)$ for all $p'\in [1,\infty]$. This is not true in the cases $X^* = M(I)$, $X^* = BV(I)$, $X^* = L^\infty(I)$ which are of interest to us. However, by the Sobolev embedding $BV(I)
\subset L^q(I)$, 
reflexivity and separability of $L^q$ spaces, $1<q<\infty$, and Pettis' theorem, we have
\[
L^{p'}_{w^*}(0,T;BV(I)) \subset L^{p'}_{w^*}(0,T;L^q(I)) = 
L^{p'}(0,T;L^q(I))\subset L^{p'}(0,T;L^1(I))\quad\mbox{for $q \in ]1, \infty[$.}
\]
One can check that
\[
L^p(0,T;L^p(I)) \subset L^p(]0,T[\times I) \subset L_w^p(0,T;L^p(I)),
\]
so by Pettis' theorem all three spaces coincide for $p\in [1,\infty[$. More precisely, the map assigning to $f \in L^{p}(]0,T[\times I)$ an $L^{p}(I)$-valued function $\widetilde{f}$ given by
$\widetilde{f}(t) = f(t,\cdot)$ is an isomorphism between them. Therefore
\[
L^\infty_{w^*}(0,T, L^\infty(I)) \stackrel{\eqref{iso}}= L^1(0,T;L^1(I))^* = L^1(]0,T[\times I)^* = L^\infty(]0,T[\times I),
\]
a fact which we implicitly use in the sequel.

We also need the following lemma about measurability.

	\begin{lemma} \label{lem-salva}
		For $p \in [1, \infty[$, let $g \in L^p_{w^*}(0,T;BV(I))$ and $\nu \in L^{p'}_{w^*}(0,T; M(I))$. Then the function
		\[t \mapsto \int_I g^*(t) \dd \nu(t)\]
		is Lebesgue measurable.
	\end{lemma}
	\begin{proof}
		We recall that $g \in L^p(0,T;L^p(I))$. Let $\varrho_\eps$ be a standard, symmetric, smooth approximate identity. For $j\in \mathbb N$ we denote by $I_j$ the set of points in $I$ whose distance
to $\partial I$ is at least $2/j$ and let
		\[g_j(t) = \varrho_{\frac{1}{j}}* \left(\mathbf 1_{I_j} \ g(t)\right).\]
		Then $g_j \in L^p(0,T;C_0(I))$, whence the function
		\[t \mapsto \int_I g_j(t) \dd \nu(t)\]
		is Lebesgue measurable (and belongs to $L^1(0,T)$). Let us now fix $t \in [0,T]$ such that $g(t) \in BV(I)$, $\nu(t) \in M(I)$. Since $g_j(t)$ are uniformly bounded and converge pointwise to
$g^*(t,\cdot)$, by dominated convergence
		\[\int_I g_j(t) \dd \nu(t) \to \int_I g^*(t) \dd \nu(t) \quad\mbox{as $j \to +\infty$,}
\]
hence the right hand side is measurable (as a pointwise limit of measurable functions).
\end{proof}

For simplicity, in this subsection we have only considered spaces of scalar-valued functions on $I$. Of course, analogous facts can be deduced for the vector-valued case.

\medskip
\noindent
{\bf Riemannian comparison theorems}
\smallskip

We will use two well-known theorems in Riemannian geometry that allow comparing our given manifold $\N$ with sectional curvature bounded above to a manifold with constant sectional curvature (\emph{space form}). We recall their statements here in generality that will be convenient to us. It will suffice to consider the case where the comparison manifold has non-negative curvature, i.\,e.\ it is a sphere or a Euclidean space.

The first of the two results can be seen as a corollary of the general Hessian comparison theorem for the distance function, see e.\,g.\ \cite[Theorem 1.1.]{SchoenYau1994}). Here we are interested in
the Hessian of the squared distance function $r_{\bp_0}^2$, where we denoted
\[r_{\bp_0}(\bp) =  \dist_{\N}(\bp, \bp_0).\]
Recall that for a smooth function $f \colon \N \to \R$, $\Hess f$ is a $(0,2)$-tensor which may be defined as
\[\Hess f (X,Y) = \langle \nabla_X \dd f, Y\rangle\]
where $\dd$ is the exterior derivative and $\nabla$ is the (Levi-Civita) covariant derivative on $\N$. Using basic properties of the Levi-Civita connection, we can rewrite
\[\Hess f (X,Y) = X(Yf) - (\nabla_X Y) f = Y(Xf) - (\nabla_Y X) f\]
which shows that the Hessian is symmetric. We note that in our case
\[\langle \dd \tfrac{1}{2}r_{\bp_0}^2, X\rangle(\bp) = - \log_{\bp} \bp_0 \cdot X(\bp) \]
provided that $r_{\bp_0}(\bp) < \injN$ \cite{jostriemannian} and
\begin{equation}\label{char-hess}
\Hess \tfrac{1}{2}r_{\bp_0}^2(X,Y) = - Y \cdot D_{\bp}\log_{\bp}\bp_0\, X.
\end{equation}

We define \begin{equation}\label{cotangent}h_{\N}(\sigma):= \left\{\begin{array}
		{ll} 1 & \text{if } K_\N\leq 0 \\  \sqrt{K_{\N}} \sigma \cot(\sqrt{K_{\N}} \sigma) & \text{if } K_\N >0.
	\end{array}\right.\end{equation}
\begin{lemma} \label{hct}
	Suppose that $r_{\bp_0}(\bp) < 2 \radN$. Then
	\[(\Hess \tfrac{1}{2} r_{\bp_0}^2)(\bp)(X,X)\geq h_\N(r_{\bp_0}(\bp))|X|^2.
	\]
\end{lemma}
The Hessian comparison theorem seems to be most commonly formulated in the literature in terms of the distance function $r_{\bp_0}$ itself. Lemma \ref{hct} can be deduced from such a statement in \cite[Theorem 27.]{petersen} by simple calculation
\[\Hess \tfrac{1}{2} r_{\bp_0}^2 = \nabla \dd \tfrac{1}{2} r^2_{\bp_0} = \nabla( r_{\bp_0} \dd r_{\bp_0}) = r_{\bp_0} \Hess r_{\bp_0} + \dd r_{\bp_0} \otimes \dd r_{\bp_0}. \]

We will also use the following Alexandrov comparison theorem \cite[4.1]{karcher-chern}.
\begin{lemma} \label{act}
  If $K_{\N} > 0$, let $S$ be the sphere of radius $1/\sqrt{K_{\N}}$. Otherwise, let $S$ be the Euclidean plane. Suppose that $\triangle$ is a geodesic triangle in $\N$ whose perimeter is smaller than $4 \radN$. Then, there exists a triangle $\triangle_s$ in $S$ with the same edge lengths as $\triangle$ satisfying
  	\[\alpha \leq \alpha_s, \quad \beta \leq \beta_s, \quad \gamma \leq \gamma_s, \]
  where $\alpha, \beta, \gamma$ are the angles of $\triangle$ and $\alpha_s, \beta_s, \gamma_s$ are the corresponding angles of $\triangle_s$.

Moreover, given two adjacent segments in $\mathcal N$ with lengths $a$ and $b$ and given the angle $\gamma$ between them, let $c$ be the length of the opposite segment and let $\triangle$ be the resulting triangle. Then there exists a triangle $\triangle_s$ in $S$ with two adjacent edges of lengths $a$ and $b$ and with angle $\gamma$ between them, and the length $c_s$ of the opposite edge satisfies $$c_s\leq c.$$

\end{lemma}

\section{Counterexamples to geodesic semiconvexity}\label{sec:counter}

Let $\lambda\in \R$ and $(X,d)$ a metric space. A curve $[0,1] \ni t \mapsto \bgamma_t \in X$ is a \emph{constant speed geodesic} if
\[d(\bgamma_0, \bgamma_t) = t\, d(\bgamma_0, \bgamma_1) \quad \text{for all } t \in [0,1].\]
A functional $\F\colon X\to ]-\infty,+\infty]$ is \emph{geodesically $\lambda$-convex} if for any $\bu,\bv\in X$ there exists a constant speed geodesic $\bgamma_t$ in $X$ from $\bu$ to $\bv$ such
that
\begin{equation}\label{def-l-conv}
\F(\bgamma_t)\le (1-t)\F(\bu)+t\F(\bv) -\tfrac12 \lambda t(1-t) d^2(\bu,\bv) \quad\mbox{for all $t\in[0,1].$}
\end{equation}
We call a functional $\F$ \emph{geodesically semiconvex} if it is geodesically $\lambda$-convex for some $\lambda\in \R$. In this section we consider $\N=\S^2$ and $X=L^2(I,\S^2)$ with the intrinsic distance
\[ d^2(\bu,\bv)= \int_I \dist_{\S^2}^2(\bu(x),\bv(x)). \]
We first observe that in this case, given $\bu, \bv \in X$, any curve $[0,1] \ni t \mapsto \bgamma_t \in X$ such that 
\begin{itemize}
\item $\bgamma_0 = \bu$, $\bgamma_1 = \bv$,  
\item $t \mapsto \bgamma_t(x)$ is a constant speed geodesic in $\N$ for a.\,e.\ $x \in I$  
\end{itemize}
is a constant speed geodesic in $X$. It is then easy to see that $TV_I^{\S^2}$ is not geodesically semiconvex. Indeed, choosing constant functions $\bu_\pm \equiv (0,0,\pm 1)$, any measurable function $\bw \colon I \to \mathbb S^2$ whose range is restricted to the equator $\mathbb S^1 \times\{0\}$ is a geodesic midpoint between $\bu_+$ and $\bu_-$. Naturally, we may choose such a function to have arbitrarily large, or infinite, total variation. A variant of this counterexample with $\bu_\pm$ jumping between $(\pm \cos \alpha, \sin \alpha, 0)$ and $(\pm \cos \alpha, -\sin \alpha, 0)$, $\alpha \to 0^+$, can be produced to show that a restriction of $TV_I^{\S^2}$ to $BV_{\rm rad}(I, \N)$ also fails to be semiconvex. One could then suspect that excluding sequences converging to antipodes could help. However, this is also not the case. In fact we shall see that no matter how we restrict the range of considered curves, it does not give semiconvexity. For $r > 0$, let us define 
$\mathcal F_r\colon L^2(I{,{\S^2}})\to [0,+\infty]$ by
\[
\mathcal{F}_{r}(\bu):=\left\{\begin{array}
	{ll} \displaystyle \int_I \dd |\bu_x|_{\N} & \mbox{if $ \ \bu\in BV(I,\mathbb S^2) \ $ and $\ \bu(x) \in \overline{B_{\mathbb S^2}(\be_1,r)}$ for a.\,e.\ $x \in I$,} \\ +\infty & {\rm otherwise, }
\end{array}\right.
\]
where we have arbitrarily chosen $\be_1 = (1,0,0) \in \mathbb S^2$ for the center of the geodesic ball. 
 
\begin{prop}
$\mathcal{F}_r$ is not geodesically semiconvex for any $r>0$.
\end{prop}

\begin{proof}
For $a>0$ sufficiently small, let $\bp_0, \bp_1, \bq_0, \bq_1 \in \S^2$ be the vertices of a square of side $a$ on $\S^2$ centered at the point $\be_1$. We choose $\bp_0,\bq_0$ (resp.\ $\bp_1,\bq_1$) to be symmetric with
respect to the plane spanned by $\be_1$ and $\be_3$, and $\bp_0,\bp_1$ (resp.\ $\bq_0,\bq_1$) to be symmetric with respect to the plane spanned by $\be_1$ and $\be_2$ (see Figure \ref{fig:semiconvex}).
Given $\eps>0$, we consider
\[
\bu:=\bp_0\chi_{\left[0,\frac{1}{2}\right[}+\bq_0\chi_{\left[\frac{1}{2},1\right]}, \qquad \bv^\eps:=\bp_0\chi_{\left[0,\frac{1}{2}-\eps\right[}+\bp_1\chi_{\left[\frac{1}{2}-\eps,\frac{1}{2}\right[}+
\bq_1\chi_{\left[\frac{1}{2},\frac{1}{2}+\eps\right[}+\bq_0\chi_{\left[\frac{1}{2}+\eps,1\right]}.
\]
Let $\bgamma_t^\eps$ be the minimizing geodesic interpolation between $\bu$ and $\bv^\eps$ in $L^2(I, \mathbb{S}^2)$. Equivalently, $t \mapsto \bgamma_t^\eps(x)$ is the minimizing geodesic segment in $\mathbb S^2$ joining $\bu(x)$ and $\bv^\eps(x)$ for $x \in I$ (see Figure \ref{fig:semiconvex}).
\begin{figure}[htb]
  \begin{center}
    \includegraphics[width=0.3\textwidth]{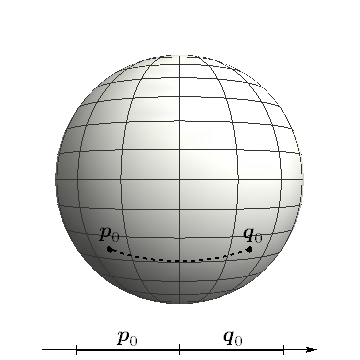}\quad \includegraphics[width=0.3\textwidth]{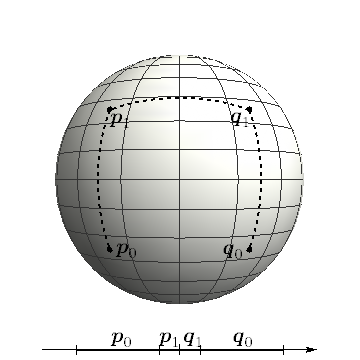} \quad
    \includegraphics[width=0.3\textwidth]{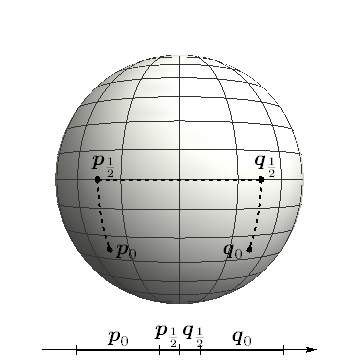}
  \end{center}
  \caption{From left to right: $\bu$, $\bv^\eps$, $\bgamma^\eps_{\frac{1}{2}}$}\label{fig:semiconvex}
\end{figure}
In order for $TV_I^{\S^2}$ to be geodesically semiconvex, the following inequality needs to be satisfied for some $\lambda \in\R$:
\[
\int_I \dd |(\bgamma_t^\eps)_x|_{\N} \leq (1-t)\int_I \dd |\bu_x|_{\N}+t\int_I\dd  |\bv^\eps_x|_{\N}-\tfrac{1}{2}\lambda t(1-t){\int_0^1}\dist_{\S^2} ^2{(\bu,\bv^\eps)}.
\]
In particular, at $t=\frac{1}{2}$, we would have
\begin{multline*}
\int_I \dd |(\bgamma_\frac{1}{2}^\eps)_x|_{\N}  \leq  \tfrac{1}{2}\dist_{\S^2}(\bp_0,\bq_0)+\tfrac{1}{2}\big(\dist_{\S^2}(\bp_0,\bp_1)+\dist_{\S^2}(\bp_1,\bq_1) +\dist_{\S^2}(\bq_1,\bq_0)\big)
\\  -\tfrac{1}{8}\lambda \eps (\dist_{\S^2}^2(\bp_0,\bp_1)+\dist_{\S^2}^2(\bq_0,\bq_1))=
2a-\tfrac14 \lambda{\eps}a
\end{multline*}
However,
\[
\int_I \dd |(\bgamma_\frac{1}{2}^\eps)_x|_{\N} = \tfrac{1}{2}\big(\dist_{\S^2}(\bp_0,\bp_1)+\dist_{\S^2}(\bq_1,\bq_0)\big) +  \dist_{\S^2}(\bp_{\frac{1}{2}},\bq_\frac{1}{2}) =
a+\dist_{\S^2}(\bp_{\frac{1}{2}},\bq_\frac{1}{2}),
\]
where $\bp_\frac{1}{2},\bq_\frac{1}{2}$ are the projections of $\bp_0,\bq_0$ onto the greatest circle parallel to $\be_1,\be_2$ (the equator in Fig.\ \ref{fig:semiconvex}). 
Therefore, one needs
\[a<\dist_{\S^2}(\bp_\frac{1}{2},\bq_\frac{1}{2})\leq a-\frac{\lambda\eps a}{4},
\]
which leads to a contradiction for any $\lambda\in \R$, by choosing $\eps=\eps_\lambda$ sufficiently small.
\end{proof}

Next, we modify the example above to show that $TV_I^{\S^2}$ fails to be semiconvex along geodesics even in the set where our solutions are found, i.\,e.\ the subset of $BV(I, \N)$ satisfying
\eqref{cle} with respect to a given datum $\bw_0$ (were it the case, one could still apply the general theory in \cite{afp} to obtain existence of solutions). Given $\bw_0\in BV(I,{\S^2})$, we thus
consider the functional $\mathcal F_{\bw_0}\colon L^2(I{,{\S^2}})\to [0,+\infty]$ defined by
\[
\mathcal{F}_{\bw_0}(\bu):=\left\{\begin{array}
   {ll} \displaystyle \int_I \dd |\bu_x|_{\N} & \mbox{if $ \ \bu\in BV(I,\mathbb S^2) \ $ and $\ |\bu_x|_{\N}\leq |\bw_{0,x}|_{\N}$} \\ +\infty & {\rm otherwise }
 \end{array}\right.
\]
\begin{prop}
There exists $\bw_0 \in BV_{\rm rad}(I,{\S^2})$ such that $\mathcal{F}_{\bw_0}$ is not geodesically semiconvex.
\end{prop}

\begin{proof}
We consider
\[
\bw_0:=\sum_{n=1}^\infty
{\left(\cos\tfrac{n-1}{n},\sin\tfrac{n-1}{n},0\right)
\chi_{[a_{n-1},a_n[}, \quad\text{where } a_n=\frac{2^n-1}{2^n}, \  n\in \mathbb N.
}
\]
Observe that $\bw_0\in BV(I{,\S^2})$ and that
\[
|(\bw_0)_x|_{\N}=\sum_{n=1}^\infty \frac{1}{n(n+1)}\delta_{a_n}.
\]
We are going to show that for any $n\in\mathbb N$ there exist $\bu^n,\bv^n\in BV(I,\S^2)$ such that $|\bu^n_x|_{\N} \leq |(\bw_0)_x|_{\N}$, $|\bv_x^n|_{\N}\leq |(\bw_0)_x|_{\N}$ and
\begin{equation}\label{counterexample}
\int_I \dd |(\bgamma_\frac{1}{2}^n)_x|_{\N} >\frac{1}{2}\int_I \dd|\bu^n_x|_{\N}+\frac{1}{2}\int_I \dd|\bv^n_x|_{\N}+n \int_I \dist_{\S^2} ^2(\bu^n,\bv^n), 
\end{equation}
$\bgamma_t^n$ being the geodesic interpolation between $\bu^n$ and $\bv^n$. Inequality \eqref{counterexample} implies that, for any $\lambda \in \R$, $\mathcal{F}_{\bw_0}$ does not satisfy
\eqref{def-l-conv} for all $\bu,\bv\in BV(I)$, whence the statement.

\smallskip

Let $\bp_{0n}, \bp_{1n}, \bq_{0n}, \bq_{1n} \in {\S^2}$ be the vertices of a square of side $(4n(n+1))^{-1}$ on $\S^2$ centered at the point $(1,0,0)$. We take $\bp_{0n},\bq_{0n}$ (resp.
$\bp_{1n},\bq_{1n})$ to be symmetric with respect to the plane spanned by $\be_1$ and $\be_3$ and $\bp_{0n},\bp_{1n}$ (resp. $\bq_{0n},\bq_{1n}$) to be symmetric with respect to the plane spanned by
$\be_1$ and $\be_2$. Let
\begin{align*}
\bu^n &:= \bp_{0n}\chi_{[0,a_n[}+\bq_{0n}\chi_{[a_n,1]}
\\
\bv^n &:= \bp_{0n}\chi_{[0,a_{n-1}[}+\bp_{1n}\chi_{[a_{n-1},a_n[}+\bq_{1n}\chi_{[a_n,a_{n+1}[}+\bq_{0n}\chi_{[a_{n+1},1]}.
\end{align*}
Observe that $\F_{\bw_0}$ is finite on $\bu^n,\bv^n$:
\begin{align*}
|\bu^n_x|_{\N} &= \frac{1}{4n(n+1)}\delta_{a_n}\leq |(\bw_0)_x|_{\N}, \\
|\bv^n_x|_{\N} &= \frac{1}{4n(n+1)}
\left(\delta_{a_{n-1}} + \delta_{a_n} + \delta_{a_{n+1}}\right)\leq |(\bw_0)_x|_{\N}.
\end{align*}
Letting $\bp_{\frac{1}{2}n},\bq_{\frac{1}{2}n}$ be the projections of $\bp_{0n},\bq_{0n}$ onto the greatest circle parallel to $\be_1,\be_2$ (the equator in Fig.\ \ref{fig:semiconvex}), we note that
\begin{align*}
\left|\left(\bgamma_{\frac{1}{2}}^n\right)_x\right|_{\N}
&= \frac{1}{8n(n+1)}
\left(\delta_{a_{n-1}} +\delta_{a_{n+1}}\right)
+\dist_{\S^2}(\bp_{\frac{1}{2}n},\bq_{\frac{1}{2}n})\delta_{a_n}
\end{align*}
In order to compute $\dist_{\S^2}(\bp_{\frac{1}{2}n},\bq_{\frac{1}{2}n})$, we recall the haversine formula for geodesic triangles on the sphere:
\begin{equation}\label{haversine}
\hav c = \hav (a-b) + \sin a\sin b \hav \gamma, \qquad \hav \theta:=\sin^2\tfrac \theta 2,
\end{equation}
where $a,b,c$ are the lengths of the three sides and $\gamma$ is the opening angle opposite to the side of length $c$. We apply \eqref{haversine} to the triangle $\bp_{0n},\bq_{0n},\bp_s=(0,0,-1)$
($\bp_s$ is the south pole in Fig.\ \ref{fig:semiconvex}) with $c=\dist_{\S^2}(\bp_{0n},\bq_{0n})= \frac{1}{4n(n+1)}$, leading to
$$
\sin^2\tfrac{1}{8n(n+1)} = 0 + \sin^2\left(\tfrac\pi 2 - \tfrac{1}{8n(n+1)}\right)\sin^2 \tfrac\gamma 2 \quad\iff \quad \gamma = 2\arcsin \tan\frac{1}{8n(n+1)}.
$$
Noting that $\gamma=\dist_{\S^2}(\bp_{\frac{1}{2}n},\bq_{\frac{1}{2}n})$, we obtain
\begin{align*}
\left|\left(\bgamma_{\frac{1}{2}}^n\right)_x\right|_{\N} &=
\frac{1}{8n(n+1)}\left(\delta_{a_{n-1}} +\delta_{a_{n+1}}\right) +2\arcsin\left(\tan\frac{1}{8n(n+1)}\right)\delta_{a_n},
\end{align*}

On the other hand,
\begin{align*}
\int_0^1\dist_{\S^2}^{2}(\bu^n,\bv^n) &=
(a_n-a_{n-1})\dist_{\S^2}^2(\bp_{0n},\bp_{1n}) + (a_{n+1}-a_{n})\dist_{\S^2}^2(\bq_{0n},\bq_{1n})
\\
&= (a_{n+1}-a_{n-1})\left(\frac{1}{4n(n+1)}\right)^2 = \frac{3}{2^{n+5} n^2(n+1)^2}.
\end{align*}
Therefore, \eqref{counterexample} reads as
\[
2\arcsin\left(\tan\left(\frac{1}{8n(n+1)}\right)\right)> \frac{1}{4n(n+1)}+ \frac{3}{2^{n+5} n(n+1)^2},
\]
which holds true for $n\gg 1$ since $\arcsin(\tan x)=x+\frac{x^3}{2}+O(x^5)>x+\frac{x^3}{3}$ for $0< x\ll 1$.
\end{proof}

\section{Regular solutions and the pointwise estimate on~$|\bu_x|_{\N}$}\label{sec:cle}

In this section, we focus on the pointwise estimate
\begin{equation}
	\label{clelimit2}
	|\bu_x(t,\cdot)|_{\N}\leq |\bu_{0,x}|_{\N} \quad\text{as measures for a.\,e. }t\in]0,T[,
\end{equation}
which is a weaker version of \eqref{clelimit}, and its consequences. First, we prove Lemma \ref{localest}, which includes \eqref{cle}, a version of \eqref{clelimit} for regular solutions.
\begin{proof}[Proof of Lemma \ref{localest}]
 Local existence of a solution $\bu$, satisfying \eqref{energyinequality} up to its maximal time of existence, $T_\dagger$, is already contained in
\cite[Theorem 2]{giacomellilasicamoll}, hence we only need to show that $T_\dagger=+\infty$ and that \eqref{cle} holds. To these aims, we note that $\bu$ is constructed in \cite{giacomellilasicamoll} as limit, as $\eps\to 0^+$ and $\delta\to 0^+$ (in this order) of smooth solutions $\bu^{\eps,\delta}$ to the following approximating problems: 
\begin{align}
 \label{ctvflowapprox}
 &\bu_t = \pi_{\bu} \bz_x\quad \text{in }]0,+\infty[\times I, \qquad \bz:= \tfrac{\bu_x}{\sqrt{\eps^2 + |\bu_x|^2}},
\\
\label{Neumannapprox}
 &\bu_x = 0\quad \text{in }]0,+\infty[\times
\partial I,
\\
\label{initapprox}
 &\bu(0,\cdot) = \bu^{\delta}_0,
\end{align}
where $\bu^{\delta}_0$ is a sequence of smooth functions that converges to $\bu_0$ uniformly and
\[
\|\bu^{\delta}_{0,x}\|_{L^\infty(I,\R^N)} \to \|\bu_{0,x}\|_{L^\infty(I,\R^N)} \quad \text{as } {\delta} \to 0^+.
\]
Hereafter we let $\bu= \bu^{\eps,\delta}$  for notational convenience. Each $\bu$ satisfies the following Bochner-type identity (see \cite[Lemma 2]{giacomellilasicamoll}):
 \begin{equation}
 \label{bochner}
\frac{1}{2}\frac{\dd}{\dd t} |\bu_x|^2 =\bu_x\cdot \bu_{xt}=
(\bu_x \cdot \bz_x)_{x}
- (\pi_{\bu} \bu_{xx}) \cdot \bz_x + \underbrace{\bz \cdot \mathcal R^\N_{\bu} (\bu_{x}, \bu_{x}) \bu_{x}}_{=0}.
\end{equation}
Here we have highlighted the role of one spatial dimension in the argument: the last term in the identity vanishes due to antisymmetry of the Riemannian tensor. Using \eqref{bochner}, arguing as in
\cite{giacomellilasicamoll} one sees that
\begin{equation}
\label{boch-global}
\frac{\dd}{\dd t}\int_I (\eps^2 + |\bu_x|^2)^{\frac{p}{2}} \le 0 \qquad \text{for all } p\ge 2,
\end{equation}
whence the global a~priori bound
\[
\|\bu_x(t,\cdot)\|_{L^\infty(I, \R^N)} \le \|\bu^{\delta}_{0,x}\|_{L^\infty(I, \R^N)}
\]
holds for all $\eps>0$ and $\delta>0$. This implies that each approximating solution, and the limit itself, exist globally in time (see the proof of \cite[Theorem 4]{giacomellilasicamoll}). Therefore $T_\dagger=+\infty$.

In order to prove \eqref{cle}, we use the method from  \cite{Lasica_thesis} (see also a version for the time-discretized problem \cite{giacomellilasicacurves}). We let $v^{\eps,\delta} =\sqrt{\eps^2 + |\bu^{\eps,\delta}_x|^2}$ and omit the double index $^{\eps,\delta}$ for notational convenience. For $p\ge 1$
and $\varphi$ a Lipschitz function on $I$, we calculate as in the proof of \cite[Lemma]{giacomellilasicacurves}:
\begin{multline}
\label{pnorm}
\frac{1}{p} \frac{\dd}{\dd t} \int_I \varphi^2 v^p =
\int_I \varphi^2 v^{p-2} \bu_x \cdot \bu_{xt} = \int_I \varphi^2 v^{p-1} \bz \cdot \bu_{xt}
\\ =
-2\int_I \varphi \varphi_x v^{p-1} \bz\cdot \pi_{\bu} \bz_x - (p-1)\int_I \varphi^2 v^{p-3} (\bu_x\cdot \bu_{xx}) (\bz \cdot \pi_{\bu}\bz_x) - \int_I \varphi^2 v^{p-1} \bz_x\cdot \pi_{\bu}\bz_{x}.
\end{multline}
Since $\bz\cdot \pi_{\bu} \bz_x= \bz\cdot \bz_x$ and $\bz_x\cdot \pi_{\bu}\bz_{x}\ge 0$, we have
\begin{equation*}
\frac{1}{p} \frac{\dd}{\dd t} \int_I \varphi^2 v^p \le
-2\int_I \varphi \varphi_x v^{p-1} \bz\cdot \bz_x - (p-1)\int_I \varphi^2 v^{p-3} (\bu_x\cdot \bu_{xx}) (\bz \cdot \bz_x).
\end{equation*}
Since
\begin{equation}\label{zzx}
\bz \cdot \bz_x = \frac{1}{2} \left(|\bz|^2\right)_x = \frac{1}{2}\left(\frac{|\bu_x|^2}{v^2}\right)_x
= \frac{1}{2} \left( 1 - \frac{\eps^2}{v^2}\right)_x = \frac{\eps^2 \bu_x \cdot \bu_{xx}}{v^4},
\end{equation}
we in fact have, using Cauchy--Schwarz inequality,
\begin{multline*}
\frac{1}{p} \frac{\dd}{\dd t} \int_I \varphi^2 v^p  \le
-2\eps^2 \int_I \varphi \varphi_x v^{p-5} \bu_x \cdot \bu_{xx} - (p-1)\eps^2\int_I \varphi^2 v^{p-7} |\bu_x\cdot \bu_{xx}|^2
\\  \le \frac{\eps^2}{(p-1)} \int_I \varphi_x^2 v^{p-3} \stackrel{p<3}\le  \frac{\eps^{p-1}}{(p-1)} \int_I \varphi_x^2.
\end{multline*}
For any $x_0\in I$ and any $0<r<R$, we choose $\varphi$ such that $\varphi\equiv 1$ in $B_r(x_0)$ and $\supp \varphi = \overline B_R(x_0)$, and $|\varphi_x|=|R-r|^{-1}$ in $B_R(x_0)\setminus B_r(x_0)$. Integrating in time and dividing both sides by $2R = |B_R(x_0)|$,
\[
\frac{1}{p} \frac{1}{2R}\int_{B_r(x_0)}  v^p(t) \le \frac{\eps^{p-1}t}{(p-1)|R-r|R} + \frac{1}{p} \frac{1}{2R}\int_{B_R(x_0)}  v^p(0).
\]
Passing to the limit as $\eps \to 0^+$, $p\to 1^+$, $r\to R^-$, $R\to 0^+$ and $\delta\to 0^+$ (in this order) completes the proof.
\end{proof}

We will later construct a strong solution $\bu$ to (\ref{smootheqn}, \ref{smoothbc}) with initial datum $\bu_0 \in BV_{rad}(I, \N)$ by approximating with a sequence $\bu^k$ of regular solutions with initial data $\bu^k \in Lip(I, \N)$. We will then easily deduce that the limit satisfies the pointwise estimate \eqref{clelimit2}. It will be used repeatedly in the successive stages of identification of the limit. In particular we will use some general facts about functions satisfying \eqref{clelimit2}, which we next illustrate.

\medskip

First of all, it turns out that continuity of $\bu^\pm(\cdot, x)$, $x\in I$, can be deduced from \eqref{clelimit2}. Similarly we can estimate moduli of continuity on $x$-slices of approximate sequence $\bu^k$.

\begin{lemma} \label{contxslice}
		Let $\bu$ be any function in $L^\infty_{w^*}(0,T; BV(I, \R^N))\cap H^1(0,T; L^2(I, \R^N))$ such that $\bu(t)\in BV_{rad}(I, \N)$ for a.\,e.\ $t \in [0,T[$, satisfying \eqref{clelimit2}. Then, $\bu^\pm(\cdot,x)$ are equicontinuous on $[0,T]$, i.\,e.\ for any $\eps > 0$ there exists $\delta > 0$ such that
		\[|\bu^+(t, x) - \bu^+(s, x)| < \eps \quad \text{if } |t-s| < \delta, \ t,s \in [0,T], \ x \in I.  \]
		
		For $k \in \N$ let $\bu^k_0 \in Lip(I, \N)$ be a sequence of functions converging strictly to $\bu_0 \in BV(I, \N)$, and let $\bu^k \in C([0, +\infty[\times \overline{I}, \N)$ with $\bu^k_t \in L^2(]0, +\infty[\times I,
		\R^N)$, $\bu^k_x \in L^\infty(]0,+\infty[ \times I,
		\R^{N})$ satisfying \eqref{energyinequality} and \eqref{cle} with $\bu^k_0$ in place of $\bu_0$.
		For any fixed $x \in I \setminus J_{\bu_0}$, $\bu^k(\cdot, x)$ are equicontinuous on $[0,\infty[$, i.\,e.\ for any $\eps > 0$ there exists $\delta > 0$ such that
		\[|\bu^k(t, x) - \bu^k(s, x)| < \eps \quad \text{if } |t-s| < \delta, \ t,s \geq 0, \ k \in \bbN.  \]
	\end{lemma}
	
	\begin{proof}
		We begin with the proof of equicontinuity of $\bu^+$. Analogous reasoning works for $\bu^-$. Take $x \in I$ and an interval $J = ]a,b] \subset I$ such that $x \in \overline{J}$.  By the rectangle inequality, for a.\,e.\ $y \in J$ and a.\,e.\ $s,t\in [0,T]$ with $t>s$,
		\begin{multline}\label{rectineq}
			|\bu^+(t, x) - \bu^+(s, x)| \leq |\bu^+(t, x) - \bu(t, y)| + |\bu(t, y) - \bu(s, y)| + |\bu(s, y) - \bu^+(s, x)|\\  \stackrel{(\ref{equiv-metr})}\leq   |\bu_x(t,\cdot)|_{\N} (J) +
			\int_s^t|\bu_t(\tau,y)|\dd\tau + |\bu_x(s,\cdot)|_{\N} (J).
		\end{multline}
		Averaging \eqref{rectineq} over $y\in J$ and using H\" older's inequality, we get
		\begin{equation}\label{quantucont}
			|\bu^+(t, x) - \bu^+(s, x)|
			\stackrel{(\ref{clelimit2})}\leq 2  |\bu_{0,x}|_{\N}(J) + \left(\frac{t-s}{|J|}\int_s^t\int_J |\bu_t|^2\right)^{\frac12}.
		\end{equation}

		Let $\eps > 0$. We denote
		\[J_{\bu_0}^{\eps/3} = \{x \in J_{\bu_0} \colon |(\bu_0)_x|_{\N}(x) \geq \eps/3\}.\]
		We claim that there exists a $\sigma > 0$ such that for every interval $J \subset I\setminus J_{\bu_0}^{\eps/3}$ with $|J|\leq \sigma$ there holds $|(\bu_0)_x|_{\N}(J) < \eps/3$. Suppose otherwise, i.\,e.\ there exists a sequence of intervals $J_j\subset I\setminus J_{\bu_0}^{\eps/3}$ such that $|J_j|\to 0$ and $|(\bu_0)_x|_{\N}(J_j) \geq \eps/3$, $j \in \bbN$. By compactness, we can assume that both sequences of endpoints of $J_j$ converge to a point $x_0 \in \overline{I}$. Then
		\[ E := \bigcap_{i=1}^\infty \bigcup_{j=i}^\infty J_j \subseteq \{x_0\}.\]
		On the other hand, by \cite[1.1, Theorem 1]{evansgariepy}
		\begin{equation}\label{ineq_E} |(\bu_0)_x|_{\N}\left( E \right) = \lim_{i \to \infty} |(\bu_0)_x|_{\N}\left( \bigcup_{j=i}^\infty J_j\right) \geq \frac{\eps}{3}.
		\end{equation}
		Therefore $E$ cannot be empty, hence $E = \{x_0\}$. In particular $x_0$ belongs to some (actually infinitely many) of the sets $J_j$. Thus, \eqref{ineq_E} contradicts the choice of $J_j$. 
		
		Possibly decreasing $\sigma$, we can assume that it is smaller than the minimal distance between two distinct points in $J_{\bu_0}^{\eps/3}$. We deduce that for every $x \in I$ there exists an interval $J \subset I\setminus J_{\bu_0}^{\eps/3}$ of form $J= ]a,b]$ such that $|J| = \sigma$ and $x \in \overline{J}$: namely, if $x\in J_{\bu_0}^{\eps/3}$ we choose $J=]x,x+\sigma]$, while if $x\notin J_{\bu_0}^{\eps/3}$ we may choose any interval $J\subset I\setminus J_{\bu_0}^{\eps/3}$ of width $\sigma$ containing $x$. For such $J$, \eqref{quantucont} becomes
		\[|\bu^+(t, x) - \bu^+(s, x)| \leq \tfrac{2}{3}\eps + \left(\frac{t-s}{\sigma}\int_0^T\int_I |\bu_t|^2\right)^{\frac12}\]
		Choosing $\delta$ so that the r.\,h.\,s.\ of
		\eqref{quantucont} is smaller than $\eps$ whenever $t-s < \delta$, we obtain the first part of the assertion.
		\smallskip
		
		Repeating the calculations in \eqref{rectineq} and \eqref{quantucont} with $\bu^k$ in place of $\bu$ and $J = ]x-h, x+h[$, using \eqref{energyinequality} and \eqref{cle}, we obtain for any $x \in I$, $h>0$ sufficiently small, $k \in \mathbb N$, and a.\,e.\ $s,t\in [0,T]$ with $t>s$,
		\begin{equation}\label{quantudcont}
			|\bu^k(t, x) - \bu^k(s, x)|  \leq 2 \int_{x-h}^{x+h} |\bu^k_{0,x}| + \left(\frac{t-s}{h} \int_I |\bu^k_{0,x}|\right)^{\frac12}.
		\end{equation}
		By \eqref{u0d-1-cons}, if $x$ is not a jump point of $\bu_0$, for any $\eps>0$ we can choose $k_0\in \mathbb N, h>0$ small enough so that $\int_{x-h}^{x+h} |\bu^k_{0,x}|<\frac{\eps}{3}$ for all $k
		\geq k_0$. Since each $\bu^k_{0,x}$ is integrable, possibly reducing $h$ we can take $k_0 = 1$. Then, we choose $\delta>0$ so that the r.\,h.\,s.\ of \eqref{quantudcont} is smaller than $\eps$ whenever
		$t-s<\delta$.
\end{proof}

	As a first application of this result, we establish another lemma.
	
	\begin{lemma}
		Let $\bu$ be any function in $L^\infty_{w^*}(0,T; BV(I, \R^N))\cap H^1(0,T; L^2(I, \R^N))$ such that $\bu(t)\in BV_{rad}(I, \N)$ for a.\,e.\ $t \in [0,T[$, satisfying \eqref{clelimit2}. Let $\varphi \in L^1(0,T; C_0(I))$, $\varphi \geq 0$. Then, the function
		\[ [0,T] \ni t  \mapsto \int_I \varphi(t) \dd |\bu_x(t)|_{\N} \]
		is Lebesgue measurable.
	\end{lemma}
	\begin{proof}
		First we recall that the function
		\[[0,T] \ni t  \mapsto \int_I \varphi(t) \dd |\bu_x(t)|\]
		is Lebesgue measurable. This is because it is the pointwise supremum of the family of measurable functions
		\[\left\{t \mapsto \int_I  \varphi(t) \boldsymbol \varrho \cdot \dd \bu_x(t) \colon \boldsymbol \varrho \in C_0(I, \RN),\ |\boldsymbol \varrho|\leq 1 \right\}.\]
		Now, let us define
		\[
		f(t,x) = \dist_{\N}(\bu^-(t,x), \bu^+(t,x)) - |\bu^+(t,x) - \bu^-(t,x)|, \quad (t,x)\in Q_T.
		\]
		By \eqref{clelimit2} we have $J_{\bu(t)} \subset J_{\bu_0}$ for a.\,e.\ $t \in [0,T]$ and we can write
		\[\int_I \varphi(t) \dd |\bu_x(t)|_{\N} = \int_I \varphi(t) \dd |\bu_x(t)| + \sum_{x \in J_{\bu_0}}\varphi(t,x) f(t,x). \]
		By Lemma \ref{contxslice}, $f(\cdot, x)$ is continuous for $x \in I$. On the other hand, given any $x \in I$, evaluation at $x$ is a continuous linear functional on $C_0(I)$, hence $t \mapsto \varphi(t, x)$ is a measurable function by the ``easy" part of Pettis measurability theorem.
	\end{proof}

%

\section{Existence of solutions}\label{sec:pas}

\subsection{Basic facts}\label{sec:basic}

Let $\bu_0\in BV(I,\N)$ and let $(\bu_0^k)_{k \in \mathbb N}\subset Lip(I,\N)$ be such that (see Lemma \ref{GM_density})
\begin{equation}
\label{u0d-1}
\bu_0^k \to \bu_0\quad \text{strictly in }BV(I,\N).
\end{equation}

Let $\bu^k$ be the regular solution to (\ref{smootheqn},\ref{smoothbc}) emanating from $\bu_0^k$, as given by Lemma \ref{localest}. We recall that, for Lipschitz solutions, \eqref{ctvflowZ} may be
rewritten as
\begin{equation}\label{eq-sff}
\bu^k_t = \bz^k_x + \bmu^k,\qquad \bmu^k:=\A_{\bu^k}(\bz^k,\bz^k)|\bu^k_x|,
\end{equation}
where $\A_{\bp}$ denotes the second fundamental form of $\N$ at $\bp \in \N$. Let us denote
\[Q_T:=]0,T[\times I.\]
For any fixed $T>0$ we have 
\begin{equation}
\label{uL1}
\sup_{0<t<T}\int_I |\bu^k| \le \int_I|\bu^k_0| + \iint_{Q_T} |\bu^k_t| \stackrel{\eqref{energyinequality}}
{\le}  \int_I|\bu^k_0| + \left(T \int_I |\bu^k_{0,x}|\right)^{1/2}\stackrel{(\eqref{u0d-1})}\le C, 
\end{equation}
where we recall that $C$ denotes a generic constant which may depend on $|I|$, $T$, $\N$ and $\|\bu_0\|_{BV(I,\N)}$.
%
%
Therefore
\begin{equation}
\label{linfty}
\sup_{(t,x)\in Q_T}|\bu^k(t,x)| \le \sup_{0<t<T}\left(\frac{1}{|I|} \int_I |\bu^k| + \int_I |\bu^k_x|\right) \stackrel{(\ref{uL1}),(\ref{energyinequality}),(\ref{u0d-1})}\le C.
\end{equation}
{Because the embedding of $\N$ in $\R^N$ is closed, it follows from \eqref{linfty} that
\begin{equation}\label{Abound}
\sup_{(t,x)\in Q_T,\ |\bv_j|\le 1,\ k \in \bbN} |\A_{\bu^k(t,x)}(\bv_1,\bv_2)|\le A < + \infty,
\end{equation}
where $A$ depends on the same quantities as $C$.
%
%
Therefore
\begin{equation}\label{Abound-m-l}
|\bmu^k(t,x)| \stackrel{(\ref{Abound})}\le A  |\bu^k_x(t,x)| \stackrel{(\ref{cle})}\le A  |\bu^k_{0,x}(x)|  \quad\text{for all } (t,x)\in Q_T.
\end{equation}
In particular,
\begin{equation}
\label{Abound-m}
\sup_{0<t<T} \|\bmu^k\|_{L^1(I, \R^N)} \le A \|\bu_{0,x}^k\|_{L^1(I, \R^N)} \stackrel{(\ref{u0d-1})}\le C .
\end{equation}
We deduce that
\begin{equation}
\label{boundmudelta}(\bmu^k) \text{ is bounded in } L^\infty_{w^*}(0,T; M(I, \R^N)),
\end{equation}
\begin{equation} \label{boundmudelta|}
(|\bu^k_x|) \text{ is bounded in } L^\infty_{w^*}(0,T; M(I)).
\end{equation}

Using \eqref{inclZ}, \eqref{energyinequality}, \eqref{boundmudelta}, \eqref{boundmudelta|}, by standard arguments (see section \ref{sec:preli}) we have for a subsequence (not relabeled):
\begin{align}
\label{udw}
 \bu^k&\stackrel{*}\rightharpoonup \bu \ \text{ in } L^\infty_{w^*}(0,T;BV(I,\R^N)), \quad \bu^k\rightharpoonup \bu \ \text{ in } H^1(0,T;L^2(I)),
\\ \label{uds}
 \bu^k &\to \bu \ \text{ in } C([0,T],L^2(I, \R^N)) \text{ and } \Lb^2\text{-a.\,e.\ in } Q_T,
\\ \label{zdw}
 \bz^k &\stackrel{*}\rightharpoonup \bz \ \text{ in } L^\infty(Q_T), \quad |\bz|\le 1,
\\ \label{mdw}
\bmu^k &\stackrel{*}\rightharpoonup \bmu \ \text{ in } L^\infty_{w^*}(0,T; M(I;\R^N)),
\\ \label{weakconvergencederivative}
 |\bu^k_x| &\stackrel{*}\rightharpoonup \bnu \ \text{ in } L^\infty_{w^*}(0,T; M(I)),
\\ \label{utw}
|\bu^k_t| &\rightharpoonup \overline{|\bu_t|} \ \text{ in } L^2(Q_T),
\end{align}
and
\begin{equation}\label{eq-sff-w2}
\bu_t = \bz_x + \bmu \quad \text{in } \mathcal D'(Q_T).
\end{equation}
From \eqref{udw} it follows that $\bu^k_x \stackrel{*}\rightharpoonup \bu_x$ in $L^\infty_{w^*}(0,T;M(I,\R^N))$. We can improve that to
\begin{equation} \label{udwt}
	\bu^k_x(t) \stackrel{*}\rightharpoonup \bu_x(t) \quad \text{for a.\,e.\ } t \in [0,T].
\end{equation}
Indeed, by \eqref{uds},
\begin{equation*}  \int_I \varphi \bu^k_x(t) = - \int_I \varphi_x \bu^k(t)  \to - \int_I \varphi_x \bu(t)  = \int_I \varphi \dd \bu_x(t) \quad \text{for a.\,e.\ } t \in [0,T]
\end{equation*}
for any $\varphi \in C^1_c(I)$, a dense subset of $C_0(I)$. Furthermore, \eqref{Abound-m-l} and \eqref{cle} respectively imply that
\begin{align}
  \label{mdw2}
 |\bmu(t,\cdot)| &\le A |\bu_{0,x}| \quad \text{for a.\,e.\ }t\in]0,T[, \\ \label{mdw3}
 \bnu(t,\cdot) &\le C |\bu_{0,x}| \quad \text{for a.\,e.\ }t\in]0,T[.
\end{align}
By (\ref{udw}) and (\ref{uds}), since the embedding of $\N$ into $\R^N$ is closed, we have $\bu(t,\cdot)\in BV(I,\N)$ for a.\,e.\ $t\in[0, T[$. Using lower semicontinuity of $TV_I^\N$ and
\eqref{u0d-1}, we deduce that the energy inequality \eqref{energylimit} holds by passing to the limit in \eqref{energyinequality}. Furthermore, from \eqref{cle}, \eqref{u0d-1-cons} and lower semi-continuity, we deduce as in the proof of \cite[Theorem]{giacomellilasicacurves} (see also the proof of \eqref{pippo}) that the pointwise estimate \eqref{clelimit2} holds. Due to \eqref{eq-sff-w2}, \eqref{udw}, \eqref{mdw2} and \eqref{zdw}, it holds that
\begin{equation}
	\label{z_bv} \bz(t,\cdot)\in BV(I,\R^N)\quad \text{for a.\,e.\ } t\in]0,T[
\end{equation}
and
\begin{equation} \label{abscontzx}
	|\bz_x(t,\cdot)| \ll \Lb + |\bu_{0,x}| \quad \text{for a.\,e.\ } t\in]0,T[. \end{equation}
Finally, in view of \eqref{linfty}, Lemma \ref{lem-dist-comp} yields
\begin{equation}\label{eq-dist-comp}
C^{-1} \dist_\N(\bp_1,\bp_2)\le |\bp_1-\bp_2|\le \dist_\N(\bp_1,\bp_2) \quad\forall \bp_1,\bp_2\in \mbox{Range}(\bu).
\end{equation}

\subsection{Identifying $\bz$ away from the jump set}

In this subsection we identify the vector field $\bz$ outside of the jump set of the solution $\bu$ and prove that it satisfies the boundary condition \eqref{zbc}.

\begin{lemma}\label{lem:z-diffuse}
Let $\bnu$, $\bu$ and $\bz$ be those obtained in Section \ref{sec:basic}. Then \eqref{zeqn} holds, that is
\[
\bz(t)=\frac{\bu_x(t)}{|\bu_x(t)|} \qquad |\bu_x^d(t)|\text{-a.\,e.}\]
Moreover, \begin{equation}
  \label{lsc_diffuse}
\frac{|\bu_x(t)|}{|\bu_{0,x}|} = \frac{\bnu(t)}{|\bu_{0,x}|} \qquad |\bu^d_{0,x}|\text{-a.\,e.}
\end{equation}
\end{lemma}

\begin{proof}
Let $J$ be a subinterval of $I$ such that $|\bu_{0,x}|(\partial J)=0$, 
let $\varphi$ be a smooth function with support in $]0,T[\times J$, such that $0\le
\varphi \le 1$, and let
\[
\ubu^k(t):=\frac{1}{|J|}\int_J\bu^k(t,x)\dd x, \qquad \ubu(t):=\frac{1}{|J|}\int_J\bu(t,x)\dd x.
\]
With an integration by parts, we have
\begin{multline}
\label{tbp1}
\iint_{Q_T} \varphi|\bu^k_x| = \iint_{Q_T} \varphi \bu^k_x\cdot\bz^k \stackrel{(\ref{eq-sff})}= -\iint_{Q_T} \varphi \bu^k\cdot(\bu^k_t-\bmu^k)-  \iint_{Q_T} \varphi_x \bu^k \cdot\bz^k
\\=
 -\iint_{Q_T} \varphi \bu^k\cdot\bu^k_t + \int_0^T \ubu^k \cdot\int_I\varphi\bmu^k + \iint_{Q_T} \varphi \left(\bu^k -\ubu^k\right)\cdot\bmu^k -  \iint_{Q_T} \varphi_x \bu^k \cdot\bz^k
\\ =: I_0+I_1+I_2+I_3.
\end{multline}
The passage to the limit as $k\to\infty$ in $I_0$, $I_1$ and $I_3$ is straightforward in view of \eqref{udw}, \eqref{uds}, \eqref{zdw} and \eqref{mdw}. For $I_2$,
%
%
we note that
\[
\sup_J|\bu^k(t,\cdot) -\ubu^k(t)|\le \int_J |\bu^k_x(t,\cdot)| {\stackrel{\eqref{cle}}\le} \int_J |\bu^k_{0,x}|, \qquad t\in ]0,T[.
\]
Therefore, by strict convergence of $(\bu_0^k)$,
\[
|I_{2}| \stackrel{\eqref{Abound-m-l}}\le
%
%
A \left(\int_J |\bu^k_{0,x}|\right) \iint_{Q_T} \varphi |\bu^k_{0,x}| \to A R:= A |\bu_{0,x}|_\N(J) \iint_{Q_T} \varphi \dd |\bu_{0,x}|
\]
since $\partial J$ does not contain jump points of $\bu$.
%
%
On the left-hand side of \eqref{tbp1}, by Fatou's Lemma and Lemma \ref{phi_lsc},
\begin{multline} \label{awaylsc}
	\int_0^T\!\! \int_I \varphi \dd \bnu(t) \dd t {\stackrel{\eqref{weakconvergencederivative}}=}
	\liminf_{k \to \infty} \iint_{Q_T}\varphi |\bu^k_x| \dd \Lb^2 \\ \geq \int_0^T\!\! \liminf_{k \to \infty} \int_I \varphi |\bu^k_x(t)| \dd \Lb \dd t
	\geq \int_0^T\!\! \int_I \varphi \dd |\bu_x(t)|_{\N} \dd t.
\end{multline}
%
%
We pass to the limit as $k\to +\infty$ in \eqref{tbp1} using the previous convergences: 
\begin{multline}\label{salva}
\iint_{Q_T} \varphi \dd |\bu_x|_{\N} \le \iint_{Q_T}\varphi \dd \bnu  \le
-\iint_{Q_T} \varphi \bu \bu_t + \int_{Q_T}\varphi \ubu \cdot \dd\bmu -  \iint_{Q_T} \varphi_x \bu \cdot\bz +A\,R
%
%
\\
\stackrel{(\ref{eq-sff-w2})}=- \iint_{Q_T} \varphi \bu^*\cdot \dd \bz_x  - \iint_{Q_T} \varphi \left( \bu^* -\ubu\right)\cdot \dd \bmu- \iint_{Q_T} \varphi_x \bu \cdot\bz +A\,R.
\end{multline}
Note that the integrals on the right hand side make sense in view of  Lemma \ref{lem-salva}. In view of \eqref{mdw2}, arguing as in the estimate of $I_{2}$ above, we may absorb the second summand on
the r.\,h.\,s.\ into $R$. Therefore
\begin{multline*}
\iint_{Q_T} \varphi \dd|\bu_x|_{\N} \le \iint_{Q_T} \varphi \dd\bnu \le
-\iint_{Q_T} \varphi \bu^*\cdot \dd\bz_x -  \iint_{Q_T} \varphi_x \bu \cdot\bz + A\,R
\\
 \stackrel{\eqref{z_bv},\eqref{green}}=  \iint_{Q_T} \varphi \bz^* \cdot \dd\bu_x  +A\,R.
\end{multline*}
%
%
Choose now $\varphi(s,x)=\tau_l(s)\xi_m(x)$, with $\tau_l$ a sequence of mollifiers concentrating at point $t$ and $\xi_m$ converging to $\chi_J$. Since $|\bu_{0,x}|(\partial J)=0$, in view of
\eqref{mdw3} and  \eqref{clelimit2} we have $|\bu_x(t,\cdot)|(\partial J)=0$ and $\bnu(t)(\partial J)=0$. Therefore, passing to the limit as $l\to +\infty$ and  $m\to +\infty$ (in this order),
for a.\,e.\ $t \in ]0,T[$ 
we obtain
\[
|\bu_x(t)|_{\N}(J) \le \bnu(t)(J)\leq  \int_J \bz(t)^*\cdot d\bu_x(t)  + A\,(|\bu_{0,x}|_{\N}(J))^2.
\]
Finally, we let $J=J_\eps=]x_0-\eps,x_0+\eps[$, divide by $|\bu_{0,x}|(J_\eps)$ and pass to the limit as $\eps\to 0$. Provided that $x_0\notin J_{\bu_0}$, we have $|\bu_{0,x}|(J_\eps)\to 0$ (see
\cite[Proposition 3.92]{afp}). Therefore, by the Besicovitch derivation theorem \cite[Theorem 2.22]{afp}, we obtain
\[
\frac{|\bu_x(t)|}{|\bu_{0,x}|} \leq \frac{\bnu(t)}{|\bu_{0,x}|} \le \frac{\bz^*(t)\cdot\bu_x(t)}{|\bu_{0,x}|} \qquad |\bu_{0,x}^d|\text{-a.\,e.,}
\]
where the measure $\bz^*(t)\cdot\bu_x(t)$ is given by $\bz^*(t)\cdot\bu_x(t)(B):=\int_B \bz^*(t)\cdot \dd \bu_x(t)$ for any Borel set $B\subset I$.
{Since $|\bz|\le 1$}, we in fact have equality:
\[
\frac{|\bu_x(t)|}{|\bu_{0,x}|} =\frac{\bnu(t)}{|\bu_{0,x}|}= \frac{\bz^*(t)\cdot\bu_x(t)}{|\bu_{0,x}|} \qquad |\bu_{0,x}^d|\text{-a.\,e.}
\]
Finally, thanks to {\eqref{clelimit2}}, by the chain rule for Radon--Nikodym derivatives, we obtain
\[
\frac{|\bu_x(t)|}{|\bu_{0,x}|} = \frac{\bz^*(t)\cdot\bu_x(t)}{|\bu_{0,x}|}=\frac{\bz^*(t)\cdot\bu_x(t)}{|\bu_x(t)|}\frac{|\bu_x(t)|}{|\bu_{0,x}|} \qquad  |\bu_{0,x}^d|\text{-a.\,e.}
\]
which implies (using again {\eqref{clelimit2}}) that
\[
\frac{\bz^*(t)\cdot\bu_x(t)}{|\bu_x(t)|}=1 \qquad |\bu_x^d(t)|\text{-a.\,e.}
\]
This yields
\[
\bz^*(t)\cdot \bu_x^d(t)=(\bz^*(t)\cdot \bu_x^d(t))\res_{I\setminus J_{\bu(t)}}=(\bz^*(t)\cdot \bu_x(t))\res_{I\setminus J_{\bu(t)}}=|\bu_x^d(t)|,
\]
i.\,e.
\[\bz^*(t)=\frac{\bu_x(t)}{|\bu_x(t)|} \qquad |\bu_x^d(t)|\text{-a.\,e.}
\]
Since $\bz(t)\in BV(I)$ {and \eqref{abscontzx} holds}, $J_{\bz(t)}$ is a null set w.\,r.\,t.\ $|\bu(t)_x^d|$. We conclude that
\[ \bz(t)=\frac{\bu_x(t)}{|\bu_x(t)|}\qquad |\bu_x^d(t)|\text{-a.\,e.} \]
\end{proof}

\begin{lemma} \label{lem:bc}
	Let $\bz$ be the one obtained in Section \ref{sec:basic}. Then \eqref{zbc} holds, that is
	\[\bz = 0 \quad \text{on } \partial I \text{ for a.\,e.\ }t \in]0,T[.\]
\end{lemma}

\begin{proof}
	Let $h \in I$. By \eqref{NeumannZ}, \eqref{eq-sff} and  \eqref{Abound-m-l},
	\begin{equation} \label{bc_est}
		\int_0^T \sup_{[0,h]} |\bz^k|  \leq \int_0^T \!\! \int_0^h |\bz^k_x| \leq \int_0^T \!\! \int_0^h |\bu^k_t|+ A \int_0^h |\bu^k_{0,x}|.
	\end{equation}
	We pass to the limit as $k\to \infty$ in \eqref{bc_est} using \eqref{utw}, \eqref{u0d-1-cons} and lower semicontinuity of the l.\,h.\,s., obtaining
	\begin{equation} \label{bc_est_lim}
		\int_0^T \sup_{[0,h]} |\bz| \leq \int_0^T \!\! \int_0^h \overline{|\bu_t|}+ A\,|\bu_{0,x}|_{\N}(]0,h]).
	\end{equation}
	Since $0\notin J_{\bu_0}$, the r.\,h.\,s.\ of \eqref{bc_est_lim} tends to $0$ as $h \to 0^+$ and we deduce \eqref{zbc}.
\end{proof}

\subsection{Identifying $\bz$ on the jump set}

In this subsection we identify the vector field $\bz$ on the solution's jump set.
\begin{lemma}\label{lem:z-jump}
Let $\bu$, $\bz$ and $\bnu$ be those obtained in Section \ref{sec:basic}. Then \eqref{zeqnj} holds, that is
\begin{equation*}
\bz^\pm(t)= \pm \boldsymbol T_{\bu(t)}^\pm \quad \text{on } J_{\bu(t)}\quad \text{for a.\,e.\ } t\in [0,T].
\end{equation*}
Moreover,
\begin{equation}\label{strictconvergencejump}
	|\bu_x(t)|_{\N}\res J_{\bu(t)} =\bnu(t)\res J_{\bu(t)} \quad \text{for a.\,e.\ }t\in [0,T].
\end{equation}
\end{lemma}
\noindent
The rest of the section is devoted to its proof. Throughout the proof, we fix $t_0$ and $x_0$ such that
\begin{equation}\label{ass1}
\bu(t_0) \in BV(I), \quad x_0\in J_{\bu(t_0)}
\end{equation}
and we work with a subsequence $\bu^k$ (not relabeled) such that
\begin{equation}\label{sub-x}
\bu^k(t_0)\to \bu(t_0) \quad\text{a.\,e.\ in } I,
\end{equation}
whose existence follows from \eqref{uds}, since $\bu^k(t_0)\to \bu(t_0)$ strongly in $L^2(I)$. We also fix the notation $x_{\pm}^h:=x_0\pm h$, for $h>0$.

We will use \emph{Fermi normal coordinates}, whose notion we now recall. Given an open set $U \subset \N$ and a geodesic segment $\bgamma\colon ]-\frac{L(\bgamma)}{2}, \frac{L(\bgamma)}{2}[ \to U$
parametrized by the distance from its midpoint, we say that a diffeomorphism
 \[F\colon U \ni \bp \mapsto (p^1, \ldots, p^n) \in V \subset \mathbb R^n\]
 is a Fermi normal coordinate chart on $U$ flat along $\bgamma$, if
 \begin{equation}\label{fermi-prop}
 \begin{array}{rcl}
  \bullet && F(\bgamma(s)) = (s,0,\ldots,0) \text{ for } s\in ]-\frac{L(\bgamma)}{2}, \frac{L(\bgamma)}{2}[,
 \\ \bullet &&	g_{ij}|_{\bgamma} \equiv \delta_{ij}\text{ and }g_{ij,k}|_{\bgamma} \equiv 0\text{ for }i, j,k=1,\ldots, n,
 \\ \bullet && \text{in particular, } \Gamma_{ij}^k|_{\bgamma}=0\text{ for }i, j, k=1,\ldots, n.
 \end{array}
 \end{equation}
Such coordinate systems are known to exist locally, see e.\,g.\ \cite{manasse}. In order to deliver an optimal result, we need a quantitative estimate on the existence set in terms of $\radN$. A natural
domain of a Fermi coordinate chart flat along $\bgamma$ is a \emph{tube} around $\bgamma$ of radius $r>0$ \cite{GrayTubes}, defined by
  \begin{equation}\label{def-ugr}
  U(\bgamma, r) = \left\{\exp_{\bq} \bw \colon\, \bq \in \bgamma, \,\bw \in T_{\bq} \N,\, \bw \perp T_{\bq}\bgamma,\, |\bw|_{\N} < r\right\}.
  \end{equation}

 \begin{prop}\label{fermiprop}
 Let $\bgamma\colon ]\!-\radN, \radN[\to \N$ be a geodesic segment parametrized by the distance from its midpoint $\bp$. There exists a Fermi normal coordinate chart on $U(\bgamma, \radN)$ flat along
 $\bgamma$. Furthermore, if $\bw_1:=\bgamma(s_1)$, $\bw_2:=\bgamma(-s_1)$, for $0\leq s_1<\radN$ and $\bw\colon]a,b[\to \N$ is a curve joining $\bw_1$ and $\bw_2$
 which is not contained in $U(\bgamma, \radN)$, then $L(\bw)\ge 2\radN$. In particular, $B_{\N}(\bp,\radN)\subseteq U(\bgamma,\radN)$.
 \end{prop}
 \noindent The proof of Proposition \ref{fermiprop} is relegated to the appendix.

\medskip
Since $\bu_0$ is rad, by {\eqref{clelimit2}} we have
\begin{equation}
\label{jump-t0}
\tfrac12 \dist_\N (\bu^-(t_0, x_0), \bu^+(t_0, x_0))< \radN.
\end{equation}
Thus, there exists a unique minimizing geodesic segment joining $\bu^-(t_0, x_0)$ and $\bu^+(t_0, x_0)$. We let $\overline{\bgamma}_{\bu^-(t_0, x_0)}^{\bu^+(t_0, x_0)}\colon ]-\radN, \radN[ \to \N$ be
the geodesic extension of that segment, pa\-ra\-me\-trized by the distance from its midpoint, which we denote $\bp_0$. {If $\radN=+\infty$ we may w.l.o.g. choose $\bp_0$ to be the midpoint of the segment.} We will work in Fermi normal coordinate chart $\bp \mapsto (p^1, \ldots, p^n)$ on
$U(\overline{\bgamma}_{\bu^-(t_0, x_0)}^{\bu^+(t_0, x_0)}, \radN)$ flat along ${\overline\bgamma}_{\bu^-(t_0, x_0)}^{\bu^+(t_0, x_0)}$. We now give closeness conditions, which will also be used later,
guaranteeing in particular that the image of $\bu^k$ on a time-space rectangle around $(t_0, x_0)$ is contained in the domain of such a chart for large $k$.
\begin{lemma}\label{localcontainmentnewer}
Assume \eqref{ass1} and \eqref{sub-x}. There exist $0 < r_0 < \radN$, $h_0, \tau_0 > 0$ and $k_0 \in \mathbb N$ such that {for any $h\in (0,h_0)$ such that  $x^{h}_\pm\notin J_{\bu_0}$ and $\bu^k(t_0,x^{h}_\pm)\to \bu(t_0,x^{h}_\pm)$ as $k\to +\infty$, it holds:} 
\begin{equation} \label{contain1}
\frac12 \int_{x_-^{h}}^{x_+^{h}} |\bu^k_x(t,\cdot)|< r_0 \quad \text{for all } t >0, \  k > k_0,
\end{equation}
\begin{equation}\label{contain2}
 \bu^k(t, x^{h}_\pm)\in B_{\N}\left(\bp_0, r_0\right)\quad\text{for all } t \in]t_0, t_0 + \tau_0[, \ k>k_0,
\end{equation}
\begin{equation}\label{contain3}
\bu^k(t,x) \in U\left(\overline{\bgamma}_{\bu^-(t_0, x_0)}^{\bu^+(t_0, x_0)}, \radN\right)
\quad\text{for all } (t,x)\in Q_{\tau_0,h}:=]t_0, t_0 + \tau_0[\times ]x^{h}_-, x^{h}_+[,\ k> k_0.
\end{equation}
\end{lemma}

\begin{proof}
Let
\begin{equation} \label{jump0}
 |\bu_{0,x}|_{\N}(\{x_0\}) = \dist_\N (\bu_0^-(x_0), \bu_0^+(x_0)) :=2\bar r_0.
\end{equation}
Since $\bu_0 \in BV_{rad}(I, \N)$, $\bar r_0<\radN$. Let 
{$r_0\in (\bar r_0,\radN)$}
%
%
and let $\eps>0$ so small that $\bar r_0+{\eps}<r_0$  {(the choices of $r_0$ and $\eps$ will be made more precise later)}. By continuity of measure $|\bu_{0,x}|_{\N}$, for all {$h<h_0$} sufficiently small we have
\begin{equation} \label{jump0cont}
	|\bu_{0,x}|_{\N}\!\left( ]x_-^{h}, x_0[  \right)< \frac{\eps}{2}, \quad |\bu_{0,x}|_{\N}\!\left( ]x_0, x_+^{h}[  \right)< \frac{\eps}{2}.
\end{equation}
We choose {any} $h{<h_0}$ 
as in the statement.

\smallskip

{
In order to prove \eqref{contain1}, we notice that
$ |\bu_{0,x}|_{\N}\!\left( ]x_-^{h}, x_+^{h}[ \right)< 2 \bar r_0 +\eps.$
 By the strict convergence \eqref{u0d-1}, $k_0>0$  exists such that
\begin{equation}\label{jump1contk}
 |\bu^k_{0,x}|_{\N}\!\left( ]x_-^{h}, x_+^{h}[ \right)< 2 \bar r_0 +2\eps \quad\mbox{for all $k>k_0$.}
\end{equation}
Therefore, by \eqref{cle} we deduce that
\begin{equation}
\label{pippo1}
 |\bu^k_{x}(t)|_{\N}\!\left( ]x_-^{h}, x_+^{h}[ \right)\le  |\bu^k_{0,x}|_{\N}\!\left( ]x_-^{h}, x_+^{h}[ \right)\le 2\bar r_0+2\eps <2r_0 \quad\mbox{for all $t>0$, $k>k_0$,}
\end{equation}
which coincides with \eqref{contain1}.

\smallskip

In order to prove \eqref{contain2},
}
we notice that \eqref{jump0}-\eqref{jump0cont} imply, by \eqref{clelimit2}, that
\begin{equation}\label{bux_small}
|\bu_x(t)|_{\N}\!\left( ]x_-^{h}, x_0[ \right)< \tfrac{\eps}{2},\quad |\bu_x(t)|_{\N}\!\{x_0\} {\le} 2\bar r_0,\quad |\bu_x(t)|_{\N}\!\left( ]x_0, x_+^{h}[ \right)< \tfrac{\eps}{2}
\end{equation}
for all $t>0$. {In particular, \eqref{jump-t0} holds. Hence $\bp_0$ is well defined and we may choose $r_0$ such that
\begin{equation}\label{def-r0}
\radN>r_0>\bar r_1 := \max\left\{\bar r_0, \dist_\N (\bu^-(t_0, x_0), \bp_0),\dist_\N (\bp_0, \bu^+(t_0, x_0)) \right\}.
\end{equation}
}
By definition of $\bp_0$, {we have
\begin{eqnarray*}
\dist_\N(\bu(t_0, x_\pm^{h}), \bp_0) & \le & \dist_\N(\bu(t_0, x_\pm^{h}), \bu^\pm(t_0, x_0)) + \dist_\N (\bu^\pm(t_0, x_0), \bp_0)
\\ & \stackrel{\eqref{bux_small},\eqref{def-r0}} \le & \frac \eps 2 + \bar r_1.
\end{eqnarray*}
Let $\eps>0$ be so small that $\bar r_1+\frac32 \eps<r_0$. Then} \eqref{bux_small} implies that
\[
\bu(t_0, x_\pm^{h}) \in B_{\N}\left(\bp_0,  r_0{-\eps}\right).
\]
Since $x_\pm^{h}$ are such that $\bu^k(t_0, x_\pm^{h}) \to \bu(t_0, x_\pm^{h})$, by equicontinuity of $t \mapsto \bu^k(t, x_\pm^{h})$ (Lemma \ref{contxslice}) we obtain
\begin{equation} \label{contxh0}
	\dist_{\N}(\bu^k(t_0 + \tau, x_\pm^{h}), \bu(t_0, x_\pm^{h})) < {\tfrac \eps 2}
\end{equation}
for $\tau$ small enough and $k$ large enough. We deduce \eqref{contain2}.

\smallskip

In order to prove \eqref{contain3}, assume by contradiction that the image of $[x^{h}_-,x^{h}_+]\ni x\mapsto \bu^k(t, x)$ is not contained in $U(\overline{\bgamma}_{\bu^-(t_0, x_0)}^{\bu^+(t_0,
x_0)}, \radN)$ for some $t \in [t_0, t_0 + \tau_0]$. We prolong this curve with the geodesic segments joining $\bu^k(t, x^{h}_\pm)$ and  $\bu^\pm(t_0, x_0)$. By Proposition \ref{fermiprop}, the
resulting curve $\tbgamma$ has $L(\tbgamma)\ge 2\radN$. On the other hand, by {\eqref{pippo1}} and \eqref{contxh0},
\[
L(\tbgamma)< 2\bar r_0 + 3 \eps <2r_0<2\radN,
\]
which is a contradiction.
\end{proof}

We now take $h \in ]0, h_0[$ such that $|\bu_{0,x}|_{\N}(\{x^h_\pm\}) = 0$,  $\tau\in ]0,\tau_0[$, and $k > k_0$. By Lemma \ref{localcontainmentnewer}, in particular by \eqref{contain3}, we can
work in Fermi coordinates. In coordinates, the equation \eqref{ctvflowZ} satisfied by $\bu^k$ takes the form (cf. \eqref{eq-sff})
\begin{equation}\label{eq-coord}
u^{k,l}_t = z^{k,l}_x + \Gamma^l_{ij}(\bu^k) u^{k,i}_x z^{k,j}.
\end{equation}
Let $\varphi^h \in C_c(I)$ be such that $\supp \varphi^h = [x^h_-, x^h_+]$, $\varphi^h(x_0) = 1$, $\varphi^h$ is affine on $[x^h_-, x_0]$ and on $[x_0, x^h_+]$. We let $Q_\tau:=]0,\tau[\times I$ and we calculate
 \begin{multline}
 \iint_{Q_\tau}\varphi^h |\bu^k_x| = \iint_{Q_\tau}\varphi^h g_{ij}(\bu^k) u^{k, i}_x z^{k, j}
 \\ = - \iint_{Q_\tau}\varphi^h_x g_{ij}(\bu^k) u^{k, i} z^{k, j} - \iint_{Q_\tau}\varphi^h g_{ij,l}(\bu^k) u^{k,l}_x u^{k, i} z^{k, j} - \iint_{Q_\tau}\varphi^h g_{ij}(\bu^k) u^{k, i}
 z^{k, j}_x
 \\ =
 - \iint_{Q_\tau}\varphi^h_x g_{ij}(\bu^k) u^{k, i} z^{k, j} - \iint_{Q_\tau}\varphi^h g_{ij,l}(\bu^k) u^{k,l}_x u^{k, i} z^{k, j}
\\ {+ \iint_{Q_\tau}\varphi^h g_{ij}(\bu^k) u^{k, i} \Gamma_{il}^j(\bu^k) u_x^{k, i} z^{k,l} - \iint_{Q_\tau}\varphi^h g_{ij}(\bu^k) u^{k, i}u^{k,j}_t}
\\ =: I^k_1+I^k_2+I^k_3+I^k_4.
\label{base4}
 \end{multline}
 As $k \to \infty$, we have
 \begin{equation}
 \label{jumpconv1} I^k_1 \to -\iint_{Q_\tau}\varphi^h_x g_{ij}(\bu) u^{i} z^{j},\qquad I^k_4 \to - \iint_{Q_\tau}\varphi^h g_{ij}(\bu) u^{i} u^{j}_t.
 \end{equation}
We estimate
\begin{multline} \label{fermiest}
	\left|I_2^k\right|{+|I_3^k|} {\stackrel{\eqref{inclZ},\eqref{linfty}}\leq} C \left(\int_{t_0}^{t_0+\tau}\!\!\!\! \sup_{x \in [x^h_-, x^h_+]} {\left|D
g\left(\bu^k(t,x)\right)\right|+\left|\bGamma\left(\bu^k(t,x)\right)\right|}\right) \left(\esssup_{t\in [t_0, t_0 + \tau]} \int_{x^h_-}^{x^h_+} \left|\bu^{k}_x(t,x)\right|\right) \\
{\stackrel{(\ref{clelimit2}),(\ref{fermi-prop})}\leq} C \left(\int_{t_0}^{t_0+\tau}\!\!\!\! \sup_{x \in [x^h_-, x^h_+]} \dist_\N\left(\bu^k(t,x), \gamma_{\bu^-(t_0, x_0)}^{\bu^+(t_0, x_0)}\right)\right)
\left(\int_I \left|\bu_{0,x}^{k}\right|\right) .
\end{multline}
We note that the first integrand on the right hand side is a one-sided Hausdorff distance between two curves. In general, given curves $\bgamma_1, \bgamma_2$, we will denote
\begin{equation}\label{def-dist}
\vec{\dd}(\bgamma_1, \bgamma_2) = \sup_{\bp \in \bgamma_1} \dist_\N(\bp, \bgamma_2).
\end{equation}
We recall that for any three curves $\bgamma_1, \bgamma_2, \bgamma_3$ there holds
\begin{equation}\label{dist-tri}
 \vec{\dd}(\bgamma_1, \bgamma_3) \leq \vec{\dd}(\bgamma_1, \bgamma_2) + \vec{\dd}(\bgamma_2, \bgamma_3).
\end{equation}
Thus we can estimate the first integrand on the r.\,h.\,s.\ of \eqref{fermiest} as
 \begin{equation}\label{bigtriangest--}
 \vec{\dd} \left(\left.\bu^k(t,\cdot)\right|_{[x^h_-, x^h_+]}, \bgamma_{\bu^-(t_0,x_0)}^{\bu^+(t_0,x_0)}\right) \leq \vec{\dd} \left(\left.\bu^k(t,\cdot)\right|_{[x^h_-, x^h_+]},
 \bgamma_{\bu^k(t,x^h_-)}^{\bu^k(t,x^h_+)}\right) +\vec{\dd} \left(\bgamma_{\bu^k(t,x^h_-)}^{\bu^k(t,x^h_+)}, \bgamma_{\bu^-(t_0,x_0)}^{\bu^+(t_0,x_0)}\right).
 \end{equation}

In order to estimate the second summand on the right hand side of \eqref{bigtriangest--}, we will make use of continuous dependence of not-too-long geodesics on their endpoints. The not-too-long
justification of this fact is provided in the Appendix.
\begin{lemma}\label{lem-what}
For any $r<\radN$ there exists a constant $C=C(r)>0$ such that
 \begin{equation}\label{what}
 \vec{\dd}\left(\bgamma_{\bp_1}^{\bq_1}, \bgamma_{\bp_2}^{\bq_2}\right) \leq C\max\left\{\dist_\N(\bp_1, \bp_2), \dist_\N(\bq_1, \bq_2)\right\}
 \end{equation}
for any $\bp_1, \bp_2, \bq_1, \bq_2 \in B_{\N}(\bp_0,r)$, $\bp_0 \in \N$.
\end{lemma}

The next lemma provides a version of 1-dimensional Poincar\'e--Sobolev inequality for solutions to the 1-harmonic map equation, formally
\begin{equation}\label{1hm_formal}
	\pi_{\bw} \left(\frac{\bw_x}{|\bw_x|}\right)_x = \bff \quad \text{on } I.
\end{equation}
As in the parabolic case, this concept needs a rigorous definition. We will only consider here the "regular" setting corresponding to Lemma \ref{localest}.
\begin{defn}
	Let $\bff \in L^1(I,\N)$. We say that $\bw \in Lip(I, \N)$ is a regular solution to \eqref{1hm_formal} if there exists $\bv \in W^{1,1}(I,\R^n)$ such that
	\begin{equation}\label{1hm_v}
		\bv(x) \in T_{\bw(x)} \N, \quad |\bv(x)| \leq 1, \quad \bv(x) \cdot \bw_x(x) = |\bw_x(x)| \quad \text{for a.\,e.\ } x \in I
	\end{equation}
	and
	\begin{equation} \label{1hm}
		\pi_{\bw} \bv_x = \bff \quad \text{a.\,e.\ on } I.
	\end{equation}
\end{defn}
Due to degeneracy of the 1-harmonic equation, the distance between $\bw$ and the arc-length parametrized geodesic (which solves \eqref{1hm_formal} with $\bff \equiv 0$) cannot be controlled pointwise by
$\bff$. Instead, we obtain a bound on the distance of the image of $\bw$ from the geodesic. It will be used to estimate the first term on the right-hand side of \eqref{bigtriangest--}.
\begin{lemma} \label{sliceest}
Let $\bff \in L^1(I,\N)$ and let $\bw \in Lip(I, \N)$ be a regular solution to \eqref{1hm_formal}
such that
	\begin{equation} \label{w_jump_r0}
		\int_{x^h_-}^{x^h_+} |\bw_x| \leq 2 r_0
	\end{equation}
	with given $h>0$, $r_0 \in ]0,\radN[$. There exists $C>0$ depending on $\N$, $r_0$ and $\|\bw\|_\infty$
such that
	\begin{equation}\label{sliceineq}
		\sup_{x \in [x_-^h,x_+^h]}\dist_\N\left(\bw(x), \bgamma_{\bw(x_-^h)}^{\bw(x_+^h)}\right) \leq C \int_{x_-^h}^{x_+^h} |\bff|.
	\end{equation}
\end{lemma}

\begin{proof}
We temporarily denote $a = x^h_-$, $b = x^h_+$. We have
\begin{equation}\label{redef-w}
\bw(x)= \bw(a) +\int_a^x\bw_x \stackrel{(\ref{1hm_v})}=\bw(a) +\int_a^x|\bw_x|\bv.
\end{equation}
Let $\bdelta$ be the constant speed geodesic emanating from $\bw(a)$ with initial speed $\bv(a)$:
\begin{equation}\label{def-delta}
\qquad \frac{\dd^2 \bdelta}{\dd s^2} = - \Gamma_{ij}(\bdelta)\frac{\dd \bdelta_i}{\dd s}\frac{\dd \bdelta_j}{\dd s}, \quad \bdelta(0) = \bw(a), \quad\bdelta'(0)=\bv(a).
\end{equation}
The geodesic $\bdelta$ is uniquely determined and satisfies $|\frac{\dd \bdelta}{\dd s}|=|\bv(a)|$ as long as it is defined. We now re-parametrize $\bdelta$ in such a way that its pace is close to that
of $\bw$, i.\,e., its scalar speed is $|\bv(a)\bw_x|$. Let
$$
\varphi(x)= \int_a^x |\bw_x(\xi)|\dd \xi, \quad \tbgamma(x)=\bdelta(\varphi(x)), \quad \widetilde \bv(x)=\frac{\dd\bdelta}{\dd s}(\varphi(x)).
$$
Then $\tbgamma_x =\frac{\dd\bdelta}{\dd s} \varphi_x = \widetilde\bv |\bw_x|$, hence
\begin{equation}\label{def-tbg}
\tbgamma (x)= \bw(a)+ \int_a^x \widetilde\bv |\bw_x|
\end{equation}
and
\begin{equation}\label{covtwexp}
\widetilde \bv_x = \frac{\dd^2 \bdelta}{\dd s^2} \varphi' \stackrel{(\ref{def-delta})} = - \Gamma_{ij}(\tbgamma)\widetilde \bv_i \widetilde \bv_j |\bw_x|.
\end{equation}
In addition, as long as $\tbgamma$ is defined,
\begin{equation}\label{constant-speed}
|\widetilde\bv(x)|=|\bv(a)|\le 1,\quad\mbox{hence}\quad \|\tbgamma\|_\infty \stackrel{(\ref{def-tbg})}\le \|\bw\|_\infty + 2r_0 \le C,
\end{equation}
where along this proof generic constants $C$ depend on $\N$, $r_0$, and $\|\bw\|_\infty$.

We claim that $C>0$ exists such that
\begin{equation}\label{distutbgamma}
\sup_{x\in [a,b]} \dist_\N(\bw(x), \tbgamma(x)) \le  C \int_a^b |\bff|.
\end{equation}
To this aim, we note that
%
in the same coordinates \eqref{1hm} reads (cf. \eqref{eq-coord})
 \begin{equation} \label{maineqGamma}
  \bv_x = - \bGamma_{ij}(\bw) |\bw_x| v^i v^j + \bff.
 \end{equation}
By \eqref{redef-w} and \eqref{def-tbg}, we have
 \begin{multline} \label{estGammadiff=}
  \left|\bGamma_{ij}(\bw(x)) - \bGamma_{ij}(\tbgamma(x))\right| = \left|\bGamma_{ij}\left(\bw(a) + \int_a^x |\bw_x| \bv\right) - \bGamma_{ij}\left(\bw(a) + \int_a^x |\bw_x| \tbv\right)\right|.
 \end{multline}
Therefore, letting
\[
M(x) = \sup_{[a,x]}|\bv -\tbv|,
\]
it follows from (\ref{estGammadiff=}) that positive constants $C_0$ and $C_1$ depending on $\N$, $r_0$ and $\|\bw\|_\infty$ exist such that
\begin{equation} \label{estGammadiff}
  \left|\bGamma_{ij}(\bw(x)) - \bGamma_{ij}(\tbgamma(x))\right| \stackrel{\eqref{constant-speed}}\leq C_0\int_a^x|\bw_x||\bv - \tbv| \le C_0\, M(x) \int_a^b |\bw_x| \stackrel{(\ref{w_jump_r0})}\le C_1\,M(x).
 \end{equation}
Recalling that $|\bv|\le 1$, it follows that, possibly increasing the value of $C_1$,
 \begin{multline} \label{Mxest}
 0 \leq M_x \leq |(\bv - \tbv)_x| \stackrel{(\ref{covtwexp}),(\ref{maineqGamma})}\leq |\bw_x|\Big(\left|\bGamma_{ij}(\bw) - \bGamma_{ij}(\tbgamma)\right||v^i v^j| + |\bGamma_{ij}(\tbgamma)| |v^i
 v^j - \tv^i \tv^j|\Big)+ |\bff|\\ \stackrel{(\ref{estGammadiff}),\eqref{constant-speed}}\leq  C_1 |\bw_x| M + |\bff|.
 \end{multline}
%
%
 Hence, denoting $\widetilde C(x) = \exp\left( C_1 \int_a^x|\bw_x|\right)$ and recalling that $\tbv(a)=\bv(a)$, by Gronwall's inequality
\[
|\bv(x) - \tbv(x)|\leq M(x) \leq \widetilde C(x) \int_a^x (\widetilde C(x'))^{-1} |\bff(x')|\dd x'\le \widetilde C(b) \int_a^b |\bff|
\]
Consequently, recalling Lemma \ref{lem-dist-comp},
%
%
\begin{multline*}
\sup_{x\in [a,b]} \dist_\N(\bw(x), \tbgamma(x)) \leq C\sup_{x\in [a,b]} |\bw(x)-\tbgamma(x)| \stackrel{(\ref{redef-w}),(\ref{def-tbg})}\le C \int_a^b |\bw_x||\bv-\tbv|\\ \le  C \widetilde
C(b)\int_a^b |\bw_x| \int_a^b |\bff|
 \stackrel{(\ref{w_jump_r0})}\le  C_2\int_a^b |\bff|,
\end{multline*}
with $C_2$ depending on $\N$, $r_0$, and $\|\bw\|_\infty$.
%
%

We are now ready to complete the proof. By triangle inequality,
\begin{equation}\label{passoq-}
\sup_{x\in [a,b]}\dist_\N\left(\bw(x), \bgamma_{\bw(a)}^{\bw(b)}\right) \le \sup_{x\in [a,b]} \dist_\N \left(\bw(x), \tbgamma(x)\right) +  \vec{\dd}\left(\tbgamma, \bgamma_{\bw(a)}^{\bw(b)} \right).
\end{equation}
We now wish to apply Lemma \ref{lem-what} to the second summand on the right-hand side of \eqref{passoq-}. Letting $\bp\in\N$ be the geodesic midpoint between $\bw(a)$ and $\bw(b)$, we note that
$\bw(a),\bw(b)\in B_{\N}(\bp,r_0)$ by \eqref{w_jump_r0}. If we assume in addition that
\begin{equation}\label{passos}
C_2\int_a^b |\bff|<\frac{\radN-r_0}{2},
\end{equation}
then \eqref{distutbgamma} implies that $\tbgamma(b)\in B_{\N}\left(\bp,r_0+\frac{\radN-r_0}{2}\right)$, so Lemma \ref{lem-what} can be applied. Thus \eqref{passoq-} turns into
\begin{multline}
\sup_{x\in [a,b]}\dist_\N\left(\bw(x), \bgamma_{\bw(a)}^{\bw(b)}\right) \le \sup_{x\in [a,b]}\dist_\N \left(\bw(x), \tbgamma(x)\right) +  C(r_0) \dist_\N\left(\tbgamma(b), \bw(b)\right) \\ \le
(1+C(r_0)) \sup_{x\in [a,b]}\dist_\N \left(\bw(x), \tbgamma(x)\right) \stackrel{\eqref{distutbgamma}}\le  (1+C(r_0))C_2\int_a^b|\bff| \quad\text{if \eqref{passos} holds.}
\label{passoq}
\end{multline}
On the other hand, if \eqref{passos} does not hold, we conclude:
\begin{align*}
	\sup_{x\in [a,b]}\dist_\N\left(\bw(x), \bgamma_{\bw(a)}^{\bw(b)}\right) & \le \sup_{x\in [a,b]}\min\{\dist_\N\left(\bw(x), \bw(a)\right), \dist_\N\left(\bw(x),\bw(b)\right)\}
\\ & \stackrel{(\ref{w_jump_r0})}< r_0 \le \frac{2 C_2 r_0}{\radN - r_0}\int_a^b |\bff|.
\end{align*}
\end{proof}

We can now conclude the proof of Lemma \ref{lem:z-jump}.
We use Lemma \ref{sliceest} to estimate the first term on the r.\,h.\,s.\ of \eqref{bigtriangest--}. Indeed, note that $\bu^k(t)$ is a regular solution to the $1$-harmonic map equation
\eqref{1hm_formal} with $\bff = \bu^k_t(t)$ for a.\,e.\ $t>0$, and that assumption \eqref{w_jump_r0} is satisfied due to  \eqref{contain1}.
On the other hand, the second term on the r.\,h.\,s.\ of \eqref{bigtriangest--} is estimated using Lemma \ref{lem-what}, whose assumptions are satisfied in view of \eqref{jump-t0} and \eqref{contain2}.
All together, this leads to
\begin{multline}\label{bigtriangest}
  \vec{\dd} \left(\left.\bu^k(t,\cdot)\right|_{[x^h_-, x^h_+]}, \bgamma_{\bu^-(t_0,x_0)}^{\bu^+(t_0,x_0)}\right)
 \leq C \int_{x^h_-}^{x^h_+} |\bu^k_t|
 \\
 + C\max\left\{\dist_\N(\bu^k(t,x^h_-), \bu^-(t_0,x_0)), \dist_\N(\bu^k(t,x^h_+), \bu^+(t_0,x_0))\right\}.
 \end{multline}

We integrate \eqref{bigtriangest} over the time interval $[t_0, t_0+ \tau]$, $\tau < \tau_0$ and plug the resulting estimate into \eqref{fermiest}. Since $\bu^k$ converge almost everywhere, we can restrict
to $h$ such that $\bu^k(t,x^h_\pm)$ converge for almost every $t>0$ and pass to the limit $k \to \infty$ obtaining
 \begin{multline} \label{jumpconvest}
\limsup_{k \to \infty}|I_2^k|+|I_3^k| \leq  C \left|\bu_{0,x}\right|_{\N}(I) \left( \iint_{Q_{\tau,h}} \overline{|\bu_t|}\right. \\
 \left.\phantom{\iint_{Q_{\tau,h}}}+\int_{t_0}^{t_0+\tau}\left(\dist_\N(\bu(t,x^h_-), \bu^-(t_0,x_0)) + \dist_\N(\bu(t,x^h_+), \bu^+(t_0,x_0))\right)\right),
\end{multline}
where $\overline{|\bu_t|}$ is the weak limit of $|\bu^\delta_t|$ in $L^2(]0,T[\times I)$ (see \eqref{utw}). By Fatou's Lemma and Lemma \ref{phi_lsc},
\begin{multline} \label{jumplsc}
\int_{t_0}^{t_0 + \tau}\!\! \int_I \varphi^h \dd \bnu(t) \dd t {\stackrel{\eqref{weakconvergencederivative}}=}
\liminf_{k \to \infty} \iint_{Q_\tau}\varphi^h |\bu^k_x| \dd \Lb^2 \\ \geq \int_{t_0}^{t_0 + \tau}\!\! \liminf_{k \to \infty} \int_I \varphi^h |\bu^k_x(t)| \dd \Lb \dd t
\geq \int_{t_0}^{t_0 + \tau}\!\! \int_I \varphi^h \dd|\bu_x(t)|_{\N} \dd t.
\end{multline}
We pass to the limit $k \to \infty$ in \eqref{base4} applying \eqref{jumpconv1}, \eqref{jumpconvest}, and \eqref{jumplsc}. We divide the result by $\tau$ and pass to the limit $\tau \to 0^+$ using Lemma
\ref{contxslice}. For a.\,e.\ $t_0\in [0,T]$ functions $\bu_t(t)$ and $\overline{|\bu_t|}(t)$ are integrable and we get
\begin{multline} \label{jumpsliceest}
\int_{x^h_-}^{x^h_+}\varphi^h \dd |\bu_x|_{\N}  \leq  \int_{x^h_-}^{x^h_+}\varphi^h \dd\bnu
\leq - \int_{x^h_-}^{x^h_+} \varphi_x^h g_{ij}(\bu) u^{i} z^{j} - \int_{x^h_-}^{x^h_+}\varphi^h g_{ij}(\bu) u^{i} u^{j}_t
\\ + C \int_I \dd \left|\bu_{0,x}\right|_{\N} \, \left( \int_{x^h_-}^{x^h_+} \overline{|\bu_t|} +  \dist_\N(\bu(t_0,x^h_-), \bu^-(t_0,x_0)) + \dist_\N(\bu(t_0,x^h_+), \bu^+(t_0,x_0))\right).
\end{multline}
Taking into account the definition of one-sided limits $u^\pm(t_0,x_0)$, the last two terms on the r.\,h.\,s.\ of \eqref{jumpsliceest} converge to $0$ as $h \to 0^+$. Recalling our choice of $\varphi$,
passing with $h\to 0^+$ in \eqref{jumpsliceest} leads to
\begin{equation} \label{jumppointest}
|\bu_x(t_0,\cdot)|_{\N}(\{x_0\}) \leq \bnu(t_0)(\{x_0\})\leq -\left[g_{ij}(\bu) u^{i} z^{j}\right]^-(t_0,x_0) +\left[g_{ij}(\bu) u^{i} z^{j}\right]^+(t_0,x_0).
\end{equation}
By the choice of Fermi coordinates, $g_{ij}=\delta_{ij}$ and $(\bu^\pm)^i= \pm\frac12 \delta_{i1}\dist_\N(\bu^-(t_0, x_0),\bu^+(t_0, x_0))$. Therefore,
\begin{equation*}
\dist_\N(\bu^-(t_0, x_0),\bu^+(t_0, x_0)) \leq \tfrac{1}{2}\dist_\N(\bu^-(t_0, x_0),\bu^+(t_0, x_0)) \left(\left(z^1\right)^-(t_0,x_0) + \left(z^1\right)^+(t_0,x_0)\right).
\end{equation*}
This inequality, together with the pointwise estimate $\left|\bz^\pm(t_0, x_0)\right| \leq 1$, implies that $\bz^+(t_0, x_0) = \bz^-(t_0,x_0) = (1,0,\ldots,0)$ in Fermi coordinates, thus \eqref{zeqnj}
holds. As a byproduct, recalling \eqref{jumppointest}, we obtain \eqref{strictconvergencejump}, concluding the proof of Lemma \ref{lem:z-jump}. \qed

\subsection{Identifying $\bu_t$ }

In this subsection we identify the products of the limit passage $k \to \infty$ in equation \eqref{ctvflowZ} satisfied by $\bu^k$, i.\,e.
\begin{equation*}
	\bu^k_t = \pi_{\bu^k} \bz^k_x .
\end{equation*}

\begin{lemma}
Let $\bu$ and $\bz$ be those obtained in Section \ref{sec:basic}. Then \eqref{maineqn} and \eqref{cantor-m}, i.\,e.
\begin{equation*}
\bu_t(t)=\pi_{\bu(t)}^*\bz_x^a(t) \quad\Lb^1\text{-a.\,e.\ in }I,
\end{equation*}
\begin{equation*}
\pi_{\bu(t)}^*\bz_x^c(t) = 0 \quad\text{as measures on } I,
\end{equation*}
hold for a.\,e.\ $t\in]0,T[$\,.
\end{lemma}

\begin{proof}
Let $J$ be a subinterval of $I$ such that $|\bu_{0,x}|(\partial J) = 0$, let $\varphi$ be smooth and nonnegative, with support in $]0,T[\times J$, and let
\[
{\underline{\pi}^k}(t):=\frac{1}{|J|}\int_J\pi_{\bu^k(t)}\dd \Lb^1, \qquad {\underline{\pi}}(t):=\frac{1}{|J|}\int_J\pi_{\bu(t)}\dd \Lb^1.
\]
We have
\begin{multline*}
\iint_{Q_T} \varphi\bu^k_t = \iint_{Q_T} \varphi \pi_{\bu^k}\bz_x^k  = \iint_{Q_T} \varphi {\underline{\pi}^k}\bz^k_x + \iint_{Q_T} \varphi \left(\pi_{\bu^k}-{\underline{\pi}^k}\right)\bz^k_x
\\  \stackrel{\eqref{eq-sff}}=\iint_{Q_T} \varphi {\underline{\pi}^k}\bz^k_x + \iint_{Q_T} \varphi \left(\pi_{\bu^k}-{\underline{\pi}^k}\right)\bu^k_t -  \iint_{Q_T} \varphi
\left(\pi_{\bu^k}-{\underline{\pi}^k}\right) \A_{\bu^k}(\bz^k,\bz^k)|\bu^k_x|
\\ =: I_1+I_2+R_k.
\end{multline*}
The passage to the limit as $k\to \infty$ on the left-hand side, as well as in $I_1$ and $I_2$, is straightforward in view of \eqref{uds}, \eqref{udw}, and \eqref{zdw}. For $R_k$, we note that
\[
|\pi_{\bu^k(t)}-{\underline{\pi}^k}(t)|\le \int_J |(\pi_{\bu^k(t)})_x|\le C \int_J |\bu^k_x(t)|\stackrel{\eqref{cle}}{\le} C \int_J |\bu^k_{0,x}|.
\]
Therefore
\begin{multline*}
|R_k|  \stackrel{\eqref{Abound}}\le C\left(\int_J |\bu^k_{0,x}|\right) \iint_{Q_T} \varphi |\bu^k_x| \stackrel{\eqref{cle}}\le C \left(\int_J |\bu^k_{0,x}|\right) \iint_{Q_T} \varphi |\bu^k_{0,x}|
\\  \stackrel{\eqref{u0d-1}}\to  C  |\bu_{0,x}|_{\N}(J) \int_0^T\!\!\int_J \varphi \dd |\bu_{0,x}|_{\N}\dd t =: R \quad\text{as } k\to \infty,
\end{multline*}
hence
\begin{multline} \label{Rest}
R \ge \left|\iint_{Q_T} \varphi\bu_t - \int_0^T\!\! \int_J \varphi \underline{\pi}(t) \dd \bz_x(t) \dd t - \iint_{Q_T} \varphi \left(\pi_{\bu}-{\underline{\pi}}\right)\bu_t \right|
\\ = \left|\iint_Q \varphi\bu_t -\int_0^T\!\!\int_J \varphi \left((\pi_{\bu(t)})^* -\underline{\pi}(t)\right) \dd \bmu(t)\dd t -  \int_0^T\!\!\int_I \varphi (\pi_{\bu(t)})^* \dd \bz_x(t) \dd t \right|.
\end{multline}
We note that, for a.\,e.\ $t \in ]0,T[$, $(\pi_{\bu(t)})^*$ is the representative of the function $\pi_{\bu(t)} \in BV(I, \R^{N\times N})$ that is defined pointwise as an average of left and right
limits:
\[(\pi_{\bu(t)})^* = \tfrac{1}{2}\left((\pi_{\bu(t)})^- + (\pi_{\bu(t)})^+\right). \]
Moreover, we clarify that
\[ \left( \int_J \varphi (\pi_{\bu(t)})^* \dd \bmu(t) \right)^i = \int_J \varphi (\pi^{ij}_{\bu(t)})^* \dd \mu^j(t) \quad \text{for } i=1,\ldots, N.\]
By Lemma \ref{lem-salva}, these are measurable functions of $t$. The same applies to $\int_I \varphi (\pi_{\bu(t)})^* \dd \bz_x(t)$.

Just as on the level of approximation, we estimate
\[
|(\pi_{\bu(t)})^*-\underline{\pi}(t)|\le |(\pi_{\bu(t)})_x|(J)\le C  |\bu_x(t)|_{\N}(J)\stackrel{\eqref{clelimit2}}{\le} C |\bu_{0,x}|_{\N}(J).
\]
Taking into account \eqref{mdw2}, the middle integral on the r.\,h.\,s.\ of \eqref{Rest} may be absorbed into $R$ and we obtain
\[
\left|\int_0^T\!\!\left(\int_I \varphi \bu_t(t) \dd \Lb^1 - \int_I \varphi \left(\pi_{\bu(t)}\right)^* \dd \bz_x(t) \right)\dd t\right|\le R.
\]
Choose now $\varphi(t,x)=\tau_l(t)\xi_m(x)$, with $\tau_l$ a sequence of mollifiers concentrating at point $t$ and $\xi_m$ converging to $\chi_J$. Passing to the limit as $l\to +\infty$ and $m\to
+\infty$ (in this order), for a.\,e.\ $t \in ]0,T[$ and for any set $J$ such that $|\bu_{0,x}|_{\N}(\partial J)=0$ (whence, in view of \eqref{cle}, $|\bu_x(t)|_{\N}(\partial J)=0$) we obtain
\begin{equation}\label{Jh}
\left|\int_J \bu_t(t) \dd \Lb^1 - \int_J \left(\pi_{\bu(t)}\right)^* \dd \bz_x(t) \right|\le C(|\bu_{0,x}|_{\N}(J))^2.
\end{equation}
From now on, we fix $t$ and disregard dependence on it for notational convenience. For $x_0\notin J_{\bu_0}$, we let $J=J_h=]x_0-h,x_0+h[$. Dividing \eqref{Jh} by $|J_h|$ and passing to the limit $h\to
0^+$ along a subsequence such that $|\bu_{0,x}|_{\N}(\partial J_h)=0$, we obtain
\begin{equation}\label{ac-passo0}
\bu_t=\frac{\bu_t \Lb^1}{\Lb^1} = \frac{\pi_{\bu}^*\bz_x}{\Lb^1} \quad \Lb^1\text{-a.\,e.\ in } I,
\end{equation}
We claim that
\begin{equation}\label{claim-RN}
\frac{\pi_{\bu}^* \bz_x}{\Lb^1} =\pi_{\bu}^*\frac{\bz_x}{\Lb^1}=\pi_{\bu}\bz_x^a \quad \Lb^1\text{-a.\,e.\ in }I.
\end{equation}
Indeed, for $\Lb^1$-a.\,e.\ $x_0\in I$ it holds that
\begin{multline*}
\left|\frac{1}{|J_h|}\int_{J_h} \pi_{\bu}^* \dd \bz_x - \pi_{\bu}^*(x_0) \frac{1}{|J_h|}\int_{J_h}\dd\bz_x\right|  \le  \frac{1}{|J_h|}\int_{J_h} \left|\pi_{\bu}^*- \pi_{\bu}^*(x_0) \right| \dd |\bz_x|
\\  \le  C |\bu_{0,x}|_{\N}(J_h) \frac{|\bz_x|(J_h)}{|J_h|}\ \to\ 0\cdot |\bz_x^a|(x_0) = 0 \quad\text{as } h\to 0^+.
\end{multline*}
Combining \eqref{ac-passo0} and \eqref{claim-RN}, we obtain \eqref{maineqn}.

Now, observe that we can rewrite \eqref{maineqn} in the form
\begin{equation}\label{ac-passo1}
\bu_t\Lb^1 =\left(\pi_{\bu}\bz_x^a\right)\Lb^1 = \pi_{\bu}\left(\bz_x^a\Lb^1\right)
\end{equation}
(the latter equality is an easy observation from the definition). On the other hand, recall that $|\bz_x|\ll \Lb^1+|\bu_{0,x}|$ (cf.\ \eqref{abscontzx}). Thus, dividing \eqref{Jh} by
$(\Lb^1+|\bu_{0,x}|)(J)$ and passing to the limit, we obtain
\[
\frac{\bu_t \Lb^1}{\Lb^1+|\bu_{0,x}|} = \frac{\pi_{\bu}\bz_x}{\Lb^1+|\bu_{0,x}|} \quad (\Lb^1+|\bu_{0,x}|)\text{-a.\,e.\ in } I\setminus J_{\bu_0}.
\]
Thus
\[
\bu_t \Lb^1= \left(\pi_{\bu}\bz_x\right)\res_{I\setminus J_{\bu_0}} = \pi_{\bu}\left(\bz_x^a\Lb^1 + \bz_x^c\right) = \pi_{\bu}\left(\bz_x^a\Lb^1\right) +\pi_{\bu}\bz_x^c.
\]
Subtracting this equality from \eqref{ac-passo1} we obtain \eqref{cantor-m}.
\end{proof}

 \begin{cor}\label{lem:secondfundamental}
  Let $(\bu,\bz)$ be a strong solution to \eqref{smootheqn}-\eqref{smoothbc}. Then,
  \begin{equation}\label{secondfund-eqn}\bu_t=\bz_x^a+\mathcal A_{\bu}(\bz,\bu_x^a).\end{equation}
\end{cor}
\begin{proof}
  The proof is identical to part of  \cite[Lemma 2]{giacomellilasicamoll}. The only difference is that we have to use \eqref{maineqn} together with Lemma \ref{lem:u_x_tangent} and the fact that, thanks
  to \eqref{green}, $$N_{\bu}^{j}\cdot \bz_x^a=(N_{\bu}^{j})_x^{a}\cdot\bz,$$ where $(N^j)_{j=1,\ldots,N-n}$ is a local orthonormal frame around $\bu(t,x)$.
\end{proof}

\subsection{Strict convergence and its consequences}

In this section, we conclude the proof of Theorem \ref{existence} by proving 
\eqref{zjumpset} and the monotonicity property \eqref{clelimit}. These facts are consequences of the following

\begin{lemma} \label{lem:strict}
Let $\bu$ be the strong solution to (\ref{smootheqn}-\ref{smoothbc}) constructed in this section and $\bu^k$ the sequence of approximating regular solutions. Then
\begin{equation} \label{weakstar_refine}
	|\bu^k_x|\stackrel{*}\rightharpoonup|\bu_x|_{\N}\quad \text{in } L^\infty_{w^*}(0,T; M(I)).
\end{equation}
Moreover, for almost every $s\in [0,T]$ there exists a subsequence $\bu^{k_l}(s)$ such that
\begin{equation}\label{strict_ae_s}
	\lim_{l \to \infty} \int_I |\bu^{k_l}_x(s)|  =  |\bu_x(s)|_{\N}(I).
\end{equation}
\end{lemma}
\begin{proof}
Let $\bnu$ be the measure defined in \eqref{weakconvergencederivative}. By \eqref{strictconvergencejump} and \eqref{lsc_diffuse}, we have for a.\,e.\ $t\in [0,T]$
\begin{equation*}
	\frac{|\bu_x(t)|_{\N}}{|\bu_{0,x}|_{\N}}= \frac{\bnu(t)}{|\bu_{0,x}|_{\N}} \qquad |\bu_{0,x}|_{\N}\text{--a.\,e.}
\end{equation*}
Then, since both measures are absolutely continuous with respect to $|(\bu_0)_x|_{\N}$ (see \eqref{clelimit2} and \eqref{mdw3}),
\begin{equation*}
	\bnu(t) = |\bu_x(t)|_{\N}.
\end{equation*}
By \eqref{weakconvergencederivative}, we have obtained \eqref{weakstar_refine}.

For $h>0$ sufficiently small, let $\varphi^h \in C_0(I)$ be such that $\varphi^h(x) = 1$ if $\dist(x,\partial I) \geq h$ and $\varphi^h$ is affine on $[0, h]$ and on
$[1-h, 1]$. By \eqref{cle}, \eqref{weakstar_refine}, \eqref{clelimit} and strict convergence of $\bu^k_0$ to $\bu_0$,
\begin{multline} \label{strict_triangle}
	\limsup_{k \to \infty}\left|\iint_{Q_T} |\bu^k_x|  - \int_0^T |\bu_x(t)|_{\N}(I)\dd t\right|  \leq \limsup_{k \to \infty}\iint_{Q_T} (1-\varphi^h)|\bu^k_x| \dd \Lb^2 \\+ \limsup_{k
\to \infty} \left|\iint_{Q_T} \varphi^h |\bu^k_x| \dd \Lb^2 - \int_0^T\!\! \int_I \varphi^h \dd |\bu_x(t)|_{\N}\dd t\right| + \int_0^T \!\!\int_I (1 - \varphi^h) \dd |\bu_x(t)|_{\N} \dd t \\ \leq  2 T
\int_I (1-\varphi^h) \dd |(\bu_0)_x|_{\N}.
\end{multline}
Since the right hand side of \eqref{strict_triangle} converges to $0$ as $h \to 0^+$, we obtain
\begin{equation} \label{strict_int}
	\iint_{Q_T} |\bu^k_x| \dd \Lb^2 \to \int_0^T |\bu_x(t)|_{\N}(I)\dd t .
\end{equation}
On the other hand, by Lemma \ref{GM_lsc}, for a.\,e.\ $t \in [0,T[$ we have
\begin{equation}\label{strict_ae_ineq}
	\liminf_{k \to \infty} \int_I |\bu^k_x(t)| \dd \Lb^1 \geq  |\bu_x(t)|_{\N}(I).
\end{equation}
From \eqref{strict_int} it follows that the inequality in \eqref{strict_ae_ineq} is in fact an equality for a.\,e.\ $t\in [0,T[$. In particular, for a.\,e.\ $s\in [0,T[$ there exists a subsequence $\bu^{k_l}$ such that \eqref{strict_ae_s} holds.
\end{proof}

\begin{proof}[Proof of \eqref{zjumpset} and \eqref{clelimit}] 
%
Applying \cite[Proposition 1.80]{afp}, it follows from \eqref{strict_ae_s} and lower semi-continuity that for a.\,e.\ $s$ there exists $k_\ell$ such that $\bu_x^{k_\ell}(s,\cdot) \stackrel{*}{\rightharpoonup} \bu_x(s,\cdot)$. Then, again by \eqref{strict_ae_s}, $\bu^{k_\ell}(s,\cdot) \to \bu(s,\cdot)$ strictly in $BV(I)$. Hence from \eqref{Abound-m-l} and \eqref{mdw} we obtain $|\bmu(s,\cdot)|\le A|\bu_x(s,\cdot)|$ for a.\,e.\ $s$. Therefore, it follows from \eqref{eq-sff-w2} that  $|\bz_x(s,\cdot)|\ll |\bu_x(s,\cdot)|$ for a.\,e.\ $s$, which implies \eqref{zjumpset}.

By \eqref{cle2},
\begin{equation} \label{clest_approx} 	
	|\bu^k_x(t)|\leq |\bu^k_x(s)|\quad \text{for } 0\leq s\leq t <T.
\end{equation} 	
Let us take a subsequence $\bu^{k_l}$ such that \eqref{strict_ae_s} holds. By \cite[Proposition 1.80]{afp} we deduce that $|\bu^{k_l}_x(s)| \stackrel{*}{\rightharpoonup}  |\bu_x(s)|_{\N}$. Then, arguing as in the proof of
\cite[Theorem]{giacomellilasicacurves}, it follows from \eqref{clest_approx} that
\begin{equation} \label{clest}
	|\bu_x(t)|_{\N} \leq |\bu_x(s)|_{\N} \quad \text{for a.\,e.\ $s,t \in ]0,T[$ with $s\le t$.}
\end{equation}
\end{proof}

\section{Variational (in)equality and uniqueness}\label{sec:unique}


In this section we prove Theorem \ref{NPC-result}. We will first show a more general "variational equality", which implies the variational inequality \eqref{var_ineq}. It will also be used repeatedly in the next section.

\begin{lemma} \label{var_eq}
	Let $\bu$ be a strong solution to (\ref{smootheqn}, \ref{smoothbc}). Let $0\leq T_1 < T_2 \leq T$, let $a, b \in \overline{I}$, $a < b$. Let $\bv \in H^1_{loc}(]T_1, T_2[\times [a,b], \N) $ be such that for a.\,e.\ $t \in ]T_1, T_2[$
	\begin{equation} \label{var_eq_cond}
		\dist_{\N}\left(\bv, \bu\right) < 2 \radN \quad \text{$|\bu_x^d|$-a.\,e.\ in } ]a,b[,
	\end{equation}
	\begin{equation}\label{var_eq_cond_j}
	\vec{\dd}(\gamma_{\bu^-}^{\bu^+},\bv)
 < 2 \radN  \quad  \text{$|\bu_x^j|$-a.\,e.\ in } ]a,b[,
	\end{equation}
where $\vec{\dd}$ is defined in \eqref{def-dist}. Then, for a.\,e.\ $t \in ]T_1, T_2[$ there holds
\begin{align}
\nonumber \dd H_{\bv, \bu}^j(t) &:= (\bz^-\cdot\log_{\bu^-} (\bv) -\bz^+\cdot\log_{\bu^+} (\bv))\dd \Hd^0
\\ \label{hess-j}
&= \left(\int_{0}^{\dist_\N(\bu^-,\bu^+)} \Hess \tfrac{1}{2} r^2_{\bv}(\bgamma_{\bu^-}^{\bu^+}(s))(\dot{\bgamma}_{\bu^-}^{\bu^+}(s), \dot{\bgamma}_{\bu^-}^{\bu^+}(s))\dd s\right) \dd \Hd^0 ,
\end{align}
which together with formulae
\begin{equation}
\label{hess-d}
H_{\bv, \bu}(t) = H_{\bv, \bu}^d(t) + H_{\bv, \bu}^j(t), \quad \dd H_{\bv, \bu}^d(t) :=\Hess \tfrac{1}{2} r_{\bv}^2(\bu)(\bz,\bz) \dd |\bu_x^d|_{\N}
\end{equation}
defines a measure $H_{\bv, \bu}(t) \in M(I)$, and we have
\begin{multline}\label{ambrosio_eq}
	\frac{1}{2}\frac{\dd}{\dd t}\int_a^b \dist_{\N}^2 (\bu,\bv) + H_{\bv, \bu}(t)(]a,b[) + \int_a^b \bv_t \cdot \log_{\bv} (\bu)\\ = \int_a^b \bz \cdot D_{\bv} \log_{\bu} (\bv)\, \bv_x - \bz \cdot \log_{\bv} (\bu)\bigg|_{a^+}^{b^-}.
\end{multline}

Moreover, if there exists $B < 2 \radN$ and a Borel set $U \subseteq ]a,b[$ such that
\begin{equation} \label{hess_comp_B_cond}
	\dist_{\N}\left(\bv, \bu\right) \leq B\quad \text{$|\bu_x^d|$-a.\,e.\ in } U, \quad
	\dist_{\N}\left(\bv, \gamma_{\bu^-}^{\bu^+}\right) \leq B  \quad  \text{$|\bu_x^j|$-a.\,e.\ in } U,
\end{equation}
then
\begin{equation} \label{hess_comp_B}
	H_{\bv, \bu}(U) \geq h_{\N}(B)|\bu_x|_{\N}(U).
\end{equation}
\end{lemma} 	

\begin{proof}
	For $t \in ]T_1, T_2[$ we calculate, implicitly using \eqref{var_eq_cond}-\eqref{var_eq_cond_j},
	\begin{equation*}
		\frac{1}{2}\frac{\dd}{\dd t}\int_a^b \dist_{\mathcal N}^2(\bu,\bv) = -\int_a^b \bv_t\cdot \log_{\bv}(\bu) -\int_a^b \bu_t\cdot \log_{\bu}(\bv),
	\end{equation*}
\begin{equation*}
	-\int_a^b \bu_t\cdot \log_{\bu}(\bv)
 \stackrel{\eqref{maineqn}}{=} -\int_a^b \bz_x^a\cdot\log_{\bu}(\bv)
		 {\stackrel{\eqref{cantor-m}}=}  -\int_{]a,b[}\log_{\bu}(\bv)^*\cdot \dd \bz_x +\int_{]a,b[}\log_{\bu}(\bv)^*\cdot\dd \bz_x^j,
\end{equation*}
\begin{equation*}	
	-\int_{]a,b[}\log_{\bu}(\bv)^*\cdot \dd \bz_x
		{\stackrel{\eqref{green}} =} - \bz \cdot \log_{\bv} (\bu)\bigg|_{a^+}^{b^-} + \int_{]a,b[}\bz^* \cdot \dd \log_{\bu}(\bv)_x,
\end{equation*}
\begin{equation*}
	\int_{]a,b[}\bz^* \cdot \dd \log_{\bu}(\bv)_x		
	 =  \int_a^b \bz\cdot D_{\bv}(\log_{\bu}(\bv))\bv_x +\int_{]a,b[} \bz\cdot D_{\bu}(\log_{\bu}(\bv)) \dd \bu_x^d+\int_{]a,b[} \bz^* \cdot \dd (\log_{\bu}(\bv))_x^j.
\end{equation*}
Summing up the above equations,
\begin{multline} \label{var_eq_mid}
	\frac{1}{2}\frac{\dd}{\dd t}\int_a^b \dist_{\mathcal N}^2(\bu,\bv) = -\int_a^b \bv_t\cdot \log_{\bv}(\bu) - \bz \cdot \log_{\bv} (\bu)\bigg|_{a^+}^{b^-} + \int_a^b \bz\cdot D_{\bv}(\log_{\bu}(\bv))\bv_x \\+\int_{]a,b[} \bz\cdot D_{\bu}(\log_{\bu}(\bv)) \dd \bu_x^d +\int_{]a,b[} \bz^* \cdot \dd (\log_{\bu}(\bv))_x^j +\int_{]a,b[}\log_{\bu}(\bv)^*\cdot\dd \bz_x^j .
\end{multline}
It remains to take care of the last three terms on the r.\,h.\,s.\ of \eqref{var_eq_mid}. We have
\[
\int_{]a,b[} \bz\cdot D_{\bu}(\log_{\bu}(\bv)) \dd \bu_x^d \stackrel{\eqref{zeqn}} = \int_{]a,b[} \bz\cdot D_{\bu}(\log_{\bu}(\bv)) \bz \dd |\bu_x^d| \stackrel{\eqref{char-hess}}= - H_{\bv, \bu}^d (]a,b[)
\]
and
\begin{multline*}
	\int_{]a,b[}\bz^* \cdot \dd (\log_{\bu}(\bv))_x^j + (\log_{\bu}(\bv))^*\cdot  \dd \bz_x^j
	\\ {\stackrel{\eqref{zjumpset}}=}  \tfrac{1}{2} \int_{]a,b[\cap J_{\bu}}(\bz^++\bz^-)\cdot (\log_{\bu^+}(\bv)-\log_{\bu^-}(\bv))+(\bz^+-\bz^-)\cdot(\log_{\bu^+}(\bv)+\log_{\bu^-}(\bv))\dd \Hd^0
	\\  = \int_{]a,b[\cap J_{\bu}}(\bz^+\cdot\log_{\bu^+}(\bv)-\bz^-\cdot\log_{\bu^-}(\bv))\dd \Hd^0 .
\end{multline*}
At this point, we have proved \eqref{ambrosio_eq}.
In order to prove the characterization \eqref{hess-j}, we
recall that at any point $x \in J_{\bu}$, $\bz^\pm$  coincides with the vector
$T_{\bu}^\pm$ tangent at $\bu^\pm$ to the geodesic segment $\bgamma_{\bu^-}^{\bu^+}$. Thus, using the property $\nabla_{\dot{\bgamma}}\dot{\bgamma} =0$ of geodesics and \eqref{var_eq_cond_j}, we can write
\begin{equation*}
	\bz^+\cdot\log_{\bu^+}(\bv) -  \bz^-\cdot\log_{\bu^-}(\bv) = \int \frac{\dd}{\dd s} \left(\dot{\bgamma}_{\bu^-}^{\bu^+} \cdot \log_{\bgamma_{\bu^-}^{\bu^+}} \bv\right)  = - \int \Hess \tfrac{1}{2} r^2_{\bv}(\bgamma_{\bu^-}^{\bu^+})(\dot{\bgamma}_{\bu^-}^{\bu^+}, \dot{\bgamma}_{\bu^-}^{\bu^+}),
\end{equation*}
whence \eqref{hess-j}. Then \eqref{hess_comp_B} is a simple consequence of \eqref{hess-j}, \eqref{hess-d}, and  Lemma \ref{hct}.
\end{proof}

\begin{proof}[Proof of Theorem \ref{NPC-result}]
Let $]a,b[ = I$ and $\bv \in C^1(I, \N)$. By our assumptions, $\radN=\infty$ and restrictions \eqref{var_eq_cond}, \eqref{var_eq_cond_j} in Lemma \ref{var_eq} are trivially satisfied. In this case, the variational equality \eqref{ambrosio_eq} reduces to
\[\frac{1}{2}\frac{\dd}{\dd t}\int_I \dist_{\N}^2 (\bu,\bv) + H_{\bv, \bu}(t)(I) = \int_I \bz \cdot D_{\bv} \log_{\bu} (\bv)\, \bv_x .\]
By \eqref{hess_comp_B} and Lemma \ref{hct}, we have
\[H_{\bv, \bu}(t)(I) \geq |\bu_x|_{\N}(I). \]
On the other hand, by \cite[bottom of p.\,540]{karcher},
\[
 D_{\bv}(\log_{\bu}(\bv))\bv_x=J'(0),
\]
 where $J(s)$ is the Jacobi field along $s\mapsto \exp_{\bu}(s\cdot \log_{\bu}(\bv))$ determined by $J(0)=0$, $J(1)=\bv_x$. Therefore, since $K_{\mathcal N}\leq 0$,
 \[
 |J'(0)|\leq |J(1)|=|\bv_x|
 \]
(see for instance \cite[(*) on p.\,545]{eberlein}).
Altogether, we deduce that the variational inequality \eqref{var_ineq} holds true for all $\bv\in C^1(I,\mathcal N)$. By strict approximation, we get that \eqref{var_ineq} holds true for all $\bv\in BV(I,\mathcal N)$. 
\end{proof}

\section{Asymptotic behavior}\label{sec:asymp}

In this section we prove Theorem \ref{thm:asymp}. Given $\bu_0\in BV_{rad}(I,\N)$, let $T>0$ and let $\bu$ a strong solution to (\ref{smootheqn},\ref{smoothbc}) in $[0,T[$ with initial datum $\bu_0$ satisfying the energy inequality \eqref{energylimit} and the pointwise estimate \eqref{clelimit}.

\begin{lemma}\label{lem:asymp_case1}
	Let $t_0 \in [0, T[$ and $B\in ]0,\radN[$. Suppose $x_2 \in I$ is such that
	\begin{equation}\label{asymp_lem_case1}
	0 < |\bu_x(t_0)|_{\N}(]0, x_2[) \leq B < \radN.
	\end{equation}
	 Then for every $\eps > 0$ there exists $\tau_* \geq 0$ depending on $B$, $x_2$ and $\bu_0$ but neither on $\bu$, $t_0$ nor on $T$ such that we have
	\[|\bu_x(t)|_{\N}(]0, x_2[) < \eps \quad \text{if } \ t_0 + \tau_* \leq t < T.\]
\end{lemma}

\begin{proof}
If
	\[|\bu_x(t_0)|_{\N}(]0,x_2[) \leq |(\bu_0)_x|_{\N}(]0,x_2[) < \eps,\]
	we can take $\tau_* = 0$. Otherwise, take the smallest $x_0 \in [0, x_2[$ such that
	\begin{equation}\label{eps_interv}
		|\bu_x(t_0)|_{\N}(]x_0,x_2[) \leq |(\bu_0)_x|_{\N}(]x_0,x_2[) \leq \tfrac{\eps}{2}.
	\end{equation}
   Denoting $J = ]x_0, x_2[$, by \eqref{energylimit},
   \begin{equation}\label{defC1} \frac{1}{|J|} \int_J\int_{t_0}^T \bu_t^2 \leq \frac{1}{|J|} \int_0^T \int_I \bu_t^2 \leq \frac{1}{|J|} TV_{\N}(\bu_0)=:\!C_1^2
   \end{equation}
(note that $x_0$, whence $|J|$, depends only on $\bu_0$).	Therefore, we can choose $x_1 \in J \setminus J_{\bu(t_0)}$ so that $\bu_1 \!:=\bu(\cdot, x_1) \in H^1(t_0,T)$ with
	\begin{equation} \label{C1ut}
		\int_{t_0}^T \bu_{1,t}^2 \leq C_1^2.
	\end{equation}
By \eqref{asymp_lem_case1} and \eqref{clelimit}, we can apply Lemma \ref{var_eq} with $\bv = \bu_1$, obtaining
	\begin{equation} \label{dist_deriv}
		\frac{1}{2}\frac{\dd}{\dd t}\int_0^{x_1} \dist_\N^2(\bu,\bu_1) = - \bu_{1,t}\cdot\int_0^{x_1}  \log_{\bu}(\bu_1)- H_{\bu_1, \bu}(]0, x_1[)  - \bz(x_1)\cdot \underbrace{\log_{\bu_1}(\bu_1)}_{=0}
	\end{equation}
and, using \eqref{hess_comp_B},
	\begin{equation} \label{ghj}
	\frac{1}{2}\frac{\dd}{\dd t}\int_0^{x_1} \dist_\N^2(\bu,\bu_1) \leq  - \bu_{1,t}\cdot\int_0^{x_1}  \log_{\bu}(\bu_1) - h_{\N}(B)|\bu_x|_{\N}(]0, x_1[).
\end{equation}
For $0< \tau < T- t_0$ we estimate using H\" older's inequality, \eqref{C1ut} and \eqref{asymp_lem_case1}
\[
\int_{t_0}^{t_0 + \tau} \!\!\! \bu_{1,t} \cdot  \int_0^{x_1}\log_{\bu}(\bu_1) \leq \int_{t_0}^{t_0 + \tau} \!\!\! |\bu_{1,t}| \sup_{[t_0, t_0 + \tau]}  \int_0^{x_1}\dist_\N(\bu,\bu_1)  \leq C_1 \sqrt{\tau}\,
|I|\, B =:\! C \sqrt{\tau}.
\]
Recall that $h_\N(B)> 0$ since $B<\radN$ (cf. Lemma \ref{hct}). Thus, integrating \eqref{ghj} over $[t_0, t_0 + \tau]$ and recalling \eqref{clelimit} yields
\begin{equation}\label{distancedissipation}
	0 \leq \left. \frac{1}{2}\int_0^{x_1} \dist_\N^2(\bu,\bu_1) \right|_{t_0 + \tau} \!\! \leq C \tau^{1/2}
	-h_{\N}(B) \tau |\bu_x(t_0 + \tau)|_{\N}(]0,x_1[)  + \frac{1}{2}\left. \int_0^{x_1} \dist_\N^2(\bu,\bu_1)\right|_{t_0}.
\end{equation}
The last term can easily be estimated in terms of only $\bu_0$:
\[\frac{1}{2}\left. \int_0^{x_1} \dist_\N^2(\bu,\bu_1)\right|_{t_0}\leq \frac{1}{2} |I|\, TV_I^{\N}(\bu_0)^2. \]
Thus, in view of \eqref{distancedissipation}, $\tau_*>0$ exists as desired such that
\[
 |\bu_x(t)|_{\N}(]0,x_1[) < \tfrac{\eps}{2} \quad \text{if } t_0+\tau_* \leq t < T
\]
and we conclude by recalling \eqref{eps_interv}.
\end{proof}

\begin{lemma} \label{lem:sphere_est_asymp2}
	Suppose that $R>0$, $0 < m < M < \pi R$ and $d, a^+, a^- \in ]0,M[$ are side lengths of a geodesic triangle in a sphere in $\R^3$ of radius $R$. Assume moreover that $d, a^- \geq m$. Let $\alpha_s^-,
\alpha_s^+$ denote the angles opposite to, respectively, $a^+, a^-$. For every $S_0 > 0$ there exists $\delta_0 = \delta_0(S_0,m,M)> 0$ such that if $a^+ \leq \delta_0$, then
	\begin{equation} \label{sphere_est_asymp2_eq}
		a^+\cos \alpha_s^+ + a^-\cos\alpha_s^- \geq d - S_0.
	\end{equation}
\end{lemma}
\begin{proof}
First of all, we observe that in the planar case; i.e in a triangle in $\R^N$, one has \begin{equation}\label{planar_case}
a^+\cos \alpha_s^+ + a^-\cos\alpha_s^- =d,\end{equation}
 without any restriction on the lengths $a^+,a^-$ and $d$.

By the haversine formula,
	\begin{equation*}
		\hav \tfrac{a^+}{R} = \hav\tfrac{a^- - d}{R} + \sin \tfrac{a^-}{R} \sin \tfrac{d}{R} \hav \alpha_s^- \ge \sin \tfrac{a^-}{R} \sin \tfrac{d}{R} \hav \alpha_s^-,
	\end{equation*}
	we obtain
	\begin{equation*}
		\hav \alpha_s^- \leq \frac{\hav \tfrac{a^+}{R}}{\sin \tfrac{a^-}{R} \sin \tfrac{d}{R}}, \quad\mbox{hence}\quad
		\cos \alpha_s^- \geq 1 - \frac{2 \sin^2 \tfrac{a^+}{2R}}{\sin \tfrac{a^-}{R} \sin \tfrac{d}{R}}.
	\end{equation*}
	Therefore, by the triangle inequality,
	\begin{equation*}
				a^+\cos \alpha_s^+ + a^-\cos \alpha_s^- \geq \left(1 - \frac{2 \sin^2 \tfrac{a^+}{2R}}{\sin \tfrac{a^-}{R} \sin \tfrac{d}{R}}\right) (d - a^+) - a^+ .
	\end{equation*}
Since $0<m\le a^-,d<M<\pi R$, for any $S_0>0$ we may choose $\delta_0 > 0$ suitably small (depending on $S_0$, $m$ and $M$)
such that \eqref{sphere_est_asymp2_eq} holds for all $a^+<\delta_0$.
\end{proof}

Note that since $\bu_0 \in BV_{rad}(I, \N)$, if $\radN<+\infty$, there exists a number $M > \radN$ such that
\begin{equation}\label{defM}
	\max_{x \in J_{\bu_0}} |(\bu_0)_x|_{\N}(\{x\}) < M < 2 \radN.
\end{equation}
\begin{lemma}\label{lem:asymp_case2} Let $\N$ be such that $\radN<+\infty$.
	Let $t_0 \in ]0, T[$. Suppose that $x_0 \in J_{\bu(t_0)}$. For every $S >0 $ there exist $\delta >0$ and $\tau_* \geq 0$ depending on $\bu_0$ and $x_0$ but not on $\bu$ or $t_0$ such that if
    \begin{equation} \label{asymp_lem_jump_ass}
    	|\bu_x(t_0)|_{\N}(]0, x_0[) < \delta
    \end{equation}
    then
    \begin{equation}
    	|\bu_x(t)|_{\N}(]0, x_0]) < S \quad \text{if }t_0 + \tau_* \leq t < T.
    \end{equation}
\end{lemma}
\begin{proof}
In case $K_\N>0$, we will use Lemma \ref{lem:sphere_est_asymp2} with $S_0=S/4$,$m=S/4$ and $R = K_{\N}^{-\frac{1}{2}}$. Let $\delta_0$ as in Lemma \ref{lem:sphere_est_asymp2}. We take
\begin{equation}\label{def-delta_2}
0<\delta <\left\{\begin{array}{cc}\min\left\{\tfrac12 S ,- \frac{S}{4 h_{\N}(M)},\tfrac12 (M-|(\bu_0)_x|_{\N}(\{x_0\})),\delta_0\right\} & {\rm if \ }K_\N>0 \\
\tfrac12 (M-|(\bu_0)_x|_{\N}(\{x_0\})) & {\rm if \ }K_\N\leq 0\end{array}\right..
\end{equation}
%
%
We take the largest $x_2 >x_0$ such that
\begin{equation}\label{sdrum}
|\bu_x(t_0)|_{\N}(J) \leq |(\bu_0)_x|_{\N}(J) \leq \tfrac{1}{2}\delta, \quad J:= ]x_0, x_2[.
\end{equation}
As in the proof of Lemma \ref{lem:asymp_case1}, we can find $x_1 \in J\setminus J_{\bu(t_0)}$ such that $\bu_1 \!:=\bu(\cdot, x_1) \in H^1(t_0,T)$ with
\begin{equation} \label{C1ut_case2}
	\int_{t_0}^T \bu_{1,t}^2 \leq C_1^2,
\end{equation}
where $C_1>0$ is defined in \eqref{defC1}. As in the proof of Lemma \ref{lem:asymp_case1}, we apply Lemma \ref{var_eq} for $t\ge t_0$ in $]0,x_1[$with $\bv = \bu_1$, obtaining \eqref{dist_deriv}.

We estimate $H_{\bu_1, \bu}$ using \eqref{hess_comp_B} as follows{:}
\begin{equation} \label{est_asymp_jump_close}
	-H_{\bu_1, \bu}(]x_0,x_1[) \leq - h_{\N}\left(|\bu_x|_{\N}(]x_0,
x_1[)\right)|\bu_x|_{\N}(]x_0, x_1[) \stackrel{(\ref{sdrum}),(\ref{def-delta_2})}\leq 0,
\end{equation}
\begin{equation} \label{est_asymp_jump_far}
		-H_{\bu_1, \bu}(]0,x_0[) \leq - h_{\N}(M)|\bu_x|_{\N}(]0, x_0[) \stackrel{(\ref{asymp_lem_jump_ass})}< - h_{\N}(M)\delta. 
\end{equation}
It remains to estimate $H_{\bu_1, \bu}(\{x_0\})$. We rewrite
\begin{multline}\label{estimate_sphere}
	H_{\bu_1, \bu}(\{x_0\}) =\bz^-\cdot\log_{\bu^-}(\bu_1) -\bz^+\cdot\log_{\bu^+}(\bu_1)
	\\   = \underbrace{\dist_{\mathcal N}(\bu^+,\bu_1)}_{=:a^+}\cos\underbrace{(\log_{\bu^+}(\bu_1),\log_{\bu^+}(\bu^-))}_{=:\alpha^+} +\underbrace{\dist_{\mathcal
N}(\bu^-,\bu_1)}_{=:a^-}\cos\underbrace{(\log_{\bu^-}(\bu_1),
		\log_{\bu^-}(\bu^+))}_{=:\alpha^-}
\end{multline}
and we note that $a^+\le\frac\delta 2$ in view of \eqref{sdrum}.  We wish to bound the right hand side of \eqref{estimate_sphere} from below by $d\!:=\dist_{\mathcal N}(\bu^-,\bu^+)$. To this end, we will apply the Alexandrov comparison theorem (Lemma \ref{act}) to
the geodesic triangle $(\bu^-, \bu^+, \bu_1)$. By \eqref{asymp_lem_jump_ass}, we have
		\begin{multline*}
			d + a^+ + a^- = \dist_{\N}(\bu^-, \bu^+) + \dist_{\N}(\bu^+, \bu^1) + \dist_{\N}(\bu^1, \bu^-) \\ \leq 2(|\bu_x|_{\N}(\{x_0\}) + |\bu_x|_{\N}(]x_0,x_1[) ) 
\stackrel{(\ref{def-delta_2}),(\ref{sdrum})}< 2(M-2\delta) +\delta < 2M < 4 \radN.
		\end{multline*}
Therefore, in the sphere of radius $R=1/\sqrt{K_\N}$, if $K_\N>0$ or in the plane if $K_\N\leq 0$,  there exists a comparison triangle with vertices $\bu^+_s,\bu^-_s,\bu_{1s}$, respective angles $\alpha_s^+,\alpha_s^-,\gamma_s$ and side lengths
$a^+,a^-,d$, satisfying $\alpha^\pm\le \alpha^\pm_s$.

In the case that $K_\N\leq 0$, by \eqref{planar_case} we obtain
\begin{equation}\label{planar_estimate}a^+\cos \alpha^+ + a^-\cos\alpha^- \geq d.\end{equation}
Else, note that if, at a given time instance $t>t_0$ we have $d \leq S/2$, then by \eqref{asymp_lem_jump_ass} and our choice of $\delta$,
\[
|\bu_x(t)|_{\N}(]0, x_0])\stackrel{(\ref{asymp_lem_jump_ass})}< \delta + d \stackrel{(\ref{def-delta_2})}< \tfrac{S}{2} + \tfrac{S}{2} = S
\]
and we are done. Therefore, we can assume that $d > S/2$. Then by triangle inequality
\[
 a^- \geq d - a^+ \ge d-\tfrac12 \delta \stackrel{(\ref{def-delta_2})}>d-\tfrac S 4> \tfrac{S}{4}.
\]
Since $\delta<\delta_0$, we can apply Lemma \ref{lem:sphere_est_asymp2} yielding
\begin{equation} \label{est_asymp2_eq}
H_{\bu_1, \bu}(\{x_0\}) \stackrel{(\ref{estimate_sphere})}=	a^+\cos \alpha^+ + a^-\cos\alpha^- \geq a^+\cos \alpha_s^+ + a^-\cos\alpha_s^- \geq d - \tfrac{S}{4}.
\end{equation}
Combining \eqref{dist_deriv}, \eqref{est_asymp_jump_close}, \eqref{est_asymp_jump_far}, and \eqref{planar_estimate}, we obtain, for $K_\N\leq 0$,
\begin{equation}\label{ghjk}
	\frac{1}{2}\frac{\dd}{\dd t}\int_0^{x_1} \dist_\N^2(\bu,\bu_1) \leq  \bu_{1,t}\cdot\int_0^{x_1}  \log_{\bu}(\bu_1) -|\bu_x|_{\N}(\{x_0\}) 
\end{equation}
In the case that $\mathcal K_\N>0$, \eqref{dist_deriv}, \eqref{est_asymp_jump_close}, \eqref{est_asymp_jump_far}, and \eqref{est_asymp2_eq}, yield
\begin{equation}\label{ghjkl}
	\frac{1}{2}\frac{\dd}{\dd t}\int_0^{x_1} \dist_\N^2(\bu,\bu_1) \leq  \bu_{1,t}\cdot\int_0^{x_1}  \log_{\bu}(\bu_1) - |\bu_x|_{\N}(\{x_0\}) + \tfrac{S}{2}.
\end{equation}
In both cases, again reasoning as in the proof of Lemma \ref{lem:asymp_case1}, we deduce that there exists a $\tau_* \geq 0$ as desired such that
\[|\bu_x(t)|_{\N}(\{x_0\})< S \quad \text{if } t_0+\tau_* \leq t < T \ {\rm if \ } K_\N\leq 0,\]
\[|\bu_x(t)|_{\N}(\{x_0\}) - \tfrac{S}{2} < \tfrac{S}{2} - \delta \quad \text{if } t_0+\tau_* \leq t < T \ {\rm if \ }K_\N>0.\]
Thus, recalling \eqref{asymp_lem_jump_ass} and our choice of $\delta$, we conclude.
\end{proof}

\begin{lemma} \label{lem:biting} There exists $t_\bullet= t_\bullet(\bu_0)>0$ such that \[TV(\bu(t_\bullet)) < \tfrac{1}{2}\radN,\]
provided that $t_\bullet < T$.
\end{lemma}
\begin{proof}We can assume that $\radN<+\infty$.
For $\bw \in BV_{rad}(I, \N)$, we denote \[\mathfrak{J}_{\bw} = \left\{x \in I \colon |(\bu_0)_x|_{\N} (\{x\}) \geq \tfrac{1}{2}\radN\right\}.\]
We fix $D>0$ equal to half of the minimum of numbers $\delta>0$ given by Lemma \ref{lem:asymp_case2} with $S = \frac{1}{2} \radN$ over $x_0\in \mathfrak{J}_{\bu_0}$. We take a partition $0 = y_0 < y_1 < \ldots < y_l = |I|$ of $I$ such that $y_k \not \in J_{\bu_0}$ and for $k = 0, \ldots, l-1$
either
\begin{equation} \label{smalljumpcond-} |(\bu_0)_x|_{\N} (]y_k, y_{k+1}[) < \radN
\end{equation}
or there exists a unique $y \in ]y_k, y_{k+1}[$ such that
\begin{equation} \label{bigjumpcond}
|(\bu_0)_x|_{\N} (]y_k, y[) < D, \quad |(\bu_0)_x|_{\N} (\{y\}) \geq \tfrac{1}{2}\radN, \quad |(\bu_0)_x|_{\N} (]y, y_{k+1}[) < \tfrac{1}{2}\radN - D,
\end{equation}
Such a partition can be constructed by first selecting points close enough to each $y \in \mathfrak{J}_{\bu_0}$ so that \eqref{bigjumpcond} holds, and then filling in the remainder of $I$ so that \eqref{smalljumpcond-} is satisfied. We then fix $B$ such that
\begin{equation}\label{defB}
\max\big\{\radN-D, \max\big\{|(\bu_0)_x|_{\N} (]y_k, y_{k+1}[): \mbox{ \eqref{smalljumpcond-} holds}\big\}\big\} <B<\radN.
\end{equation}
%
%
Thus \eqref{smalljumpcond-} may be rewritten as
\begin{equation}
\label{smalljumpcond} |(\bu_0)_x|_{\N} (]y_k, y_{k+1}[) <B< \radN\ .
\end{equation}
We produce the time instance $t_\bullet$ in $l$ steps corresponding to stages of evolution of $\bu$. In $k$-th step, we begin with a time instance $t_{k-1}\geq 0$ ($t_0 = 0$) such that, for any $\bu$, either
\begin{equation}
\label{smalljumpcondtk-1} |\bu_x(t_{k-1})|_{\N} (]0, y_k[) \le B
\end{equation}
or there exists $y \in ]0, y_k[$ such that
\begin{equation} \label{bigjumpcondtk-1}
	|\bu_x(t_{k-1})|_{\N} (]0, y[) < D, \quad |\bu_x(t_{k-1})|_{\N} (\{y\}) \geq \tfrac{1}{2}\radN, \quad |\bu_x(t_{k-1})|_{\N} (]y, y_k[) < \tfrac{1}{2}\radN - D.
\end{equation}
If $k < l$, we want to produce a time instance $t_k \geq t_{k-1}$ such that, for any $\bu$, either
\begin{equation} \label{smalljumpcondtk} |\bu_x(t_k)|_{\N} (]0, y_{k+1}[) \le B 
\end{equation}
or there exists $y \in ]0, y_{k+1}[$ such that
\begin{equation} \label{bigjumpcondtk}
	|\bu_x(t_k)|_{\N} (]0, y[) < D, \quad |\bu_x(t_k)|_{\N} (\{y\}) \geq \tfrac{1}{2}\radN, \quad |\bu_x(t_k)|_{\N} (]y, y_{k+1}[) < \tfrac{1}{2}\radN - D.
\end{equation}
We split our procedure into four cases.
\begin{itemize}
	 \item If \eqref{smalljumpcondtk-1} and \eqref{smalljumpcond} are satisfied, we use Lemma \ref{lem:asymp_case1} with $t_0=t_{k-1}$, $x_2=y_k$, and $\eps < B - |(\bu_0)_x|_{\N} (]y_k, y_{k+1}[)$. Since $y_k\notin J_{\bu}$, this yields \eqref{smalljumpcondtk}.

	\item Otherwise, if \eqref{bigjumpcondtk-1} and \eqref{smalljumpcond} are satisfied, we first use Lemma \ref{lem:asymp_case2} with $t_0=t_{k-1}$, $x_0=y$, and $S = \frac{1}{2}\radN$. Recalling the choice of $D$, this yields $|\bu_x(t_{k-1}+\tau_*)|_{\N} (]0, y]) < \frac12\radN$, hence (by \eqref{bigjumpcondtk-1}, \eqref{clelimit}, and \eqref{defB}) $|\bu_x(t_{k-1}+\tau_*)|_{\N} (]0, y_k[) < \radN-D<B$. Then we use Lemma \ref{lem:asymp_case1} with $t_0=t_{k-1}+\tau_*$, $x_2=y_k$, and $\eps < B - |(\bu_0)_x|_{\N} (]y_k, y_{k+1}[)$. Since $y_k\notin J_{\bu}$, this yields \eqref{smalljumpcondtk}.

	 \item Otherwise, if \eqref{smalljumpcondtk-1} and \eqref{bigjumpcond} are satisfied, we use Lemma \ref{lem:asymp_case1} with $t_0=t_{k-1}$, $x_2=y_k$, and $\eps < D-|(\bu_0)_x|_{\N} (]y_k, y[)$,
%
%
where $y$ is the unique element of $\mathfrak{J}_{\bu_0} \cap ]y_k, y_{k+1}[$. This yields $|\bu_x(t_{k-1}+\tau_*)|_{\N} (]0, y[)<D$; by \eqref{bigjumpcond}, this yields \eqref{bigjumpcondtk}.

	 \item Otherwise, \eqref{bigjumpcondtk-1} and \eqref{bigjumpcond} are necessarily satisfied. We first use Lemma \ref{lem:asymp_case2} with  $t_0=t_{k-1}$, $x_0=y$, and $S = \frac{1}{2}\radN$. As above, this yields $|\bu_x(t_{k-1}+\tau_*)|_{\N} (]0, y_k[) < \radN-D<B$. Then we use Lemma \ref{lem:asymp_case1} with $t_0=t_{k-1}+\tau_*$, $x_2=y_k$, and $\eps < D-|(\bu_0)_x|_{\N} (]y_k, y[)$,
%
%
where $y$ is the unique element of $\mathfrak{J}_{\bu_0} \cap ]y_k, y_{k+1}[$. Then, $|\bu_x(t_{k-1}+\tau_*)|_{\N} (]0, y[)<D$; by \eqref{bigjumpcond}, this yields \eqref{bigjumpcondtk}.
\end{itemize}
In the last,  $l$-th step, either \eqref{smalljumpcondtk} or \eqref{bigjumpcondtk} hold with $y_{k+1}=1$.
%
%
If \eqref{smalljumpcondtk} holds, we use Lemma \ref{lem:asymp_case1} with $\eps = \frac{1}{2} \radN$. Otherwise, \eqref{bigjumpcondtk} is satisfied. We first use Lemma \ref{lem:asymp_case2} with $S = \frac{1}{2}\radN$. Then we use Lemma \ref{lem:asymp_case1} with $\eps = \frac{1}{2}\radN$.
\end{proof}

We may now proceed with an argument similar to the one in the proof of \cite[Theorem 3]{giacomellilasicamoll}. It follows from Lemma \ref{lem:biting} that for each $t \in [t_\bullet, T[$ there exists a ball $B_{N}(\bp(t), R) \subset \N$ with $R < \frac{1}{2} \radN$ that contains the image of $\bu(t)$. By \cite[Theorem 1.2]{karcher}, the center of mass of $\bu(t)$, i.\,e.\ the unique minimizer $\bp_c(t)$ of $\bp\mapsto
\int_I\dist_\N^2(\bu(t),\bp)$ exists and satisfies
\begin{equation}\label{zeromean}
\int_I \log_{\bp_c(t)} \bu(t) =0.
\end{equation}
Moreover, it follows from Lemma \ref{contxslice}, that for $\tau >0$ sufficiently small (independently of $t$), the image of $\bu(t+\tau)$ is contained in $B_{N}(\bp(t), \frac{1}{2} \radN)$. (In fact, one can show that the image of $\bu(t)$ is contained in $B_{N}(\bp(t_\bullet), \frac{1}{2} \radN)$ for $t \in [t_\bullet, T[$, but we do not make use of this here.) Therefore, we can use \cite[Corollary 1.6]{karcher} in conjunction with \cite[5.8., Theorem 3]{evans} to show that $\bp_c \in H^1_{loc}([t_\bullet, T[, \N)$. Thus, we can use Lemma \ref{var_eq}, obtaining
\begin{equation}\label{passw}
\frac{1}{2}\frac{\dd}{\dd t}\int_I \dist_\N^2(\bu,\bp_c) \stackrel{(\ref{zeromean})}\leq -h_{\N}(|\bu_x|_{\N}(I))  |\bu_x|_{\N}(I)
\end{equation}
for $t_{\bullet} \leq t < T$. Since (ignoring for a moment dependence on time)
\[ |\bu_x|_{\N}(I) \geq \sup_{I} \dist_\N(\bu, \bu(0)) \geq |I|^{-\frac{1}{2}}\left( \int_I \dist_\N^2(\bu, \bu(0))\right)^{\frac{1}{2}} \geq |I|^{-\frac{1}{2}}\left( \int_I \dist_\N^2(\bu,
\bp_c)\right)^{\frac{1}{2}},\]
we obtain
\[\frac{\dd}{\dd t} \left( \int_I \dist_\N^2(\bu, \bp_c)\right)^{\frac{1}{2}} \leq - C  \]
as long as $\int_I \dist_\N^2(\bu, \bp_c) > 0$, and so $\bu(t, \cdot)$ is a constant function for
\[t \geq t_\bullet + C^{-1} \left( \int_I \dist_\N^2(\bu(t_\bullet,\cdot), \bu(t_\bullet,0))\right)^{\frac{1}{2}}. \]

\section{The case $\N=\S_+^{N-1}$}\label{sec:sphere}

In this section we show the equivalence between our notion of solution and the one introduced in \cite{gmmn} whenever both are applicable, i.\,e.\ in the case that the spatial domain is an interval and
$\N$ is the (hyper)octant $\S^{N-1}_+ = \{ (p^1, \ldots, p^n)\in \S^{N-1} \colon p^k > 0, k=1, \ldots, n\}$ in the (hyper)sphere $\S^{N-1}$. We first recall, in the case of single spatial dimension, the
concept of solution to (\ref{smootheqn},\ref{smoothbc}) given in \cite{gmmn}.

\begin{defn} Let $T \in ]0, + \infty]$ and let
	\[{\bu \in H^1_{loc}([0,T[; L^2(I, \S^{N-1}_+)) \cap L^\infty_{loc}([0,T[; BV(I,\S^{N-1}_+))}.\]
	We say that $\bu$ is a solution to (\ref{smootheqn},\ref{smoothbc}) in the sense of \cite{gmmn} if there exists
	\[{\bz \in L^\infty_{w^*}(0,T;L^\infty(I, \R^N)) \cap L^2_{w^*,loc}([0,T[;BV(I,\R^N))}\]
	satisfying for a.\,e.\ $t\in]0,T[$ \eqref{zineq}, \eqref{zbc} and
\begin{equation}
  \label{sphere-eqn} \bu_t \, \Lb^1 =\bz_x+\frac{\bu^*}{|\bu^*|}|\bu_x|\quad \text{as Radon measures on } I,
\end{equation}
\begin{equation}
  \label{wedge-eqn} \bu_t\wedge \bu=(\bz\wedge \bu)_x \quad \text{in } L^2(I,\R^N \wedge \R^N),
\end{equation}
and
\begin{equation}\label{tangent-eqn}
  \bz\cdot \bu=0 \quad \text{a.\,e.\ on } I.
\end{equation}
\end{defn}
In addition, any solution to (\ref{smootheqn},\ref{smoothbc}) in the sense of \cite{gmmn} satisfies
\begin{equation} \label{zeqn-sphere}
\bu_x \cdot \bz^* = |\bu^*| |\bu_x|
\end{equation}
as Radon measures in a.\,e.\ time instance \cite[Proposition 3.5]{gmmn}.
\begin{thm}
  $\bu$ is a strong solution to (\ref{smootheqn},\ref{smoothbc}) in the sense of Definition \ref{defsol} iff it is a solution in the sense of \cite{gmmn}.
\end{thm}
\begin{proof}
We recall that
\begin{equation}\label{qaz}
T_{\bu} \S^{N-1} =\{ \bv \in \R^N \colon \bu \cdot \bv =0\}.
\end{equation}
Suppose that $\bu$ is a strong solution in the sense of Definition \ref{defsol} and fix a time instance such that (\ref{maineqn}-\ref{zbc}) hold. Denoting by $\bu^g$ the geodesic representative
$\frac{\bu^*}{|\bu^*|}$, we have
  \begin{equation*}\bu_t = \pi_{\bu} \bz_x^a = \pi_{\bu^g} \bz_x - \pi_{\bu^g} \bz^j_x - \pi_{\bu^g} \bz_x^c.
  \end{equation*}
  Due to symmetry of the sphere, at any jump point of $\bu$, $\T^+_{\bu} - \T^-_{\bu}$ is parallel to $\bu^g$, whence orthogonal to $T_{\bu^g}\mathbb S^{N-1}$, so by \eqref{zeqnj}
  \begin{equation}\label{sphere-jump-identity}
   \pi_{\bu^g} \bz^j_x = 0.
   \end{equation}
  Recalling also property \eqref{cantor-m},
  \begin{equation*}
  \bu_t = \left( \mathbf I - \frac{\bu^*}{|\bu^*|} \otimes \frac{\bu^*}{|\bu^*|}\right) \bz_x = \bz_x - \frac{\bu^*}{|\bu^*|^2} \bu^* \cdot \bz_x\end{equation*}
  Using the Leibniz' formula for BV functions in 1D, property \eqref{ztang}, and recalling \eqref{qaz},
  \[
  \bu^* \cdot \bz_x = - \bu_x \cdot \bz^* = - (\bu_x^d + \bu_x^j) \cdot \bz^*.
  \]
Again due to symmetry of the sphere, using \eqref{zeqnj} it is easy to check that at any jump point $|\bz^*| = |\bu^*|$ and $\bz^* \parallel \frac{\bu_x}{|\bu_x|}$. Consequently, $\bu_x^j \cdot
\bz^*(t) = |\bu^*| |\bu_x^j|$. Together with \eqref{zeqn}, this yields \eqref{zeqn-sphere}, demonstrating \eqref{sphere-eqn}. Now, taking the wedge product of \eqref{sphere-eqn} with $\bu^*$ and using
Leibniz' formula, we get
  \[\bu_t \wedge \bu = \bz_x \wedge \bu^* = (\bz \wedge \bu)_x - \bz^* \wedge \bu_x.  \]
  By \eqref{zeqn} and \eqref{zeqnj}, the last term vanishes, leading to \eqref{wedge-eqn}. Condition \eqref{tangent-eqn} follows trivially from \eqref{ztang}.

   Next, we consider a solution $\bu$ in the sense of \cite{gmmn}. Taking the Radon derivative of \eqref{sphere-eqn} with respect to $\Lb^1$ yields
  \begin{equation}\label{sphere-eqn-ac}\bu_t = \bz_x^a + \bu |\bu_x^a|.
  \end{equation}
  Equation \eqref{zeqn-sphere} implies $|\bu_x^a| = \bz \cdot \bu_x^a$. Recalling that in the case $\N = \S^{N-1}_+$, $\A_{\bp}(X,Y) = \bp\, X\cdot Y$, we obtain \eqref{secondfund-eqn},
  which is an equivalent form of \eqref{maineqn}. Subtracting \eqref{sphere-eqn-ac} from \eqref{sphere-eqn}, we get
  \begin{equation} \label{sphere_sing}
  0 = \bz_x^s + \bu^g |\bu_x^s|,
  \end{equation}
  whence \eqref{zjumpset} follows. Projecting \eqref{sphere_sing} onto $T_{\bu^g} \N$ yields
  \[\pi_{\bu^g} \bz_x^s = 0,\]
  whence \eqref{cantor-m} follows. Since $\bz^-$ (resp.\ $\bz^+$) is left-continuous (resp.\ right-continuous) and tangent spaces of a submanifold in $\R^N$ vary continuously with respect to the point,
  \eqref{ztang} is an immediate consequence of \eqref{tangent-eqn}. \ Equation \eqref{zeqn} follows from \eqref{zeqn-sphere} by Radon derivation, taking into account \eqref{zineq}. Equation
  \eqref{wedge-eqn} implies that $\bz \wedge \bu$ does not have jumps, so
  \begin{equation} \label{wedge-jump} \bz^+ \wedge \bu^+ = \bz^- \wedge \bu^-
  \end{equation}
  at any jump point of $\bu$. Taking wedge product of \eqref{wedge-jump} with $\bu^\pm$, we obtain $\bz^+ \wedge \bu^- \wedge \bu^+ = 0 = \bz^- \wedge \bu^+ \wedge \bu^-$, so $\bz^+$ and $\bz^-$ are in
  the plane spanned by $\bu^+$ and $\bu^-$. Since we already know that $\bz^\pm$ belongs to $T_{\bu^\pm}\N$, we see that $\bz^\pm$ is parallel to $\T_{\bu}^\pm$. Now, \eqref{wedge-jump} allows us to
  deduce that $|\bz^-| = a = |\bz^+|$ with $a\geq 0$. By \eqref{zeqn-sphere} and \eqref{ztang}, at any jump point of $\bu$,
  \begin{equation}\label{zeqn-sphere-jump} \bu^+ \cdot \bz^- - \bu^-\cdot \bz^+ = |\bu^+ + \bu^-| |\bu^+ - \bu^-|.
  \end{equation}
  Denoting $\bu^+ \cdot \bu^- = \cos \alpha$, we have $\bu^+ \cdot \bz^- = \pm a \sin \alpha$ and $\bu^-\cdot \bz^+ = \mp a \sin \alpha$ (the signs are opposite, because the orientation of pairs
  $(\bu^+, \bz^+)$ and $(\bu^-, \bz^-)$ in the plane is the same due to \eqref{wedge-jump}). Thus, having in mind that $|\bu^\pm|=1$, we obtain from \eqref{zeqn-sphere-jump}
  \[ \pm 2 a \sin \alpha = 2 \sqrt{1 - \cos^2 \alpha} = 2 |\sin \alpha|.\]
  Therefore, $a=1$ and $\bu^+ \cdot \bz^- >0$, $\bu^-\cdot \bz^+ <0$. This concludes the proof of \eqref{zeqnj}.

\end{proof}

\section{Examples of solutions}
\label{sec:examples}
In this section we provide two (semi-)explicit examples of strong solutions to (\ref{smootheqn},\ref{smoothbc}), which may also help understanding the features of Definition \ref{defsol}.

\medskip

\noindent {\bf Example 1: flow on a geodesic.}

\medskip

Take a geodesic segment $\bgamma$ in $\N$ of length $2s_0<\injN$, parametrized by its arc-length originating from its midpoint. Assume $\bu_0=\bgamma\circ \sigma_0$ with $\sigma_0\in BV(I;[-s_0,s_0])$. Then a strong solution to \eqref{smootheqn}-\eqref{smoothbc} is given by $\bu=\bgamma\circ \sigma$, where $(\sigma,\zeta)$ solves the total variation flow with initial datum $\sigma_0$. It follows from \cite[Def.\ 2.5 and Thm.\ 2.6]{acmbook} that
\begin{align}
\label{tv-sigma}
& \sigma\in L^\infty_{loc}([0,+\infty);BV(I))\cap W^{1,2}_{loc}([0,+\infty);L^2(I)), \ \|\sigma\|_\infty \le s_0,
\\
\label{tv-zeta}
& \zeta \in L^2_{loc}([0,+\infty);H^1(I)), \ \|\zeta\|_\infty\le 1, \ \zeta|_{\partial I}=0 \ \mbox{in $L^2([0,+\infty))$}, \
\\
\label{tv-pde}
& \sigma_t=\zeta_x \in L^2([0,+\infty)\times I),
\\
\label{tv-id}
&
\int_I \zeta(t) \dd \sigma_x(t) = \int_I \dd |\sigma_x| \quad\mbox{for a.\,e.\ $t$}.
\end{align}
With respect to the general Definition 2.5 in \cite{acmbook}, $\|\sigma\|_\infty \le s_0$ follows from weak comparison principle, the additional regularities of $\zeta$ and its trace follow from \eqref{tv-pde} since $I\subseteq \R$, and the additional regularity of $\sigma$ follows from \cite[Thm 3.6]{brezisomm} since $\sigma_0\in BV(I)$.

Because of $\|\sigma\|_\infty \le s_0$, $\bu$ is well defined. Its regularity properties are immediate from \eqref{tv-sigma}, and the inequality $\dist_\N(\bu^-,\bu^+)<\injN$ follows from $\|\sigma\|_\infty\le s_0<\frac12 \injN$. Let
\begin{align}
\label{def-z}
\bz(t,x)&:=\zeta(t,x)\bgamma'(\sigma(t,x)),
\end{align}
so that for a.\,e.\ $t\ge 0$ it holds that
\begin{align}
\label{zx1}
\bz_x(t)&=\underbrace{\zeta_x(t)}_{\in L^2(I)}\underbrace{\bgamma'(\sigma(t))}_{\in BV(I)} + \underbrace{\zeta(t)}_{\in C(\overline I)}\underbrace{(\bgamma'(\sigma(t)))_x}_{\in \M(I)}
\\ \label{zx2}
& = \zeta_x(t)\bgamma'(\sigma(t)) + \zeta\bgamma''(\sigma(t))\sigma_z^d(t) + \zeta(t) \left((\bgamma'(\sigma(t)))^+-(\bgamma'(\sigma(t)))^-\right)\llcorner J_{\bu(t)} \ .
\end{align}
It is easily seen from \eqref{def-z}-\eqref{zx1} that the regularity properties of $\bz$, as well as \eqref{zineq} and \eqref{zbc}, follow from \eqref{tv-zeta} and the smoothness of $\N$. We now argue for a.\,e.\ fixed $t$ and omit dependence on it for notational convenience. Since $\bgamma$ is a geodesic we have $\nabla_{\bgamma'}\bgamma'=0$, hence $\pi_{\bu}\bgamma''=0$. Therefore, by \eqref{zx2}, $\pi_{\bu} \bz_x^c = \pi_{\bu} \bgamma''(\sigma(t))\zeta\sigma_z^c =0$, i.\,e.\ \eqref{cantor-m} holds, and
$$\
\bu_t= \bgamma'(\sigma)\sigma_t \stackrel{(\ref{tv-pde})}=\bgamma'(\sigma)\zeta_x = \pi_{\bu} \bgamma'(\sigma)\zeta_x + \pi_{\bu}\bgamma''\zeta\sigma^a_x  \stackrel{(\ref{zx2})} = \pi_{\bu} \bz_x^a,
$$
whence \eqref{maineqn}. Conditions \eqref{ztang} and \eqref{zjumpset} are obvious by definition. It follows from \eqref{tv-id} and $|\zeta|\le 1$ that $\zeta\sigma_x=|\sigma_x|$ as measures, which means in particular that
\begin{equation}\label{z-rn}
\zeta \sigma_x^d=|\sigma_x^d| \ \mbox{ $|\sigma_x^d|$-a.\,e.} \quad\mbox{and}\quad \zeta \sigma_x^j=|\sigma_x^j| \ \mbox{pointwise on $J_{\bu}$.}
\end{equation}
Hence on the diffuse part
$$
\bz=\bgamma'(\sigma)\zeta \stackrel{(\ref{z-rn})}= \bgamma'(\sigma)\frac{\sigma_x^d}{|\sigma_x^d|} =\frac{ \bgamma'(\sigma)\sigma_x^d}{|\bgamma'(\sigma)\sigma_x^d|} = \frac{\bu_x^d}{|\bu_x^d|} \mbox{$\bu_x^d$-a.\,e.},
$$
which is \eqref{zeqn}.  At a jump point $x_0\in I$, if $\sigma^+-\sigma^-=\sigma(x_0^+)-\sigma(x_0)^->0$, then by \eqref{z-rn} $\zeta=1$ and $(\bgamma'(\sigma))^\pm=T_{\bu}^\pm$, hence $\bz^\pm = \zeta (\bgamma'(\sigma))^\pm = (\bgamma'(\sigma))^\pm = T_{\bu}^\pm$ at $x=x_0$. If the opposite inequality holds, then $\zeta=-1$ and $(\bgamma'(\sigma))^\pm=-T_{\bu}^\pm$, yielding again $\bz^\pm= T_{\bu}^\pm$. Hence \eqref{zeqnj} holds, which completes the verification of Definition \ref{defsol}.

\medskip

\noindent {\bf Example 2: piecewise constant initial data.}

\medskip

The case of a single jump, $\bu_0(x) = \bu_0^- \chi_- + \bu_0^+\chi_+$ with $\dist_\N(\bu_0^-,\bu_0^+)<\injN$, $\chi_-=\chi_{[0,x_0]}$ and $\chi_+=\chi_{[x_0,1]}$, boils down to Example 1: $\bu=\bgamma\circ\sigma$, where $\bgamma$ is the geodesic connecting $\bu_0^-$ to $\bu_0^+$ and $\sigma_0(x)=-s_0\chi_-+s_0\chi_+$, where $2s_0=L(\bgamma)$. Simple computations, see e.\,g.\ \cite[Remark 2.7]{bonfortefigalli}, permit to solve explicitly for $\sigma$, yielding
$$
\bu(t,x) = \bgamma\left(-s_0+x_0^{-1} t\right) \chi_0(x) +  \bgamma\left(s_0-(1-x_0)^{-1} t\right) \chi_1 =: \bu^-\chi_0 +\bu^+\chi_1.
$$
In particular, the jump point remains fixed at $x_0$ and vanishes at $t=2s_0x_0(1-x_0)$, and $\bu^\pm$ move towards each other along $\bgamma$.

\medskip

We now consider the case of two jump points, the generalization to a finite number of them being conceptually straightforward. Let then
$$
\bu_0 = \ba_{0,0}\chi_0+\ba_{0,1}\chi_1+\ba_{0,2}\chi_2, \quad \chi_i=\chi_{[x_{i-1},x_i]}, \quad 0=x_{-1} <x_0<x_1<x_2=1,
$$
with $\dist_\N(\ba_{0,i},\ba_{0,i+1})<\injN$.  We make the Ansatz that $\bu$ is piecewise constant with fixed jump points, i.\,e.\ $\bu(t,x) = \ba_0(t)\chi_0(x)+\ba_1(t)\chi_1(x)+\ba_2(t)\chi_2(x)$, and that $\bz$ is piecewise linear: then (1.6), (1.9) and (1.10)
%
%
are trivial, and (1.8), (1.11) and (1.12)
%
%
enforce
\begin{align*}
\bz(t,x) & = \frac{x}{x_0 \dd_{0,1}}\log_{\ba_0(t)}\ba_1(t) \chi_0 - \frac{1-x}{(1-x_1)\dd_{1,2}}\log_{\ba_2(t)}\ba_1(t) \chi_2
\\
& + \frac{1}{x_1-x_0}\left( \frac{x-x_0}{\dd_{1,2} }\log_{\ba_1(t)}\ba_2(t)  -  \frac{x_1-x}{\dd_{0,1} }\log_{\ba_1(t)}\ba_0(t) \right) \chi_1
\end{align*}
where $\dd_{i,i+1}=\dist_N(\ba_i,\ba_{i+1})=|\log_{\ba_i}\ba_{i+1}|$.  With this choice (1.7)
%
%
holds and (1.5)
%
%
yields the ODE system
\begin{equation}\label{e2-syst}
\left\{\begin{array}{l}
\ba_0' = \frac{1}{x_0 \dd_{0,1}}\log_{\ba_0(t)}\ba_1(t)
\\
\ba_1'= \frac{1}{x_1-x_0}\left( \frac{1}{\dd_{1,2} }\log_{\ba_1(t)}\ba_2(t)  +  \frac{1}{\dd_{0,1} }\log_{\ba_1(t)}\ba_0(t) \right)
\\
\ba_2' = \frac{1}{(1-x_1) \dd_{1,2}}\log_{\ba_2(t)}\ba_1(t)
\end{array}\right.
\end{equation}
which has a unique solution, hence (1.5)
%
%
holds, as long as $\dd_{i,i+1}>0$. Hence $\bu$ is indeed a solution provided $\dist_\N(\ba_{i},\ba_{i+1})<\injN$. To check this, we note that
\begin{align*}
\frac12 \frac{\dd}{\dd t} \dist_\N^2(\ba_0,\ba_1) &=\frac12 \frac{\dd}{\dd t} |\log_{\ba_0}\ba_1|^2 = -\log_{\ba_0}\ba_1 \cdot \ba_0' - \log_{\ba_1}\ba_0  \cdot \ba_1'
\end{align*}
and simple computations using \eqref{e2-syst} yield
\begin{align*}
\frac12 \frac{\dd}{\dd t} \dist_\N^2(\ba_0,\ba_1) \leq - \frac{1}{x_0} \dist_\N (\ba_0,\ba_1).
\end{align*}
Analogously $\frac12 \frac{\dd}{\dd t} \dist_\N^2(\ba_1,\ba_2)  \leq - \frac{1}{1-x_1} \dist_\N (\ba_1,\ba_2)$, whence the desired bound. Therefore $\bu$ is a strong solution. Solving these ODEs for $\dist_\N^2$ implies that the first jump point vanishes at $t_0<2\min\{x_0 \dist_\N(\ba_0,\ba_1), (1-x_1)\dist_\N(\ba_1,\ba_2)\}$. For $t>t_0$ $\bu$ behaves as in the first part of the example (one jump point).

\bigskip

\noindent\textit{Acknowledgments.} Partial support is acknowledged from GNAMPA-INdAM Projects.
The second author was partially supported by grants no.\ 2020/36/C/ST1/00492 and 2024/55/D/ST1/03055 of the National Science Center, Poland.
 The third author acknowledges partial support by project PID2022-136589NB-I00
and the network RED2022-134077 funded by MCIN/AEI/ 10.13039/501100011033 and
by ERDF A way of making Europe. Part of this work was created during second author’s Postdoctoral Fellowship at SBAI Department, Sapienza University of Rome, and JSPS Postdoctoral Fellowship at the University of Tokyo.

\appendix\section{Appendix}\label{sec:app1}
In this Appendix, we collect several results on Riemannian manifolds that we use in the paper and we have not found in the literature.

\begin{lemma}\label{lem:u_x_tangent}
  For any $\bw\in BV(I;\mathcal N)$, $\bw_x^a\in T_{\bw}\N$ $\Lb$-a.\,e.
\end{lemma}
\begin{proof}
By definition of $BV(I;\mathcal N)$, any representative of $\bw$ belongs to $\N$ $\Lb$-a.\,e. Therefore, if we identify $\bw$ with its semicontinuous representative, then $\bw$ belongs to $\N$
everywhere. Furthermore, the classical derivative $\bw'$ exists $\Lb$-a.\,e.\ and the $L^1$ function it defines coincides with $\bw_x^a$ \cite[Theorem 3.28]{afp}. At any point $x_0$ where $\bw'(x_0)$
exists,
	\[\bw(x_0 + h) - \bw(x_0) = O(|h|)\]
and therefore
\[
\bw(x_0 + h) - \bw(x_0) \in T_{\bw(x_0)} \N + {O(|\bw(x_0 + h) - \bw(x_0)|^2) = T_{\bw(x_0)} \N + O(|h|^2)}. 
\]
Consequently, $\bw'(x_0) \in T_{\bw(x_0)} \N$.
\end{proof}

\begin{proof}[Proof of Lemma \ref{lem-dist-comp}]
Since the embedding is closed and it is  a homeomorphism, $\overline{B(0,R)}\cap \N$ is compact in $\R^N$ and in $(\N,\dist_\N)$. The function
$$
f(\bp_1,\bp_2):=\left\{\begin{array}{ll} \frac{|\bp_1-\bp_2|}{\dist_\N(\bp_1,\bp_2)} & \bp_1\ne \bp_2 \\
1 & \bp_1= \bp_2
\end{array}\right.
$$
is continuous (at $\bp_1=\bp_2$), since $\bp_2-\bp_1= \exp^{-1}_{\bp_1}\bp_2 + O(|\bp_1-\bp_2|^2)$ as $\bp_2\to \bp_1$, hence
$$
1+o(1) {\le} \frac{|\bp_1-\bp_2|}{\dist_\N(\bp_1,\bp_2)}\le 1 \quad\mbox{as $\bp_2\to \bp_1$.}
$$
Hence $f$ is bounded {from} below in $\overline{B(0,R)}\cap \N$.
\end{proof}

\begin{proof}[Proof of Proposition \ref{fermiprop}]

Let
	\[V\!= \, ]\!-\radN, \radN[ \times B^{n-1}_{\radN}(\boldsymbol 0) = \left\{(p^1, \ldots, p^n)\in \Rn \colon |p_1|< \radN, \sum_{k=2}^n p_k^2 < \radN^2 \right\}.\]
	We denote by $\bp_0$ the point $\bgamma(0)$ and, by $P(t)$, the parallel transport from $T_{\bp_0}\N$ to $T_{\bgamma(t)}\N$ along $\bgamma$. We recall that $t \mapsto P(t)$ is the family of linear
isometries such that, for every $\bv \in T_{\bp_0}\N$, $P(0)\bv = \bv$ and
	\begin{equation}\label{paralleldef}
	\left(P(t) \bv \right)_{;t} = \boldsymbol 0 \quad \text{for  } t \in ]\!-\radN, \radN[,
	\end{equation}
	i.\,e.\ the covariant derivative of $P(\cdot) \bv$ vanishes. We choose an orthonormal basis $\be_1, \ldots, \be_n$ of $T_{\bp_0}\N$ such that $\be_1 = \bgamma'(0)$. We define a smooth map $G\colon V
\to U(\bgamma)=U(\bgamma,\radN)$ by
	\[G(p^1, \ldots, p^n) = \exp_{\bgamma(p^1)} P(p^1) \sum_{k=2}^n p^k \be_k,
\]
where $U(\bgamma,\radN)$ is defined in \eqref{def-ugr}. We will prove that $G$ is a diffeomorphism and that its inverse is a Fermi normal coordinate chart on $U(\bgamma)$ flat along $\bgamma$. It is easy to check that $G$ is surjective.
	
	Since $\radN \leq \injN$, $\bgamma$ does not self-intersect. Therefore, $G(\cdot, 0, \ldots,0)$ is injective. Let $p^1 \in ]\!-\radN, \radN[$ and for $i=1,2$ let $\bw_i \in T_{\bgamma(p^1)}\N$,
$\bw_i \perp T_{\bp_1}\bgamma$, $|\bw_i|_{\N} = 1$. Again, since $\radN \leq \injN$, geodesic segments $\{\exp_{\bgamma(p^1,\cdot,\cdots,\cdot)} t \bw_i, t \in ]0, \radN[ \}$ cannot intersect. Thus, $G(p^1, \cdot, \ldots,
\cdot)$ is injective for all $p^1 \in ]\!-\radN, \radN[$.
	
	Let $\bp_1, \bp_2 \in \bgamma$, $\bp_1 \neq \bp_2$, and let $\bw_i \in T_{\bp_i} \N$, $\bw_i \perp T_{\bp_i}\bgamma$, $|\bw_i|_{\N} = 1$ for $i=1,2$. We denote by $\bgamma_i$ the geodesic segment
$\{\exp_{\bp_i} t \bw_i, t \in ]0, \radN[ \}$. Since $2\radN \leq \injN$, $\bgamma_i$ does not intersect $\bgamma$ for $i=1,2$.
%
%
	We will now show that $\bgamma_1$ and $\bgamma_2$ do not intersect. We suppose that they do and denote by $\bp_3$ the first point in their intersection. Points $\bp_1, \bp_2, \bp_3$ together with
parts of $\bgamma$, $\bgamma_1, \bgamma_2$ form a geodesic triangle $\triangle$. We first consider the case $K_\N \leq 0$. Since $2 \radN \leq \injN$, $\triangle$ does not meet the cut locus of its
vertices. 	By Alexandrov's comparison theorem (Lemma \ref{act}), there exists a triangle $\triangle'$ in the Euclidean plane of the same edge lengths whose angle sizes are not smaller than
those of $\triangle$. By definition of $\bw_1$, $\bw_2$, $\triangle'$ has two angles of size at least $\frac{\pi}{2}$ which cannot be.
	
	Next, we suppose $K_{\N} > 0$. In order to apply Alexandrov's comparison theorem, we now need to check that the perimeter of $\triangle$ is not greater than $\frac{2 \pi}{\sqrt{K_\N}}$, which is
indeed the case by definition of $\radN$. Thus, there exists a triangle $\triangle'$ in the sphere of radius $\frac{1}{\sqrt{K_\N}}$ of the same edge lengths as $\triangle$ whose angle sizes are not
smaller than those of $\triangle$. We denote the vertex of $\triangle'$ corresponding to $\bp_i$ by $\bq_i$ for $i=1,2,3$. Since $\triangle'$ is contained in the open hemisphere centered at $\bq_3$, it
is again an impossibility that both $\angle \bq_1\bq_2\bq_3$ and $\angle \bq_2\bq_1\bq_3$ are larger than $\frac{\pi}{2}$. This concludes the proof that $G$ is injective.

	It remains to show that the inverse of $G$ is smooth. It is enough to demonstrate that the linear map $DG(p^1, \ldots, p^n)\colon V \to T_{G(p^1, \ldots, p^n)}\N$ is an isomorphism for all $(p^1,
\ldots, p^n) \in V$. For readability, throughout this argument we will denote $X = P(p^1) \sum_{k=2}^n p^k \be_k$. Appealing to \eqref{paralleldef}, we calculate
	\begin{equation}\label{DGe1prel}
	DG(p^1, \ldots, p^n) \be_1 = D_{\bp} (\exp_{\bp} Y) \left|_{\subalign{\bp &= \bgamma(p^1) \\ Y&=X}}\right. \bgamma'(p^1).
	\end{equation}
	We recall that, since $\exp_{\bp}\log_{\bp} \bq = \bq$ trivially whenever $\dist_{\N}(\bq, \bp)< \injN$, we have
	\[\left.D_{\bp}(\exp_{\bp} Y)\right|_{Y = \exp^{-1}_{\bp} \bq}+ (D\exp_{\bp})(\log_{\bp} \bq) D_{\bp} \log^{-1}_{\bp} \bq = 0,\]
	and so, denoting $\bq = \exp_{\bp} Y$,
	\begin{equation} \label{Dpexpp}
	D_{\bp}(\exp_{\bp} Y) = - D\exp_{\bp} (Y)\, D_{\bp} \log_{\bp} \bq = D\exp_{\bp} (Y)\, D^2_{\bp}\, \tfrac{1}{2} \dist_{\N}^2(\bp, \bq).
	\end{equation}
	Plugging \eqref{Dpexpp} into \eqref{DGe1prel} yields
	\[DG(p^1, \ldots, p^n) \be_1 = D \exp_{\bgamma(p^1)} (X)\, D^2_{\bp}\, \tfrac{1}{2} \dist_{\N}^2(\bp, \exp_{\bgamma(p^1)} X)\left|_{\bp = \bgamma(p^1)}\right. \, P(p^1) \be_1.\]
	On the other hand, for $k = 2, \ldots, n$, we have
	\begin{equation*}\label{DGek}
	DG(p^1, \ldots, p^n) \be_k = D \exp_{\bgamma(p^1)} (X)\, P(p^1) \be_k.
	\end{equation*}
	Summing up,
	\[DG(p^1, \ldots, p^n) = D \exp_{\bgamma(p^1)} (X)\,A\, P(p^1),\]
	where $A \colon T_{\bgamma(p^1)}\N \to T_{\bgamma(p^1)}\N$ is a linear endomorphism whose matrix in basis $P(p^1) \be_1, \ldots,$ $P(p^1) \be_n$ is lower triangular of form
	\[\begin{bmatrix}
	a_{11} & 0 &  \ldots & 0 \\
	a_{12} & 1 & & 0 \\
	\vdots &  & \ddots & \\
	a_{1n} & 0 & & 1
	\end{bmatrix}.   \]
	Now, since $|X|_{\N} < \radN$, we have
	\[a_{11} = P(p^1) \be_1\cdot D^2_{\bp}\, \tfrac{1}{2} \dist_{\N}^2(\bp, \exp_{\bgamma(p^1)} X)\left|_{\bp = \bgamma(p^1)}\right. \, P(p^1) \be_1 > 0, \]
	hence $A$ is invertible. $D \exp_{\bgamma(p^1)} (X)$ is invertible for the same reason.
	
	We have proved that $F = G^{-1}$ is a diffeomorphism. The fact that it satisfies properties \eqref{fermi-prop} is well-known and easy to check.

\medskip

We now prove the additional assertion. It suffices to suppose that $K_\N>0$. Let $\bw_1,\bw_2\in \bgamma$ be symmetric with respect to $\bp$ and let $\bw\colon]a,b[\to \N$ be a curve joining $\bw_1$ and $\bw_2$ which is not contained in
$U:=U(\bgamma, \radN)$. Let $\bw_0$ be the first point which is not contained in $U$: then $\bw_0\in \partial U\setminus U$. Since $G$ is continuous up to $\partial V$, $\bw_0=G(p_0^1,\dots,p_0^n)$
for some $(p_0^1,\dots,p_0^n)\in \partial V$. Consider the geodesic segments  $\bgamma_1=\bgamma_{\bw_1}^{\bw_0}$ and $\bgamma_2=\bgamma_{\bw_0}^{\bw_2}$. Since $L(\bw)\ge L(\bgamma_1)+L(\bgamma_2)$, it
will suffice to show that $L(\bgamma_1)+L(\bgamma_2)
\ge 2\radN$. We now distinguish two cases.

\smallskip

If $|p^1_0|<\radN$, then $\sum_k (p^k_0)^2=\radN^2$, that is, $\bw_0=\exp_{\tilde \bp}\bv_0$ with $|\bv_0|=\radN$ and $\tilde \bp\in \bgamma$. Consider the geodesic $\tbgamma$ joining $\bw_0$
and $\tilde\bp$, parametrized by $\exp_{\tilde \bp}(t\bv_0)$, $t\in [0,1]$. We have $L(\tbgamma)=\radN$ and, by construction of Fermi coordinates, $\angle \bw_0\tilde\bp\bw_1= \angle
\bw_0\tilde\bp\bw_2=\pi/2$. Suppose that $L(\bgamma_1)<\radN$. Then the perimeter of the geodesic triangle with vertices $\tilde\bp$, $\bw_0$ and $\bw_1$ is strictly smaller than $4\radN$,
whence {the second part of} Lemma \ref{act} may be applied on the sphere of radius $R=\frac{1}{\sqrt{K_\N}}${, yielding (with corresponding notation) $0\le \frac1R c_s\le \frac1R c=\frac1R L(\bgamma_1)<\frac1R \radN\le \frac\pi 2$, hence $\cos( \tfrac1R L(\bgamma_1))\le \cos (\frac1R c_s)$.} By the Pythagorean theorem on the sphere, this leads to
\[
\cos( \tfrac1R L(\bgamma_1)) \leq\cos (\tfrac1R c_s)=\cos(\tfrac1R L(\tbgamma)) \cos(\tfrac1R L(\bgamma_{\bw_1}^{\tilde \bp}))= \cos(\tfrac1R\radN) \cos(\tfrac1R L(\bgamma_{\bw_1}^{\tilde \bp})).
\]
{Since $\frac1R\radN\leq \frac{\pi}{2}$, this implies that $\cos( \tfrac1R L(\bgamma_1)) \le \cos(\tfrac1R\radN)$, that is,} $L(\bgamma_1) \geq \radN$, a contradiction. By the same argument $L(\bgamma_2) \geq \radN$.

\smallskip

If $|p^1_0|=\radN$, assume w.\,l.\,o.\,g.\ that $p^1_0=-\radN$, let $\tilde\bp=G(-\radN,0,\dots,0)$, and assume by contradiction that
\begin{equation}\label{vbn0}
L(\bgamma_1)+L(\bgamma_2)<2\radN.
\end{equation}
We have, by the symmetry of $\bw_1$ and $\bw_2$,
\begin{equation}\label{vbn1}
L(\bgamma_{\bw_1}^{\tilde \bp})+ L(\bgamma_{\bw_2}^{\tilde \bp})=2\radN,
\end{equation}
and we can assume w.l.o.g. that $L(\bgamma_{\bw_1}^{\tilde \bp})\leq\radN$. Using \eqref{vbn0} and \eqref{vbn1}, we compute:
\begin{equation*}
L(\bgamma_{\bw_0}^{\tilde \bp}) + L(\bgamma_2) + L(\bgamma_{\bw_2}^{\tilde\bp}) < L(\bgamma_{\bw_0}^{\tilde \bp}) +4\radN - L(\bgamma_1) - L(\bgamma_{\bw_1}^{\tilde \bp}),
\end{equation*}
whence by the triangle inequality
\begin{equation*}
L(\bgamma_{\bw_0}^{\tilde \bp}) + L(\bgamma_2) + L(\bgamma_{\bw_2}^{\tilde\bp}) < 4\radN,
\end{equation*}
and by the same argument also the perimeter of the triangle with vertices in $\bw_0$, ${\tilde \bp}$ and $\bw_1$ is less than $4\radN$. We can suppose that $\bw_0\neq \tilde\bp$ since otherwise, by \eqref{vbn1}, we would have finished the proof. Therefore {the second part of} Lemma \ref{act} may be applied on the sphere of radius $R=\frac{1}{\sqrt{K_\N}}$ to both triangles, keeping the angles $\angle \bw_0\tilde\bp\bw_1= \angle \bw_0\tilde\bp\bw_2=\pi/2$. Then (with corresponding notation) $c_{s,i}\leq  L(\bgamma_i)$ 
 for $i=1,2$. Again by the Pythagorean theorem on the sphere, we have
\begin{align*}
\cos(\tfrac1R c_{s,i}) =\cos (\tfrac1R L(\bgamma_{\bw_0}^{\tilde \bp})) \cos(\tfrac1R L(\bgamma_{\bw_i}^{\tilde \bp})),\quad i=1,2.
\end{align*}
{By \eqref{vbn0}, $\tfrac1R c_{s,i}<\pi$. Therefore we may pass to the inverse, yielding}
\begin{align*} \frac{c_{s,1}+c_{s,2}}{R} &=\arccos(\cos(\tfrac1R L(\bgamma_{\bw_0}^{\tilde \bp})) \cos(\tfrac1R L(\bgamma_{\bw_1}^{\tilde \bp})))+\arccos(\cos(\tfrac1R L(\bgamma_{\bw_0}^{\tilde \bp}) \cos(\tfrac1R L(\bgamma_{\bw_2}^{\tilde \bp})))
\\ & \stackrel{\eqref{vbn1}}= \arccos(\cos(\tfrac1R L(\bgamma_{\bw_0}^{\tilde \bp})) \cos(\tfrac1R L(\bgamma_{\bw_1}^{\tilde \bp})))+\arccos(\cos(\tfrac1R L(\bgamma_{\bw_0}^{\tilde \bp}) \cos(\tfrac1R (2{\rm rad}_\N -L(\bgamma_{\bw_1}^{\tilde \bp}))).
\end{align*}
Note that the function $x\mapsto \arccos(a\cos(x-b))-x$ is decreasing in $\R$. Therefore
\begin{align*}
\arccos(a\cos b) + \arccos(a\cos(x-b))-x  & \ge \arccos(a\cos b) + \arccos(a\cos(\pi-b))-\pi
\\
& = \arccos(a\cos b) + \arccos(-a\cos b)-\pi =0
\end{align*}
for all $x\le \pi$. Applying this inequality with $x= \tfrac 2 R \radN\in [0,\pi]$ yields
$$
\frac{c_{s,1}+c_{s,2}}{R} \geq  \frac{2{\rm rad}_\N}{R},
$$
in contradiction with \eqref{vbn0} since $L(\bgamma_i)\ge c_{s,i}$.
The last assertion follows from the previous one choosing $\bw:=[\bp,\bq,\bp]$ for any $\bq\in B_{\N}(\bp,{\rm rad}_\N)$.
\end{proof}

In the remainder of the section we prove Lemma \ref{lem-what}. Let us first consider the case of an Alexandrov triangle on the sphere $\mathbb S^2$.

\begin{lemma}\label{lem-what-S2}
	Let $r_0<\pi$, $\bp\in \mathbb S^{2}$, and $\bq,\tilde \bq\in B_{\N}(\bp,r_0)$. There exists a constant $C>0$ (depending on $r_0$) such that
	\begin{equation}\label{what-sphere}
		\forall\, \hat \bp\in \bgamma_{\bp}^{\bq} \ \exists\, \overline\bp\in \bgamma_{\bp}^{\tilde \bq} : \ \dist_{\mathbb S^{2}}(\hat \bp, \overline \bp) \le C \dist_{\mathbb S^{2}}(\bq, \tilde \bq)
		%
		%
	\end{equation}
\end{lemma}

\begin{proof}
	Let $a< r_0$, $b< r_0$, 
$\tilde a\le a$, $\tilde b\le b$, $c=\dist_{\mathbb S^{2}}(\bq, \tilde \bq)$, and $\tilde c$ be the lengths of $\bgamma_{\bp}^{\bq}$, $\bgamma_{\bp}^{\tilde\bq}$, $\bgamma_{\bp}^{\hat \bp}$,
$\bgamma_{\bp}^{\overline\bp}$, $\bgamma_{\bq}^{\tilde\bq}$, and $\bgamma_{\hat\bp}^{\overline\bp}$, respectively. Hence $a,b,c,\tilde a$ are given, $\tilde b$ is to be chosen, and  $\tilde c$ is to be
estimated. Let $\alpha$ be the opening angle between $\bgamma_{\bp}^{\bq}$ and $\bgamma_{\bp}^{\tilde\bq}$. We recall that $\hav\theta=\sin^2(\theta/2)$, and that
	\begin{eqnarray*}
		\hav c = \hav(a-b)+\sin a \sin b \hav\alpha, \quad  \hav \tilde c = \hav(\tilde a-\tilde b)+\sin \tilde a \sin \tilde b \hav\alpha.
	\end{eqnarray*}
	Plugging the former into the latter, we see that
	\begin{equation}\label{havtc}
		\hav \tilde c = \hav(\tilde a-\tilde b)+\frac{\sin \tilde a}{\sin a} \frac{\sin \tilde b}{\sin b}(\hav c- \underbrace{\hav(a-b)}_{\ge 0}).
	\end{equation}
	Note that
	$$
	\frac{\sin \tilde a}{\sin a}\le \left\{\begin{array}{ll} 1 & \mbox{if $a\le \tfrac{\pi}{2}$} \\ \sqrt{C_0} & \mbox{if $\tfrac{\pi}{2}\le a<r_0$}\end{array}\right.
	$$
	with $C_0 = C_0(r_0)\ge 1$, and the same holds for $b,\tilde b$. Therefore
	\begin{equation}\label{app1-base}
		\hav \tilde c \leq \hav(\tilde a-\tilde b)+C_0 \hav c.
	\end{equation}
If $\tilde a\le b$, we choose $\overline \bp$ such that $\tilde b=\tilde a$. Then \eqref{app1-base} simplifies to $\hav \tilde c \le C_0 \hav c$. If on the other hand $\tilde a>b$, we choose $\overline
\bp=\tilde\bq$, i.\,e.\ $\tilde b=b$. Then \eqref{havtc} reads as
	\begin{equation*}
		\hav \tilde c = \hav(\tilde a-b) +\frac{\sin \tilde a}{\sin a} (\hav c- \hav(a-b)) \le \sqrt{C_0} \hav c +\sin \tilde a\left( \frac{\hav(\tilde a-b)}{\sin\tilde a} -\frac{\hav (a-b)}{\sin
a}\right).
	\end{equation*}
	Since the function $a\mapsto \frac{\hav (a-b)}{\sin a}$ is increasing for $b\le a<\pi$ 
and $C_0\ge 1$, we obtain again $\hav \tilde c\leq C_0\hav c$. Now,
	$$
	\hav \tilde c \le C_0 \hav c \ \iff \ \tilde c \le 2\arcsin\left(\sqrt{C_0}\sin(c/2)\right).
	$$
	From Taylor's expansion, we see that $\tilde c\le 2\sqrt{C_0} c$ for $c\le \epsilon$. For $c\ge \epsilon$ we have $\hav c\ge \epsilon^2/8$: hence we may just write $\tilde c\le \pi = \frac{\pi
c}{c}\le \frac{8\pi }{\epsilon^2}c$, and the claim follows choosing $C= \max\{2\sqrt{C_0},8\pi/\epsilon^2\}$.
\end{proof}

\begin{proof}[Proof of Lemma  \ref{lem-what}]
	Consider first the case that $\mathcal N=\mathbb{S}^2$. We take $\bp:=\bp_1$, $\bq:=\bq_1$ and $\tilde \bq:=\bq_2$ in the previous Lemma \ref{lem-what-S2}. By the assumption of Lemma \ref{lem-what},
we have $r_0:=2r < 2 \rm{rad}_{\S^2} = \pi$. Therefore, for any $\hat\bp\in \bgamma_{\bp_1}^{\bq_1}$ there exists $\overline \bp\in \bgamma_{\bp_1}^{\bq_2}$ such that
$$
{\rm dist}_{\S^2}(\hat \bp,\overline \bp)\leq C{\rm dist}_{\S^2}(\bq_1,\bq_2).
$$
Applying again Lemma \ref{lem-what-S2}, with $\bp:=\bq_2$, $\bq:=\bp_1$ and $\tilde\bq:=\bp_2$ it follows that there exists $\overline{\overline \bp}\in \bgamma_{\bq_2}^{\bp_2}$ such that
	$$
	\dist_{\S^2}(\overline \bp,\overline{\overline \bp})\leq C \dist_{\S^2}(\bp_1,\bp_2),
	$$ which implies that
	\begin{equation}\label{what-sn}
	\dist_{\S^2}(\hat\bp,\bgamma_{\bp_2}^{\bq_2})\le\dist_{\S^2}(\hat\bp,\overline{\overline \bp})\leq C(\dist_{\S^2}({\bp_1},{\bp_2})+\dist_{\S^2}({\bq_1},{\bq_2})) \quad\forall \hat \bp \in
\bgamma_{\bp_1}^{\bq_1}.
	\end{equation}
Taking the supremum with respect to $\hat\bp$ and recalling the definition \eqref{def-dist} of $\vec{\dd}$ yields \eqref{what} in the case $\mathcal N=\mathbb{S}^2$.

By rescaling we pass from the unit sphere $\S^2$ to a sphere with curvature $K$. Let now $\N$ be a generic manifold without boundary, such that its sectional curvature is bounded above by $K_{\mathcal
N} > 0$. We consider the curve $\bgamma=[\bp_1,\bp_2,\bq_2,\bq_1,\bp_1]$. By the assumption of Lemma \ref{lem-what} we have $\bp_1, \bp_2, \bq_1, \bq_2 \in B_{\N}(\bp_0,r)$ with $r<\radN$; in
particular,
\[\dist_{\N}(\bp_1, \bq_1) < 2r, \quad \dist_{\N}(\bp_2, \bq_2) < 2r.\]
	Thus, if $L(\bgamma)\ge 4\radN$, we necessarily have
$$
2m:=\dist_{\N}(\bp_1, \bp_2) + \dist_{\N}(\bq_1, \bq_2)>4(\radN -r).
$$
By triangle inequality,
\begin{align}
\vec{\dd}(\bgamma_{\bp_1}^{\bq_1}, \bgamma_{\bp_2}^{\bq_2}) & \stackrel{(\ref{def-dist})}= \sup_{\bp\in \bgamma_{\bp_1}^{\bq_1}} \dist_\N(\bp, \bgamma_{\bp_2}^{\bq_2}) \le  \sup_{\bp\in
\bgamma_{\bp_1}^{\bq_1}}\left(\dist_\N(\bp,\bp_1)+\dist_\N(\bp_1,  \bgamma_{\bp_2}^{\bq_2} )\right)
\\ & \le \dist_\N(\bq_1,\bp_1) + \dist_\N(\bp_1,\bp_2)
\end{align}
and analogously
\begin{align}
\vec{\dd}(\bgamma_{\bp_1}^{\bq_1}, \bgamma_{\bp_2}^{\bq_2}) \le \dist_\N(\bp_1,\bq_1) + \dist_\N(\bq_1,\bq_2),
\end{align}
hence
\begin{align}
\vec{\dd}(\bgamma_{\bp_1}^{\bq_1}, \bgamma_{\bp_2}^{\bq_2}) & < 2r +m  = \left(1 + \frac{2r}{m}\right)m < \left(1 + \frac{r}{2(\radN - r)}\right)m
\\ & \le \left(2 + \frac{r}{\radN - r}\right)\max\left\{\dist_{\N}(\bp_1, \bp_2), \dist_{\N}(\bq_1, \bq_2)\right\}.
\end{align}
Otherwise, if $L(\bgamma)< 4 \radN$, we can apply Reshetnyak's majorization theorem \cite{Reshetnyak1968K} (see also \cite{alexander2019alexandrov}).
	We first observe that since $\bp_1, \bp_2, \bq_1, \bq_2 \in B_{\N}(\bp_0,r)$, by convexity of $B_{\N}(\bp_0,r)$, the whole curve $\bgamma$ is contained in $B_{\N}(\bp_0,r)$. Geodesic balls of
radius less than $\min\{{\rm inj}_N,\frac{\pi}{2\sqrt{K_\N}}\}$ are ${\rm CAT}(K_\N)$ spaces (see \cite[Propostion 6.4.6]{buser1981gromov} or \cite[Theorem 4.3]{ALEXANDER199667}). Thus, there is a
convex region $D$ in the sphere of curvature $K_{\N}$ and a length-non-increasing map $F\colon D\to \mathcal N$ such that $F(\partial D)=\bgamma$ which is length preserving on $\partial D$
\cite[9.56]{alexander2019alexandrov}. Therefore, \cite[9.54]{alexander2019alexandrov} any geodesic subarc of $\bgamma$ in $\mathcal N$ is the image of a geodesic subarc of $\partial D$ in $\S^2$. This
implies that
\begin{align*}
		{\rm dist}_{\N}(\hat\bp,\bgamma_{\bp_2}^{\bq_2}) & \leq {\rm dist}_{\S^2}(F^{-1}(\hat\bp),\bgamma_{F^{-1}(\bp_2)}^{F^{-1}(\bq_2)})
\\
&  \stackrel{(\ref{what-sn})}\leq C({\rm dist}_{\S^2}({F^{-1}(\bp_1)},{F^{-1}(\bp_2)})+{\rm dist}_{\S^2}({F^{-1}(\bq_1)},{F^{-1}(\bq_2)}))
\\
& =C({\rm dist}_{\N}({\bp_1},{\bp_2})+{\rm dist}_{\N}({\bq_1},{\bq_2})) \quad\forall \hat\bp\in\bgamma_{\bp_1}^{\bq_1}.
	\end{align*}
\end{proof}

\bibliographystyle{twojstyl}
\bibliography{./aniso}
\end{document}